\documentclass[11pt,reqno]{amsart}

\usepackage{microtype}
\usepackage{amsmath,amssymb,mathtools}
\usepackage{booktabs,tabularx,array}
\usepackage{xcolor}
\usepackage{enumitem}
\usepackage[numbers,sort&compress]{natbib}
\usepackage[colorlinks=true,linkcolor=blue!55!black,citecolor=blue!55!black,urlcolor=blue!55!black]{hyperref}
\hypersetup{
  pdftitle={Logarithmic Chowla Correlations Across All Shift Scales},
  pdfauthor={Jizhou Guo},
  pdfsubject={Quantitative logarithmic Chowla correlations},
  pdfkeywords={Chowla conjecture, Liouville function, logarithmic correlations, shift averages, harmonic multipliers, short arithmetic progressions, exceptional characters}
}

\newcommand{\N}{\mathbb{N}}
\newcommand{\Z}{\mathbb{Z}}
\newcommand{\R}{\mathbb{R}}

\newcommand{\T}{\mathbb{R}/\mathbb{Z}}
\newcommand{\cD}{\mathcal{D}}
\newcommand{\cE}{\mathcal{E}}
\newcommand{\cP}{\mathcal{P}}
\newcommand{\cS}{\mathcal{S}}
\newcommand{\cT}{\mathcal{T}}
\newcommand{\1}{\mathbf{1}}
\newcommand{\eps}{\varepsilon}
\newcommand{\dd}{\,\mathrm{d}}
\newcommand{\Li}{\lambda}
\newcommand{\WP}{W_{\mathrm{Pil}}}
\newcommand{\WM}{\mathcal{W}}

\newtheorem{theorem}{Theorem}[section]
\newtheorem{proposition}[theorem]{Proposition}
\newtheorem{lemma}[theorem]{Lemma}
\newtheorem{corollary}[theorem]{Corollary}

\newtheorem{publishedinput}[theorem]{Input}
\theoremstyle{definition}
\newtheorem{definition}[theorem]{Definition}
\theoremstyle{remark}
\newtheorem{remark}[theorem]{Remark}

\title[Logarithmic Chowla correlations across all scales]{Logarithmic Chowla Correlations\\Across All Shift Scales}
\author{Jizhou Guo}
\address{Dots Studio, RedNote}
\email{mitsuha2021b@gmail.com}
\email{sjtu18640985163@sjtu.edu.cn}
\date{September 1, 2026}
\subjclass[2020]{11N37, 11N35}
\keywords{Chowla conjecture, Liouville function, logarithmic correlations, multiplicative functions, polynomial shifts, harmonic multipliers, short intervals}

\begin{document}

\begin{abstract}
Let \(\lambda(n)=(-1)^{\Omega(n)}\) be the Liouville function.  We prove a
fixed power-logarithmic bound for its logarithmically weighted two-point
correlations across the full shift range.  There is an absolute \(c>0\)
such that every sufficiently large \(x\) admits a single set
\(\mathcal E_x\subseteq[1,x]\) with
\[
 |\mathcal E_x\cap[1,H]|
 \ll_A H(\log x)^{-A}
 \qquad(1\le H\le x)
\]
for every fixed \(A>0\), while
\[
 \max_{\substack{1\le h\le x\\h\notin\mathcal E_x}}
 \sup_{1\le y\le x}
 \left|\sum_{n\le y}\frac{\lambda(n)\lambda(n+h)}n\right|
 \ll (\log x)^{1-c}.
\]
The same exceptional-set formulation extends, without an upper cutoff, to
all positive integer shifts.  Earlier full-range theorems average over the
shift; here a fixed saving holds pointwise outside one set whose density in
every initial segment is smaller than every fixed negative power of
\(\log x\).

The new middle-scale argument combines a general-good-modulus Liouville
deletion lemma with a linear bad-modulus score, a progression Fourier
estimate, and a Mellin-localized dilation that separates divisor-dependent
endpoints.  Maximal fixed-moment bounds evacuate the low prefix and control
the long-shift range.  Assuming GRH for primitive Dirichlet \(L\)-functions,
we also prove, uniformly for \(h\in\mathbb N\) and \(1\le y\le x\),
\[
 \left|\sum_{n\le y}\frac{\lambda(n)\lambda(n+h)}n\right|
 \le \log(2\min\{h,y\})+O((\log x)^{1-c_{\mathrm G}})
\]
for an absolute \(c_{\mathrm G}>0\), with no exceptional shifts.

\end{abstract}

\maketitle

\section{Introduction and main results}\label{sec:introduction}

The two-point Chowla conjecture predicts that
\[
 \sum_{n\le x}\Li(n)\Li(n+h)=o(x)
\]
for every fixed nonzero integer \(h\).  It remains open even at \(h=1\).
Tao proved its logarithmically averaged form for fixed affine forms
\citep{tao2015logarithmic}.  At shift one, Helfgott and Radziwi\l\l{}
obtained a quantitative logarithmic saving by an expansion argument
\citep{helfgott2021expansion}, and Pilatte proved the power-logarithmic
estimate
\begin{equation}\label{eq:pilatte-intro}
 \sum_{n\le x}\frac{\Li(n)\Li(n+1)}n
 \ll(\log x)^{1-c_0}
\end{equation}
for an absolute \(c_0>0\) \citep{pilatte2025improved}.

Uniformity in a growing shift is a separate issue.  Tao and
Ter\"av\"ainen proved quantitative two-point estimates for general
multiplicative functions, uniformly over small progression moduli, residue
classes, and shifts, outside a common exceptional set of dyadic scales
\citep{taoteravainen2026quantitative}.  Their Liouville specialization
allows a sufficiently small absolute polylogarithmic range.  On the other
hand, the averaged-shift theorem of Matom\"aki, Radziwi\l\l{}, and Tao reaches
the full range of shift lengths \citep{matomaki2015averaged}; its later
quantitative refinement in \citep{menon2026short} remains an average over
the shift.  The question addressed here is whether a fixed power-logarithmic
saving can hold pointwise throughout the full shift range after deleting one
exceptional set that is substantially smaller than the saving itself.

Our main result gives such a set, simultaneously at every shift cutoff and
every terminal point.

\begin{theorem}[Global shift range]\label{thm:global-u}
There is an absolute constant \(c>0\) such that every sufficiently large
\(x\) admits a set
\[
 \cE_x\subseteq[1,x]\cap\N
\]
chosen independently of the exponent below, with the following properties.
For every fixed \(A>0\), uniformly for \(1\le H\le x\),
\begin{equation}\label{eq:global-prefix-intro}
 |\cE_x\cap[1,H]|\ll_A H(\log x)^{-A},
\end{equation}
and
\begin{equation}\label{eq:global-corr-intro}
 \max_{\substack{1\le h\le x\\h\notin\cE_x}}
 \sup_{1\le y\le x}
 \left|\sum_{n\le y}\frac{\Li(n)\Li(n+h)}n\right|
 \ll(\log x)^{1-c}.
\end{equation}
\end{theorem}

The lower threshold for \(x\) and the constants implicit in
\eqref{eq:global-prefix-intro} are ineffective because the proof invokes
Siegel's theorem.  The saving exponent \(c\) is absolute.

Thus the correlation saving and the exceptional-set exponent are completely
decoupled: \(c\) is fixed, while \(A\) is arbitrary.  The quantifier order
in \eqref{eq:global-prefix-intro} is also essential.  One set \(\cE_x\)
works for every \(A\), every initial segment, and every terminal point; the
set is not reselected when one of these parameters changes.  We will avoid
using ``all shifts'' without this qualifier.

The natural range in Theorem~\ref{thm:global-u} contains the new analytic
work.  A translated-moment argument removes the formal upper cutoff.

\begin{corollary}[All positive shift scales]\label{cor:all-positive-shifts}
There are an absolute \(c_+>0\) and, for every sufficiently large \(x\), a
single set \(\cE_x^+\subseteq\N\) such that for every fixed \(A>0\) and
every finite \(H\ge1\),
\begin{equation}\label{eq:all-positive-prefix-intro}
 |\cE_x^+\cap[1,H]|\ll_A H(\log x)^{-A},
\end{equation}
whereas
\begin{equation}\label{eq:all-positive-corr-intro}
 \sup_{h\in\N\setminus\cE_x^+}\sup_{1\le y\le x}
 \left|\sum_{n\le y}\frac{\Li(n)\Li(n+h)}n\right|
 \ll(\log x)^{1-c_+}.
\end{equation}
\end{corollary}

There is also a natural conditional statement with no exceptional shifts.

\begin{theorem}[GRH scale profile]\label{thm:grh-all-shifts}
Assume GRH for all primitive Dirichlet \(L\)-functions.  There is an
absolute \(c_{\rm G}>0\) such that, uniformly for \(h\in\N\) and
\(1\le y\le x\),
\begin{equation}\label{eq:grh-profile-intro}
 \left|\sum_{n\le y}\frac{\Li(n)\Li(n+h)}n\right|
 \le \log(2\min\{h,y\})+O((\log x)^{1-c_{\rm G}}).
\end{equation}
Consequently, for every prescribed function \(H(x)\ge1\) satisfying
\(\log H(x)=o(\log x)\),
\begin{equation}\label{eq:grh-subpower-intro}
 \sup_{1\le h\le\lfloor H(x)\rfloor}
 \left|\sum_{n\le x}\frac{\Li(n)\Li(n+h)}n\right|
 =o(\log x).
\end{equation}
\end{theorem}

The remaining results below supply the three regimes used in the proof of
Theorem~\ref{thm:global-u}.  We retain the stronger structure available in
the low-shift range and the sharper distributional information available in
long shift windows.

\begin{theorem}[Almost-all super-polylogarithmic shifts]
\label{thm:super-main}
For every fixed
\begin{equation}\label{eq:super-eta-range-intro}
 0<\eta<1
\end{equation}
there is a constant \(c_\eta>0\) such that all sufficiently large \(x\)
admit a single set
\begin{equation}\label{eq:super-exception-intro}
 \cE_{x,\eta}\subseteq
 [1,\exp((\log x)^\eta)]\cap\mathbb Z,
\end{equation}
which, for every fixed \(C>0\), satisfies simultaneously for all
\(1\le H\le\exp((\log x)^\eta)\)
\begin{equation}\label{eq:super-prefix-exception-intro}
 |\cE_{x,\eta}\cap[1,H]|
 \ll_{\eta,C}H(\log x)^{-C},
\end{equation}
and for which
\begin{equation}\label{eq:super-correlation-intro}
 \max_{\substack{1\le h\le\exp((\log x)^\eta)\\
                    h\notin\cE_{x,\eta}}}
 \sup_{1\le y\le x}
 \left|\sum_{n\le y}\frac{\Li(n)\Li(n+h)}n\right|
 \ll_\eta(\log x)^{1-c_\eta}.
\end{equation}
The set \(\cE_{x,\eta}\) contains at most one prime.
\end{theorem}

The lower threshold and the constants implicit in
\eqref{eq:super-prefix-exception-intro} are ineffective for the same reason.

Theorem~\ref{thm:super-main} reaches every fixed stretched-logarithmic
subpower range \(\exp((\log x)^\eta)=x^{o(1)}\), \(\eta<1\).  Its exceptional set has
prefix density smaller than every fixed negative power of \(\log x\),
uniformly at every shorter cutoff; it is not merely a density-one statement
with an unspecified rate.  On long
shift windows, the harmonic weight admits a different argument which gives
all fixed moments.

\begin{theorem}[Long-shift moment bounds]\label{thm:long-shift-moments}
Fix \(1<p<\infty\).  Uniformly for \(x\ge2\), integers \(H\ge1\), and
integers \(L\ge0\),
\begin{equation}\label{eq:long-shift-moment-intro}
 \frac1H\sum_{L<h\le L+H}
 \sup_{1\le y\le x}
 \left|\sum_{n\le y}\frac{\Li(n)\Li(n+h)}n\right|^p
 \ll_p 1+\frac{x}{H}.
\end{equation}
More generally, the same estimate holds with \(\Li(n+h)\) replaced by
an arbitrary \(1\)-bounded sequence \(b(n+h)\), uniformly in \(b\).
Thus one exceptional set works simultaneously for every truncation
\(y\le x\).  In particular, if \(H\ge x\), every fixed maximal moment is
bounded uniformly in \(x,H,L\).
\end{theorem}

A direct Markov consequence is especially useful.  Fix \(B\ge0\) and
\(\varepsilon,C>0\).  If \(H\ge x/(\log x)^B\), then in every interval of
\(H\) consecutive positive shifts, all but
\[
 O_{B,\varepsilon,C}\bigl(H(\log x)^{-C}\bigr)
\]
shifts satisfy
\begin{equation}\label{eq:long-shift-tail-intro}
 \left|\sum_{n\le x}\frac{\Li(n)\Li(n+h)}n\right|
 \le (\log x)^\varepsilon.
\end{equation}
Thus the exact endpoint \(H=x\) has an arbitrarily power-logarithmically
sparse exceptional set, with no Landau--Page alternative.

The factor \(x/H\) can be removed on substantially shorter polynomial
windows.  The threshold below comes from the published almost-all maximal
Fourier-uniformity theorem for M\"obius on intervals longer than
\(X^{1/3+\varepsilon}\), transferred here to Liouville and then used one
physical block at a time.

\begin{theorem}[All moments beyond one third]
\label{thm:one-third-moments}
Fix \(\theta>1/3\) and \(1<p<\infty\).  Uniformly for all sufficiently
large \(x\), all integers \(H\ge x^\theta\), all \(L\ge0\), and every
\(1\)-bounded sequence \(b:\mathbb Z\to\mathbb C\),
\begin{equation}\label{eq:one-third-moment-intro}
 \frac1H\sum_{L<h\le L+H}
 \sup_{1\le y\le x}
 \left|\sum_{n\le y}\frac{\Li(n)b(n+h)}n\right|^p
 \ll_{p,\theta}1.
\end{equation}
More precisely, for every fixed \(A>0\), the contribution above the shift
scale satisfies
\begin{equation}\label{eq:one-third-high-tail-intro}
 \left(\frac1H\sum_{L<h\le L+H}
 \sup_{H\le y\le x}
 \left|\sum_{H<n\le y}\frac{\Li(n)b(n+h)}n\right|^p\right)^{1/p}
 \ll_{p,\theta,A}(\log x)^{-A},
\end{equation}
where the sum is empty if \(H\ge x\).  Thus, in normalized \(L^p\), the
full correlation is its prefix at \(n\le H\) plus an arbitrarily
power-logarithmically small tail.  In particular, the same bounds hold for
\(b=\Li\).
\end{theorem}

The moment theorem gives a rapid distributional tail: for every fixed
\(P>0\) and every \(T\ge1\),
\begin{equation}\label{eq:one-third-distribution-intro}
 \#\left\{L<h\le L+H:
 \sup_{1\le y\le x}
 \left|\sum_{n\le y}\frac{\Li(n)b(n+h)}n\right|>T\right\}
 \ll_{P,\theta}HT^{-P}.
\end{equation}
Consequently, for every fixed \(\theta>1/3\) and \(\varepsilon,C>0\),
outside \(O_{\theta,\varepsilon,C}(H(\log x)^{-C})\) shifts in any window
of length \(H\ge x^\theta\), one has
\begin{equation}\label{eq:one-third-tail-intro}
 \left|\sum_{n\le x}\frac{\Li(n)\Li(n+h)}n\right|
 \le(\log x)^\varepsilon.
\end{equation}
The same exceptional set gives this inequality with the left side replaced
by its supremum over all terminal points \(1\le y\le x\).
Thus the exceptional power \(C\) is arbitrary and is no longer coupled to
the correlation-saving exponent.  In particular, taking
\(\varepsilon=1/2\) gives a fixed power-log consequence with the
absolute choice \(c_0=1/2\) for every fixed \(\theta>1/3\) and arbitrary
fixed \(C>0\).  The threshold \(\theta>1/3\) reflects the range of the
published quantitative input used here.

The two all-moment theorems are maximal in the summation endpoint, rather
than estimates at one preselected terminal point.  They will be used below
for prefix evacuation, the high shift range, and the extension beyond
\(h=x\).

\begin{corollary}[The exceptional-character-free case]
\label{cor:super-all-shifts}
Fix \(0<\eta<1\).  If the Landau--Page exceptional primitive real
character in the conductor range used in the proof does not exist, then
\eqref{eq:super-correlation-intro} holds for every
\(1\le h\le\exp((\log x)^\eta)\).  In particular, this conclusion holds
under the generalized Riemann hypothesis for Dirichlet \(L\)-functions.
\end{corollary}

At fixed polylogarithmic scale the exceptional set can be removed, with a
saving exponent independent of the fixed power.

\begin{theorem}[All fixed polylogarithmic ranges]\label{thm:main}
There exists an absolute constant \(c>0\) such that, for every fixed
\(A>0\), there are \(C_A>0\) and \(x_A\ge3\) for which
\begin{equation}\label{eq:main-theorem}
 \sup_{1\le h\le(\log x)^A}
 \left|\sum_{n\le x}\frac{\Li(n)\Li(n+h)}n\right|
 \le C_A(\log x)^{1-c}
\end{equation}
for every \(x\ge x_A\).
\end{theorem}

The threshold \(x_A\) and the constant \(C_A\) may be ineffective through
Input~\ref{input:menon-character}; no effective dependence on \(A\) is
asserted.

The distinction between the pointwise theorems is quantitative and logical.
Theorem~\ref{thm:main} has no exceptional shifts and its saving exponent is
independent of the fixed power \(A\), but \(A\) must be fixed before
\(x\to\infty\).  Theorem~\ref{thm:super-main} supplies a one-conductor
exceptional set in each fixed stretched-logarithmic range, while its saving
may depend on \(\eta\).  Theorem~\ref{thm:one-third-moments} is instead a
shift-window moment theorem and does not produce a pointwise arithmetic
description of its exceptional shifts.  Theorem~\ref{thm:global-u} uses
both mechanisms but is stronger than either separately: it retains a fixed
pointwise saving through all shift scales while the density exponent of one
global exceptional set is arbitrary.

\begin{table}[t]
\centering
\caption{Quantitative two-point results and the shift quantifier relevant
to this paper.  The table suppresses hypotheses not used in the
comparison.}\label{tab:comparison}
\scriptsize
\begin{tabularx}{\textwidth}{@{}>{\raggedright\arraybackslash}p{0.18\textwidth}*{3}{>{\raggedright\arraybackslash}X}@{}}
\toprule
Result & averaging / scales & shift range & conclusion \\
\midrule
Tao (2016) & logarithmic, all large \(x\) & each fixed shift & qualitative \(o(\log x)\) \\
Helfgott--Radziwi\l\l{} (2021) & logarithmic & shift one & \((\log\log x)^{-1/2}\)-type saving \\
Pilatte (2026) & logarithmic, all large \(x\) & shift one & fixed power of \(\log x\) \\
Tao--Ter\"av\"ainen (2025, v2 2026) & natural averages outside a common scale-exception set & one small absolute polylogarithmic range & fixed power of \(\log N\) \\
Matom\"aki--Radziwi\l\l{}--Tao; Menon & natural, averaged over shifts & every \(1\le H\le x\) & quantitative mean saving; no fixed pointwise saving with arbitrary exceptional exponent \\
Theorem~\ref{thm:main} & logarithmic, all large \(x\) & every fixed polylogarithmic range & one absolute power saving, no shift exceptions \\
Theorem~\ref{thm:super-main} & logarithmic, all large \(x\) & \(h\le\exp((\log x)^\eta)\), \(\eta<1\) & fixed power saving outside an explicitly sparse shift set \\
Theorem~\ref{thm:global-u} & logarithmic, all large \(x\) & full natural range \(1\le h\le x\) & fixed power saving outside one all-prefix arbitrary-log sparse set; maximal in \(y\) \\
Theorem~\ref{thm:long-shift-moments} & logarithmic, shift-window moments & $L<h\le L+H$ & maximal fixed moments $\ll_p1+x/H$; one exceptional set for all cutoffs \\
Theorem~\ref{thm:one-third-moments} & logarithmic, shift-window moments & $L<h\le L+H$, $H\ge x^\theta$, $\theta>1/3$ & maximal fixed moments $\ll_{p,\theta}1$; the maximal tail above $H$ is arbitrarily log-power small \\
\bottomrule
\end{tabularx}
\end{table}

\paragraph{Progression-of-starts Fourier uniformity.}
For a common frequency \(\alpha\), the new local object is
\begin{equation}\label{eq:apf-intro}
 \sum_{Z<y\le2Z}
 \left|\sum_{t\le M}\Li(y+ht)e(\alpha t)\right|.
\end{equation}
If one first embeds each progression in an ordinary interval of length
\(hM\), character or additive orthogonality loses the density \(1/h\).
Instead we dualize the outer norm and apply Ramar\'e's identity before any
character-wise triangle inequality.  After the fourth power is expanded,
three congruences in the start variables cancel the apparent \(h^3\)
geometric-volume loss.  This proves the minor arcs uniformly in the full
shift range.

\paragraph{Major arcs and one conductor.}
At a rational frequency \(a/q\), splitting \(t\) modulo \(q\) produces a
moving residue class modulo \(hq\).  Averaging the physical left endpoint
over a complete modulus decouples this residue from the start and converts
\eqref{eq:apf-intro} into the full short-progression variance of
Klurman--Mangerel--Ter\"av\"ainen
\citep{klurman2023shortaps}.  Their variance theorem leaves a possible
real-character main term.  We prove that, across every parent interval,
divisor reduction, Abel block, and subset-rescaled base used here, all
surviving terms are organized by one Landau--Page conductor.  Siegel's theorem
makes that conductor larger than every fixed power of the local parameter,
which gives \eqref{eq:super-exception-intro} and the assertion about prime
shifts.

\paragraph{Transfer to Chowla correlations.}
Writing \(n=r+hq\), we translate the quotient \(q\), rather than the
physical integer \(n\), before Fourier inversion.  The second Fourier
factor then has length \(2M\), not \(2hM\).  Direct interpolation of the
product factor between \(L^1\) and \(L^\infty\), followed by the \(L^4\)
bound for the divisor polynomial, gives a fourth-root transfer with no
exceptional-frequency split.  Before the centred step, block primes dividing
\(h\) are pruned; their reciprocal mass is smaller than every fixed power of
\(R^{-1}\), while the fourth moment can only decrease.  The centred term then
follows from the arbitrary-interval, arbitrary-residue decoupling theorem of
Tao and Ter\"av\"ainen.  The exact parameter construction leaves the strict
region \(\eta<1\).

\paragraph{Harmonic multipliers and polynomial shift windows.}
At the endpoint we use the logarithmic weight directly.  Davenport's
uniform exponential-sum estimate for \(\Li\) implies that on each dyadic
block the Fourier multiplier of \(\Li(n)/n\) decays faster than every fixed
power of the block index.  Plancherel and a binary decomposition inside
each dyadic shell give decay for the maximal partial-sum operator, with an
inessential square-root loss in the logarithmic exponent.  The dyadic
harmonic mass gives the \(\ell^1\) and \(\ell^\infty\) endpoints, and
sublinear interpolation yields a summable maximal operator norm on every
fixed \(\ell^p\), \(1<p<\infty\).  This proves
Theorem~\ref{thm:long-shift-moments} without any exceptional character.
For \(H\ge x^\theta\), \(\theta>1/3\), we partition every physical shell
above \(H\) into intervals of length at most \(H\).  The published
almost-all maximal Fourier estimate of
\citet{matomaki2026higherII}, together with a square-divisor transfer from
M\"obius to Liouville, makes the good local convolution multipliers
arbitrarily power-logarithmically small.  A translated block grid makes
the total harmonic mass of the bad physical intervals equally small.
This proves Theorem~\ref{thm:one-third-moments} and removes the global
\(x/H\) loss.

\paragraph{The middle shift range.}
The new issue in Theorem~\ref{thm:global-u} is the interval between the
stretched-logarithmic pointwise theorem and the one-third moment theorem.
On a dyadic shift band \(H<h\le2H\), we first set
\(Y=H^{3/2}\exp(O((\log H)^{2.01/5}))\).  The maximal all-moment theorem
controls every prefix \(y\le Y\) outside one set of size
\(O_A(H(\log x)^{-A})\) for all fixed \(A\).  This \emph{prefix
evacuation} leaves a high tail on which the quotient intervals are long
enough for the centred theorem and translation boundaries are negligible.

The physical moduli in that tail are no longer in KMT's automatic small
core.  We classify KMT character main terms by primitive conductor and
delete all but at most one Landau--Page conductor on each robust physical
band.  Rather than discard the remaining bad gcd strata by a union bound,
we retain their exact linear coefficient
\[
 W_{X,q}(h)=\frac1{hq}
 \sum_{\substack{m\mid hq\\m\ {
m bad}}}\varphi(m).
\]
The generator structure of the bad moduli makes the average of this score
stretched-exponentially small in \(\log H\); Markov therefore gives one
shift set satisfying every fixed power of \((\log x)^{-1}\).  Because the
score is independent of the Abel prefix length, it survives the complete
major-arc transfer without a square-root loss.

\paragraph{Mellin-localized dilation.}
Localizing the final correlation tail creates the divisor-dependent interval
\((dY,dy]\), so the earlier common-parent transfer cannot simply be reused.
After placing the active divisor product \(d'\) in its existing dyadic
block, Mellin inversion separates
\[
 w(k/(d'N))
 =\frac1{2\pi}\int_{\R}\widehat w(t)(k/N)^{it}d'^{-it}\,dt.
\]
The factor \(d'^{-it}\) is a unit divisor coefficient and \(k^{it}\) is a
bounded multiplicative twist.  Consequently the divisor fourth moment,
the scored progression estimate, and the arbitrary-function centred bound
all persist.  A uniform Mellin cutoff and a fresh complex KMT minimizer on
each frequency slice close the constrained-twist endpoint.  This produces
the middle-band high-tail estimate.  Dyadic stitching with the low and high
regimes proves Theorem~\ref{thm:global-u}; translated moments beyond \(x\)
then prove Corollary~\ref{cor:all-positive-shifts}.

The paper is organized as follows.  Section~\ref{sec:inputs} records the
published inputs.  Sections~\ref{sec:flexible-minor}--\ref{sec:main-proof}
prove Theorem~\ref{thm:main}.  Sections~\ref{sec:progression-fourier} and
\ref{sec:exceptional-modulus} establish the progression Fourier estimate
and its one-conductor exceptional set.  Section~\ref{sec:superpolylog-transfer}
proves Theorem~\ref{thm:super-main}.  Section~\ref{sec:maximal-moments}
proves the long-shift moment theorem, the all-moment theorem beyond one
third.
Sections~\ref{sec:general-good-deletion}--\ref{sec:global-u-proof} prove the
new middle-band estimate and Theorem~\ref{thm:global-u}.
Section~\ref{sec:grh-all-shifts} proves Theorem~\ref{thm:grh-all-shifts}.
Section~\ref{sec:global-consequences} records consequences and the remaining
pointwise boundary.  The appendices give dependency and parameter ledgers.

\section{Notation and published inputs}\label{sec:inputs}

Write \(e(t)=\exp(2\pi i t)\), and identify \(\T\) with \([0,1)\) when
integrating.  All implied constants are absolute unless a subscript indicates
their dependence.  The notation \(u\asymp v\) means \(u\ll v\ll u\).  Empty
integer endpoints in short sums are interpreted in the usual way.

\subsection{Pilatte's prime blocks}

Fix once and for all a sufficiently small absolute \(\eps_1>0\).  For a
large parameter \(H\), put
\begin{equation}\label{eq:global-parameters}
 R=\log H,\qquad
 H_0=\exp(R^{1-\eps_1}),\qquad
 J=\lfloor\eps_1^2\log R\rfloor.
\end{equation}
We use the prime blocks \(\cP_1,\ldots,\cP_J\) constructed by Pilatte
\citep[Section~2]{pilatte2025improved}.  Set
\[
 V_j=\sum_{p\in\cP_j}\frac1p,\qquad
 V=\max_{1\le j\le J}V_j,\qquad
 \cP=\bigsqcup_{j=1}^J\cP_j,
\]
and let
\[
 \cD=\{p_1\cdots p_J:p_j\in\cP_j\},
 \qquad
 \WP=\sum_{d\in\cD}\frac1d=\prod_{j=1}^JV_j.
\]
For every choice \(p_i\in\cP_i\), Pilatte's block-separation lemma gives
\begin{equation}\label{eq:block-geometry}
 p_1\cdots p_i<p_{i+1}^{1/10},\qquad
 p_1\cdots p_J<H,\qquad H_0<p_i<H.
\end{equation}
Moreover,
\begin{equation}\label{eq:block-masses}
 V_j=\frac{\eps_1\log R}{2J}+o(1),\qquad
 \WP\asymp_{\eps_1}V^J,\qquad
 V^J\asymp_{\eps_1}R^{c_0(\eps_1)},
\end{equation}
where
\begin{equation}\label{eq:c0-definition}
 c_0(\eps_1)=\eps_1^2\log\frac1{2\eps_1}>0.
\end{equation}
After decreasing \(\eps_1\) once, we also require
\begin{equation}\label{eq:subset-loss}
 2^JV^{2J}\le R^{1/40}
\end{equation}
for all sufficiently large \(R\).  This is possible because the exponent of
\(R\) in the left-hand side tends to zero with \(\eps_1\).

For nonempty \(I\subseteq[J]\), let
\[
 \cD_I=\left\{\prod_{i\in I}p_i:p_i\in\cP_i\right\},
 \qquad
 \cD_I(M)=\cD_I\cap(M/2,M],
\]
where \(M\) ranges over dyadic values.  Put
\[
 V_I[M]=\sum_{d\in\cD_I(M)}\frac1d,\qquad
 Q_{I,M}(\alpha)=\sum_{d\in\cD_I(M)}\frac{e(\alpha d)}d.
\]
We use two further consequences of Pilatte's framework.

\begin{publishedinput}[Fourth moment]\label{input:fourth-moment}
For every nonempty \(I\subseteq[J]\) and every dyadic \(M\),
\begin{equation}\label{eq:fourth-moment}
 \int_0^1|Q_{I,M}(\alpha)|^4\dd\alpha
 \ll\frac{V^{4J}}{M(\log M)^4}.
\end{equation}
\end{publishedinput}

\begin{publishedinput}[Typical-set complement]\label{input:sieve-complement}
Let \(P_1<Q_1\) be the first prime range in the typical-factorization
construction below.  If \(\cT\) denotes the associated arithmetic typical
set, then
\begin{equation}\label{eq:sieve-complement}
 \#\{m\le3Z:m\notin\cT\}
 \ll Z\frac{\log P_1}{\log Q_1}.
\end{equation}
\end{publishedinput}

These are Pilatte's Lemma~C.2 and Menon's Proposition~2.13, respectively
\citep{pilatte2025improved,menon2026short}.

\subsection{Menon's character estimate}

Let \(\kappa_M>0\) be Menon's exponent-pair constant, and fix an absolute
\(C_M>3\) so large that
\begin{equation}\label{eq:CM-choice}
 \frac{C_M\kappa_M}{24}-\frac1{15}>\frac52.
\end{equation}
Menon notes that \(C_M=10^{10}\) is admissible.  The following is the
specialization of his Corollary~3.7 used in Section~\ref{sec:flexible-major}.

\begin{publishedinput}[Character-twisted short intervals]
\label{input:menon-character}
Fix \(D,B'>0\).  Suppose
\begin{equation}\label{eq:menon-character-hyp}
 Y\ge L\ge10,\qquad Y^{1/2}\le X_0\le Y,\qquad
 q\le(\log Y)^{B'},
\end{equation}
and let \(\chi\pmod q\).  If \(\cT\) is the arithmetic typical set below,
and on the interval in question it agrees with the local set from Menon's
Proposition~3.1, choose the auxiliary last-range parameter \(R_{\mathrm M}\)
in Menon's equation~(2.7) sufficiently large in terms of \(D\).  Then
\begin{equation}\label{eq:menon-character}
 \frac1Y\int_Y^{2Y}\left|
  \frac1L\sum_{\substack{y<m\le y+L\\m\in\cT}}
        \chi(m)\Li(m)
 \right|^2\dd y
 \ll_{D,B'}
 \frac{(\log Q_1)^{1/3}}{P_1^{\kappa_M/24}}
 + (\log Y)^{-D}.
\end{equation}
The constant is ineffective but uniform once \(D\), \(B'\), and this
fixed choice of \(R_{\mathrm M}=R_{\mathrm M}(D)\) are fixed.  Here \(Y,L\)
are the local ambient scale and short length (the \(X,h\) of Menon's
Corollary~3.7), while \(X_0\) denotes the auxiliary square-root scale.  This
renaming keeps the local variables distinct from Menon's internal auxiliary
parameters.
\end{publishedinput}

\subsection{Centred decoupling}

The centred part of the proof is supplied by the general quantitative
correlation theorem of Tao and Ter\"av\"ainen.  We record only the specialized
interface needed below.

\begin{publishedinput}[Arbitrary-interval centred decoupling]
\label{input:tt-decoupling}
Let \(I_q\) be an arbitrary interval of integers of length \(L\), with
\(L\ge\exp(R^{2.01})\).  Let \(I'\subseteq[J]\) be nonempty, choose arbitrary
residues \(b_p\pmod p\), and let \(g_1,g_2\) be \(1\)-bounded functions.
Then, for \(\sigma\in\{-1,+1\}\),
\begin{multline}\label{eq:tt-decoupling}
 \sum_{(p_i)_{i\in I'}}
 \left|\sum_{q\in I_q}
  \prod_{i\in I'}\left(\1_{q\equiv b_{p_i}\ (p_i)}-\frac1{p_i}\right)
  g_1(q)g_2\left(q+\sigma\prod_{i\in I'}p_i\right)
 \right|\\
 \ll V^{0.51|I'|}L.
\end{multline}
\end{publishedinput}

This is the form of \citep[Theorem~3.3]{taoteravainen2026quantitative}
obtained by retaining the actual interval length.  The theorem explicitly
permits arbitrary intervals, arbitrary residue assignments, and different
bounded functions, so no translation-invariance extension is being assumed.

\subsection{Almost-all local Fourier uniformity}

For an interval $I$, write
\[
 \left|\sum_{n\in I\cap\mathbb Z}a(n)\right|^*
 :=\sup_{P\subseteq I\cap\mathbb Z}
   \left|\sum_{n\in P}a(n)\right|,
\]
where $P$ ranges over arithmetic progressions.

\begin{publishedinput}[Almost-all maximal M\"obius Fourier uniformity]
\label{input:almost-all-mobius-fourier}
Fix $\delta,A>0$.  Uniformly for
\[
 X^{1/3+\delta}\le K\le X^{1-\delta},
\]
one has
\begin{multline}\label{eq:published-mobius-fourier}
 \operatorname{meas}\left\{y\in[X,2X]:
 \sup_{\alpha\in\mathbb R/\mathbb Z}
 \left|\sum_{y<n\le y+K}\mu(n)e(\alpha n)\right|^*
 >K(\log X)^{-A}\right\}\\
 \ll_{A,\delta}X(\log X)^{-A}.
\end{multline}
\end{publishedinput}

This is the circle-nilmanifold, linear-phase specialization of
\citet[Theorem~1.1(i)]{matomaki2026higherII}.  That paper was published in
\emph{Inventiones Mathematicae} in 2026.  The supremum over polynomial
orbits in the source makes the exceptional set independent of the Fourier
frequency, and its star is the maximal progression norm displayed above.
For the present specialization one takes the polynomial orbit to be the
linear phase \(e(\alpha n)\); the source's supremum over arithmetic
progressions is exactly the displayed star norm, so no extra frequency or
progression union is introduced here.

\section{The flexible Menon estimate: minor arcs}\label{sec:flexible-minor}

The first new ingredient is a localized version of Menon's exponential-sum
estimate in which the circle cutoff may be any fixed power of the outer
logarithm.  We distinguish the number \(J_M\) of Menon prime ranges from the
number \(J\) of Pilatte blocks in Section~\ref{sec:inputs}.

Given \(X_0\) and \(P_1<Q_1\), define recursively
\[
 P_j=\exp\!\left(j^{4j}(\log Q_1)^{j-1}\log P_1\right),\qquad
 Q_j=\exp\!\left(j^{4j+2}(\log Q_1)^j\right),
\]
and take \(J_M\) maximal subject to
\begin{equation}\label{eq:menon-J}
 Q_{J_M}\le\exp\!\left(\frac{\log X_0}{\log\log X_0}\right).
\end{equation}

\begin{definition}[Arithmetic typical set]\label{def:typical-set}
Set
\begin{equation}\label{eq:typical-set}
 \begin{aligned}
 \cT(P_1,Q_1,X_0)=\{m\ge1:\;&\text{for every }1\le j\le J_M,\\
 &\text{some prime }p\in[P_j,Q_j]\text{ divides }m\}.
 \end{aligned}
\end{equation}
\end{definition}

This definition is independent of a finite support convention.  On a local
base interval \([Y,2Y]\), Menon's set is precisely
\(\cT(P_1,Q_1,X_0)\cap[1,3Y]\).  This observation prevents a top dyadic
block from crossing an artificial cutoff.

\begin{proposition}[Localized flexible Menon estimate]
\label{prop:flexible-menon}
Fix \(B_0>0\).  There is an ineffective \(Z_*(B_0)\) such that the
following holds for \(Z\ge Z_*(B_0)\).  Suppose \(Z\ge K\ge10\) and
\begin{equation}\label{eq:flexible-hyp}
 (\log K)^5\le\Omega\le
 \min\{K^{1/(2C_M)},(\log Z)^{B_0}\}.
\end{equation}
Put
\begin{equation}\label{eq:flexible-scales}
 X_0=(4Z)^{1/2},\qquad
 K_0=\exp\!\left(\frac{\log X_0}{\log\log X_0}\right),\qquad
 \widetilde K=\min(K,K_0),
\end{equation}
and
\begin{equation}\label{eq:flexible-prime-range}
 P_1=\Omega^{C_M},\qquad Q_1=\widetilde K/\Omega^3,\qquad
 \cT=\cT(P_1,Q_1,X_0).
\end{equation}
Then, uniformly in \(\alpha\in\T\),
\begin{equation}\label{eq:flexible-conclusion}
 \frac1Z\int_Z^{2Z}\left|
  \frac1K\sum_{\substack{x<n\le x+K\\n\in\cT}}
       \Li(n)e(\alpha n)
 \right|\dd x
 \ll_{B_0}
 \frac{(\log K)^{1/4}\log\log K}{\Omega^{1/4}}.
\end{equation}
The implied constant and threshold may depend on \(B_0\); no uniformity is
claimed if \(B_0\) grows with \(Z\).
\end{proposition}

We first prove the complete minor-arc input.  The point is to verify from the
argument, rather than from its last displayed line, that no earlier step uses
Menon's printed numerical upper bound on \(\Omega\).

\begin{lemma}[Weighted minor arcs]\label{lem:weighted-minor}
Let \(X_A\ge K\), let
\[
 \cS_A=\cT(P_1,Q_1,X_A^{1/2})\cap[1,X_A],
\]
and suppose coprime \(a,q\) satisfy
\begin{equation}\label{eq:minor-approximant}
 \Omega<q\le K/\Omega,
 \qquad
 \left|\alpha-\frac aq\right|\le\frac{\Omega}{Kq}.
\end{equation}
For every measurable \(\vartheta\), \(|\vartheta|\le1\), supported on
\([0,X_A]\),
\begin{equation}\label{eq:weighted-minor}
 \left|\int_{\R}\vartheta(x)
  \sum_{\substack{x<n\le x+K\\n\in\cS_A}}
       \Li(n)e(\alpha n)\dd x\right|
 \ll
 \frac{(\log K)^{1/4}\log\log K}{\Omega^{1/4}}KX_A,
\end{equation}
provided
\begin{equation}\label{eq:minor-dependencies}
 P_1=\Omega^{C_M},\qquad Q_1\le K/\Omega^3,\qquad P_1<Q_1.
\end{equation}
\end{lemma}

\begin{proof}
Let \(\cP_0\) be the primes in \([P_1,Q_1]\), and let
\(\omega_{\cP_0}(m)\) count their distinct prime divisors.  Let
\(\cS_A'\) impose the support condition \([1,X_A]\) and the prime-factor
conditions for the ranges \([P_j,Q_j]\), \(j\ge2\).  Ramar\'e's identity,
including the prime-square correction, is
\begin{equation}\label{eq:ramare}
 \1_{\cS_A}(n)=
 \sum_{\substack{p\in\cP_0,\ m\ge1\\mp=n}}
 \frac{\1_{\cS_A'}(mp)}
      {\1_{p\nmid m}+\omega_{\cP_0}(m)}.
\end{equation}
If \(n\in\cS_A\) has \(r\) distinct prime divisors in \(\cP_0\), every
denominator on the right is \(r\), including when \(p^2\mid n\), and there
are exactly \(r\) summands.  If a later prime-factor condition fails, both
sides vanish.

Cover \([P_1,Q_1]\) by dyadic prime intervals \([P,2P]\).  For one such
interval write
\begin{equation}\label{eq:Jmp}
 \mathcal J(mp)=\int_{\R}\vartheta(x)
             \1_{x<mp\le x+K}\dd x,
\end{equation}
and denote by \(\mathcal L_P\) the corresponding part of the left-hand side
of \eqref{eq:weighted-minor}.  Replacing \(\1_{p\nmid m}\) in
\eqref{eq:ramare} by \(1\) changes this dyadic contribution by at most
\begin{equation}\label{eq:prime-square-error}
 K\sum_{P\le p\le2P}\sum_{\substack{m\le X_A/P\\p\mid m}}1
 \ll \frac{KX_A}{P}
 \ll \frac{KX_A(\log K)^{1/4}}
          {\Omega^{1/4}\log P}.
\end{equation}
The last bound uses \(P\ge P_1=\Omega^{C_M}\) and sufficiently large
parameters.

After this replacement, discard the remaining denominator and apply H\"older.
The dyadic contribution is bounded by
\begin{equation}\label{eq:minor-holder}
 \left(\frac{X_A}{P}\right)^{3/4}
 \left(
  \sum_{m\le X_A/P}
  \left|\sum_{\substack{P\le p\le2P\\p\in\cP_0\\mp\le X_A}}
       \Li(p)e(\alpha mp)\mathcal J(mp)\right|^4
 \right)^{1/4}.
\end{equation}
Expanding the fourth power, the common \(m\)-sum is a geometric progression.
For a fixed prime quadruple it is bounded by
\begin{equation}\label{eq:minor-geometric}
 \min\left\{\frac KP,
  \frac1{\|(p_1+p_2-p_3-p_4)\alpha\|}\right\}.
\end{equation}
The four variables in \eqref{eq:Jmp} occupy a set of measure
\(O(X_AK^3)\): once the first variable is fixed, the other three lie in
intervals of length \(O(K)\).  The standard upper-bound sieve followed by
Cauchy--Schwarz gives, uniformly for \(|r|=O(P)\),
\begin{equation}\label{eq:four-prime-count}
 \#\{p_1+p_2-p_3-p_4=r:P\le p_i\le2P\}
 \ll\frac{P^3}{(\log P)^4}.
\end{equation}
Consequently the fourth power in \eqref{eq:minor-holder} is at most
\begin{equation}\label{eq:expanded-fourth}
 \frac{X_AK^3P^3}{(\log P)^4}
 \sum_{|r|=O(P)}
 \min\left\{\frac KP,\frac1{\|r\alpha\|}\right\}.
\end{equation}

The approximation \eqref{eq:minor-approximant} also gives
\(|\alpha-a/q|\le q^{-2}\).  Vinogradov's lemma therefore bounds the last
sum by
\begin{align}
 \left(\frac Pq+1\right)
 \left(\frac KP+q\log q\right)
 &\ll \frac{K\log K}{\Omega}.               \label{eq:vinogradov-endgame}
\end{align}
Indeed, the four terms are
\[
 K/q,\qquad P\log q,\qquad K/P,\qquad q\log q,
\]
and \eqref{eq:minor-approximant}--\eqref{eq:minor-dependencies} give
\[
 K/q\le K/\Omega,\qquad
 P\log q\le K\log K/\Omega^3,
\]
\[
 K/P\le K/\Omega,\qquad
 q\log q\le K\log K/\Omega.
\]
Inserting \eqref{eq:vinogradov-endgame} into
\eqref{eq:expanded-fourth} and then \eqref{eq:minor-holder} yields
\begin{equation}\label{eq:one-dyadic-minor}
 |\mathcal L_P|
 \ll\frac{KX_A(\log K)^{1/4}}
          {\Omega^{1/4}\log P}.
\end{equation}
Finally,
\begin{equation}\label{eq:dyadic-minor-sum}
 \sum_{\substack{P\text{ dyadic}\\P_1\ll P\ll Q_1}}\frac1{\log P}
 \ll\log\log Q_1-\log\log P_1+O(1)
 \ll\log\log K.
\end{equation}
Equations \eqref{eq:prime-square-error}--\eqref{eq:dyadic-minor-sum}
prove \eqref{eq:weighted-minor}.  They also list every occurrence of
\(\Omega\) in the minor-arc proof.  Apart from \(P_1<Q_1\), only
\eqref{eq:minor-approximant} and \eqref{eq:minor-dependencies} are used;
there is no upper restriction of the form
\(\Omega\le(\log X_A)^{c}\).
\end{proof}

We record the elementary admissibility check for the proposition.  For
sufficiently large \(Z\), \(10<P_1<Q_1\).  If \(K\le K_0\), then
\begin{equation}\label{eq:prime-range-case1}
 P_1\Omega^3=\Omega^{C_M+3}
 \le K^{1/2+3/(2C_M)}<K=\widetilde K.
\end{equation}
If \(K>K_0\), then
\begin{equation}\label{eq:prime-range-case2}
 (C_M+3)\log\Omega
 \le(C_M+3)B_0\log\log Z=o(\log K_0),
\end{equation}
so again \(P_1\Omega^3<K_0=\widetilde K\).  The remaining spacing
conditions are unchanged from Menon's construction.

Dirichlet approximation supplies coprime \(a,q\) such that
\begin{equation}\label{eq:dirichlet-flexible}
 1\le q\le K/\Omega,\qquad
 \left|\alpha-\frac aq\right|\le\frac{\Omega}{Kq}.
\end{equation}
When \(q>\Omega\), apply Lemma~\ref{lem:weighted-minor} with
\(X_A=4Z\).  For \(x\in[Z,2Z]\),
\[
 x<n\le x+K\quad\Longrightarrow\quad 1\le n<3Z<4Z,
\]
so arithmetic membership in \(\cT\) agrees with membership in the ambient
set of the lemma.  If
\[
 F(x)=\sum_{\substack{x<n\le x+K\\n\in\cT}}\Li(n)e(\alpha n),
\]
choose the measurable phase
\[
 \vartheta(x)=\1_{[Z,2Z]}(x)
 \begin{cases}\overline{F(x)}/|F(x)|,&F(x)\ne0,\\0,&F(x)=0.\end{cases}
\]
It is \(1\)-bounded and supported in \([0,4Z]\).  Hence
\begin{equation}\label{eq:flexible-minor-result}
 \int_Z^{2Z}|F(x)|\dd x
 \ll\frac{(\log K)^{1/4}\log\log K}{\Omega^{1/4}}KZ.
\end{equation}
This proves Proposition~\ref{prop:flexible-menon} on the minor arcs.

\section{The flexible Menon estimate: major arcs}\label{sec:flexible-major}

We complete the proof of Proposition~\ref{prop:flexible-menon}.  Assume that
the approximating denominator in \eqref{eq:dirichlet-flexible} satisfies
\(q\le\Omega\), and write
\[
 \alpha=a/q+\theta,\qquad |\theta|\le\Omega/(Kq).
\]
For \(0\le K'\le K\), put
\begin{equation}\label{eq:I-Kprime}
 I(K')=\int_Z^{2Z}\left|
  \sum_{\substack{x<n\le x+K'\\n\in\cT}}
       \Li(n)e(an/q)
 \right|\dd x.
\end{equation}
Partial summation in the short interval, followed by integration in \(x\),
gives
\begin{equation}\label{eq:major-partial-summation}
 \int_Z^{2Z}\left|
  \sum_{\substack{x<n\le x+K\\n\in\cT}}
       \Li(n)e(\alpha n)
 \right|\dd x
 \ll I(K)+\frac{\Omega}{Kq}\int_0^KI(K')\dd K'.
\end{equation}

\subsection{Short endpoints}

For \(0\le K'\le qQ_1\), the trivial estimate is
\begin{equation}\label{eq:major-short-trivial}
 I(K')\ll(K'+1)Z.
\end{equation}
Consequently
\begin{align}
 \frac{\Omega}{Kq}\int_0^{qQ_1}I(K')\dd K'
 &\ll \frac{\Omega qQ_1^2Z}{K}+\frac{\Omega Q_1Z}{K} \notag\\
 &\le \frac{KZ}{\Omega^4}+\frac Z{\Omega^2}
 \ll \frac{KZ}{\Omega^4}.                 \label{eq:major-short-endpoints}
\end{align}
The last inequality uses
\(K\ge\Omega^{2C_M}\ge\Omega^2\), which is equivalent to the upper bound
\(\Omega\le K^{1/(2C_M)}\) in \eqref{eq:flexible-hyp}.

\subsection{Residue classes and localization}

Suppose now that \(K'>qQ_1\).  Split the inner sum in
\eqref{eq:I-Kprime} into residue classes \(b\pmod q\), and write
\begin{equation}\label{eq:residue-gcd}
 d=(b,q),\qquad b=db_0,\qquad q=dq_0.
\end{equation}
Since \(d\le q\le\Omega<P_1\), no prime divisor of \(d\) belongs to any
Menon range \([P_j,Q_j]\).  Complete multiplicativity gives the exact
identity
\begin{equation}\label{eq:typical-dilation}
 \1_{\cT}(dm)\Li(dm)=\Li(d)\1_{\cT}(m)\Li(m).
\end{equation}
After substituting \(x=dy\) and applying character orthogonality modulo
\(q_0\), the contribution of the residue class \(b\) is at most
\begin{equation}\label{eq:character-reduction}
 \frac d{\varphi(q_0)}\sum_{\chi\pmod{q_0}}
 \int_Y^{2Y}\left|
  \sum_{\substack{y<m\le y+L\\m\in\cT}}
       \overline\chi(m)\Li(m)
 \right|\dd y,
\end{equation}
where
\begin{equation}\label{eq:Y-L-definition}
 Y=Z/d,\qquad L=K'/d.
\end{equation}
There is a single base interval \([Y,2Y]\); no dyadic block is asked to cross
a finite support cutoff.

Fix
\[
 \delta=\frac12\left(
 \frac{C_M\kappa_M}{24}-\frac1{15}-\frac52\right)>0
\]
and choose
\[
 D>\max\left\{3B_0+10,\ B_0\left(\frac52+\delta\right)+1\right\}.
\]
Then choose Menon's auxiliary
last-range parameter \(R_{\mathrm M}=R_{\mathrm M}(D)\) sufficiently large
as required in his Proposition~3.1 and Corollary~3.7.  These choices are
made before \(Z\) tends to infinity.  We now verify every remaining
hypothesis of Input~\ref{input:menon-character}.  First,
\begin{equation}\label{eq:YL-bounds}
 L\le K/d\le Z/d=Y,\qquad
 L>qQ_1/d\ge Q_1>10.
\end{equation}
Since \(d\le\Omega\le(\log Z)^{B_0}=Z^{o(1)}\), for sufficiently large
\(Z\),
\begin{equation}\label{eq:X0-local-bounds}
 Y^{1/2}\le X_0=2Z^{1/2}\le Y.
\end{equation}
Also \(\log Y\sim\log Z\), and therefore
\begin{equation}\label{eq:q0-bound}
 q_0\le\Omega\le(\log Y)^{B_0+1}.
\end{equation}
Finally, \eqref{eq:YL-bounds} gives \(m\le y+L\le3Y\).  Hence arithmetic
membership in \(\cT\) agrees exactly with Menon's Proposition~3.1 local set
on every integer in \eqref{eq:character-reduction}.  We may apply
\eqref{eq:menon-character} with \(B'=B_0+1\), \(D=D(B_0)\), and the fixed
choice \(R_{\mathrm M}=R_{\mathrm M}(D)\).

By \eqref{eq:CM-choice},
\eqref{eq:flexible-hyp}, and \(\log Q_1\le\log K\le\Omega^{1/5}\),
\begin{equation}\label{eq:character-first-term}
 \frac{(\log Q_1)^{1/3}}{P_1^{\kappa_M/24}}
 \le\Omega^{1/15-C_M\kappa_M/24}
 \ll\Omega^{-5/2-\delta}
\end{equation}
by the definition of \(\delta\).  Since
\(\Omega\le(\log Z)^{B_0}\) and \(\log Y\sim\log Z\), the initial choice
of \(D=D(B_0)\) gives
\begin{equation}\label{eq:character-second-term}
 (\log Y)^{-D}\ll_{B_0}\Omega^{-5/2-\delta}.
\end{equation}
It follows that
\begin{equation}\label{eq:character-L2}
 \int_Y^{2Y}\left|
  \sum_{\substack{y<m\le y+L\\m\in\cT}}
       \overline\chi(m)\Li(m)
 \right|^2\dd y
 \ll_{B_0}\Omega^{-5/2-\delta}L^2Y.
\end{equation}
Cauchy--Schwarz, with the strict \(\delta\)-margin discarded, yields
\begin{equation}\label{eq:character-L1}
 \int_Y^{2Y}\left|
  \sum_{\substack{y<m\le y+L\\m\in\cT}}
       \overline\chi(m)\Li(m)
 \right|\dd y
 \ll_{B_0}\Omega^{-5/4}LY.
\end{equation}

\subsection{Residue count and conclusion}

Insert \eqref{eq:character-L1} into
\eqref{eq:character-reduction}.  Since \(L=K'/d\) and \(Y=Z/d\), the
class satisfying \((b,q)=d\) contributes at most
\begin{equation}\label{eq:one-residue-class}
 \frac{K'Z}{d\Omega^{5/4}}.
\end{equation}
The residue count is exact:
\begin{equation}\label{eq:exact-divisor-count}
 \sum_{b\pmod q}\frac1{(b,q)}
 =\sum_{d\mid q}\frac{\varphi(q/d)}d
 \le\sum_{d\mid q}\varphi(q/d)=q.
\end{equation}
Thus, for every \(K'>qQ_1\),
\begin{equation}\label{eq:I-Kprime-bound}
 I(K')\ll_{B_0}\frac{qK'Z}{\Omega^{5/4}}.
\end{equation}
At \(K'=K\), this is \(\ll KZ/\Omega^{1/4}\), and
\begin{align}
 \frac{\Omega}{Kq}\int_{qQ_1}^KI(K')\dd K'
 &\ll_{B_0}
 \frac{\Omega}{Kq}\int_{qQ_1}^K
       \frac{qK'Z}{\Omega^{5/4}}\dd K' \notag\\
 &\ll\frac{KZ}{\Omega^{1/4}}.              \label{eq:major-integrated}
\end{align}
Together, \eqref{eq:major-partial-summation},
\eqref{eq:major-short-endpoints}, and \eqref{eq:major-integrated} prove
\eqref{eq:flexible-conclusion} on the major arcs.  The minor arcs were
settled in \eqref{eq:flexible-minor-result}, so
Proposition~\ref{prop:flexible-menon} follows.

\begin{remark}\label{rem:fixed-B}
The proof fixes \(B'=B_0+1\), a sufficiently large \(D=D(B_0)\), and then
Menon's auxiliary \(R_{\mathrm M}=R_{\mathrm M}(D)\) before using the
character estimate.  The implied constant and starting threshold are
therefore allowed to depend ineffectively on \(B_0\).
This is why the proposition applies to every fixed logarithmic exponent but
does not apply uniformly to an exponent growing with \(Z\).
\end{remark}

\section{One-sided truncation and uniform uncentring}\label{sec:uncentring}

Fix a nonempty \(I\subseteq[J]\), a dyadic block
\(\cD_I(M)\subset(M/2,M]\), and an integer interval
\(I_N=(N,2N]\cap\Z\).  Throughout this section assume
\begin{equation}\label{eq:uncentring-range}
 1\le h\le R^{3A},\qquad \log N\ge R^{2.1}-R,
\end{equation}
where \(A>0\) is fixed.  Set
\begin{equation}\label{eq:circle-cutoff}
 \WM=R^{12A+6},\qquad B_0=6A+3,
\end{equation}
and let
\begin{equation}\label{eq:K-definition}
 K=\lceil hM\rceil.
\end{equation}
The Menon typical set in this section is the set from
Proposition~\ref{prop:flexible-menon} with local base \(Z=N\), short length
\(K\), and circle cutoff \(\Omega=\WM\).

We first record the discrete consequence of the proposition.  For integer
\(K\), the summand in \eqref{eq:flexible-conclusion} is constant on unit
cells of \(x\).  Splitting \([N,2N]\) into such cells and absorbing the two
endpoint cells gives
\begin{equation}\label{eq:discrete-flexible-menon}
 \sup_{\alpha\in\T}\sum_{N<n\le2N}
 \left|\sum_{\substack{t\le K\\n+t\in\cT}}
       \Li(n+t)e(\alpha t)\right|
 \ll_A E_KKN+O(K),
\end{equation}
where
\begin{equation}\label{eq:EK}
 E_K=\frac{(\log K)^{1/4}\log\log K}{\WM^{1/4}}.
\end{equation}
All hypotheses of Proposition~\ref{prop:flexible-menon}, including those
needed after a later subset rescaling, are verified in
Section~\ref{sec:parameter-discharge}.

Define
\begin{equation}\label{eq:T-Mh}
 T_{M,h}=
 \sum_{n\in I_N}\sum_{d\in\cD_I(M)}
 \frac{\Li(n)\Li(n+hd)}d.
\end{equation}
Translation averaging gives
\begin{equation}\label{eq:translation-average}
 T_{M,h}
 =\frac1K\sum_{t\le K}\sum_{n\in I_N}
  \sum_{d\in\cD_I(M)}
  \frac{\Li(n+t)\Li(n+t+hd)}d
  +O(KV_I[M]).
\end{equation}
The boundary term follows because translating an interval by \(t\) changes
at most \(O(t)\) terms.

\subsection{The complement is removed before Fourier inversion}

Insert the typical cutoff only in the first Liouville factor:
\begin{align}\label{eq:one-sided-truncation}
 \Li(n+t)\Li(n+t+hd)
 &=\1_{\cT}(n+t)\Li(n+t)\Li(n+t+hd)\notag\\
 &\quad+\1_{\cT^c}(n+t)\Li(n+t)\Li(n+t+hd).
\end{align}
This identity avoids asking one typical set to control two different short
lengths.  By Input~\ref{input:sieve-complement}, Fubini, and
\(|\Li|\le1\), the contribution of the second line is
\begin{equation}\label{eq:complement-bound}
 |\operatorname{Comp}_{M,h}|
 \ll NV_I[M]\frac{\log P_1}{\log Q_1}.
\end{equation}
The factor \(K\) counting translations is cancelled by the \(1/K\) in
\eqref{eq:translation-average}; in particular, no factor \(h\) occurs.

Here
\begin{equation}\label{eq:P1Q1-application}
 P_1=\WM^{C_M},\qquad Q_1=K/\WM^3,
\end{equation}
because Section~\ref{sec:parameter-discharge} shows \(K<K_0\).  Hence
\begin{equation}\label{eq:complement-ratio}
 \frac{\log P_1}{\log Q_1}
 \ll_A\frac{\log R}{\log M}.
\end{equation}

\subsection{Fourier inversion on the typical part}

Define
\[
 F_n(\alpha)=\sum_{t\le K}
  \1_{\cT}(n+t)\Li(n+t)e(\alpha t),
\]
\[
 G_n(\alpha)=\sum_{k\le2K}\Li(n+k)e(-\alpha k),
 \qquad
 Q_h(\alpha)=\sum_{d\in\cD_I(M)}\frac{e(\alpha hd)}d.
\]
Fourier inversion gives the typical main term exactly as
\begin{equation}\label{eq:typical-fourier}
 \frac1K\sum_{n\in I_N}\int_0^1
 Q_h(\alpha)F_n(\alpha)G_n(\alpha)\dd\alpha.
\end{equation}
Multiplication by the integer \(h\) preserves Haar measure on \(\T\), so
Input~\ref{input:fourth-moment} gives
\begin{equation}\label{eq:Qh-fourth}
 \int_0^1|Q_h(\alpha)|^4\dd\alpha
 =\int_0^1|Q_{I,M}(\alpha)|^4\dd\alpha
 \ll\frac{V^{4J}}{M(\log M)^4}.
\end{equation}

Put
\[
 \mathcal S(\alpha)=\sum_{n\in I_N}|F_n(\alpha)G_n(\alpha)|.
\]
The bound \(\|G_n\|_\infty\le2K\) and
\eqref{eq:discrete-flexible-menon} give
\begin{equation}\label{eq:fixed-product-linfty}
 \|\mathcal S\|_\infty\ll_A E_KK^2N.
\end{equation}
Here the endpoint term \(O(K)\) in
\eqref{eq:discrete-flexible-menon} is absorbed by its \(E_KKN\) term.
Indeed, \(K\ge M\ge H_0\) and \(\WM=R^{12A+6}\) give
\[
 E_K^{-1}
 =\frac{\WM^{1/4}}{(\log K)^{1/4}\log\log K}
 \ll_A R^{3A+3/2},
\]
whereas \(N\ge\exp(R^{2.1}-R)\).  Thus \(E_KN\gg_A 1\) for all
sufficiently large \(R\).  For each \(n\), Parseval and
Cauchy--Schwarz give \(\int_\T|F_nG_n|\ll K\), and hence
\begin{equation}\label{eq:fixed-product-lone}
 \|\mathcal S\|_1\ll KN.
\end{equation}
Interpolation between \eqref{eq:fixed-product-linfty} and
\eqref{eq:fixed-product-lone} gives
\[
 \|\mathcal S\|_{4/3}\ll_A E_K^{1/4}K^{5/4}N.
\]
Denote by \(\operatorname{Typ}_{M,h}\) the typical contribution inside the
translation average in \eqref{eq:translation-average}.  Direct H\"older
in \eqref{eq:typical-fourier}, followed by \eqref{eq:Qh-fourth} and
\(K/M\ll h\), gives
\begin{equation}\label{eq:typical-block-bound}
 |\operatorname{Typ}_{M,h}|
 \ll_A Nh^{1/4}E_K^{1/4}\frac{V^J}{\log M}.
\end{equation}

\subsection{Summation over the divisor scale}

Since
\begin{equation}\label{eq:log-K-log-M}
 \log K=\log M+O_A(\log R),
\end{equation}
we have
\begin{equation}\label{eq:EK-fourth-root}
 \begin{aligned}
 E_K^{1/4}\ll_A{}&\WM^{-1/16}(\log M)^{1/16}\\
 &\times(\log\log M)^{1/4}.
 \end{aligned}
\end{equation}
Because \(M\) is dyadic and \(H_0\le M\le H=e^R\), write
\(M=2^k\).  Comparison of
\((k\log 2)^{-15/16}(\log(k\log 2))^{1/4}\) with the corresponding
integral in \(k\) (equivalently, in a variable comparable to \(\log M\))
yields
\begin{equation}\label{eq:dyadic-log-sum}
 \sum_{M\text{ dyadic}}(\log M)^{-15/16}
                 (\log\log M)^{1/4}
 \ll R^{1/16}(\log R)^{1/4}.
\end{equation}
Insert the fourth-root bound into the block estimate and use
\eqref{eq:dyadic-log-sum}.  Therefore
\begin{equation}\label{eq:typical-summed}
 \sum_M|\operatorname{Typ}_{M,h}|
 \ll_A V^JN h^{1/4}\WM^{-1/16}
 R^{1/16}(\log R)^{1/4}.
\end{equation}
The parameter choice \eqref{eq:circle-cutoff} was made so that
\begin{equation}\label{eq:h-cancellation}
 h^{1/4}\WM^{-1/16}
 \le R^{3A/4}R^{-(12A+6)/16}=R^{-3/8}.
\end{equation}
Thus the typical contribution is
\begin{equation}\label{eq:typical-power-saving}
 \sum_M|\operatorname{Typ}_{M,h}|
 \ll_A V^JNR^{-5/16}(\log R)^{1/4}.
\end{equation}
The complement is smaller, since
\begin{equation}\label{eq:complement-summed}
 \sum_M|\operatorname{Comp}_{M,h}|
 \ll_A NV^J\frac{\log R}{R^{1-\eps_1}}
 =o(V^JNR^{-3/10}).
\end{equation}
The translation boundaries are also negligible:
\begin{equation}\label{eq:translation-summed}
 \sum_M K V_I[M]\ll hHV^J=o(V^JNR^{-100}),
\end{equation}
because \(h\le R^{3A}\), \(H=e^R\), and
\(N\ge\exp(R^{2.1})\).  Finally,
\((\log R)^{1/4}\ll R^{1/80}\), so
\begin{proposition}[Uniform local uncentring]
\label{prop:uniform-local-uncentring}
Under \eqref{eq:uncentring-range},
\begin{equation}\label{eq:uncentring-final}
 \sum_M|T_{M,h}|
 \ll_A V^JNR^{-3/10}.
\end{equation}
The constant is uniform in the interval \(I_N\) and in every local base
\(N\) satisfying the displayed lower bound.
\end{proposition}

\section{Block expansion and the centred form}\label{sec:centering}

Define
\begin{equation}\label{eq:S1-definition}
 S_{1,h}=\sum_{n\in I_N}\sum_{\substack{d\in\cD\\d\mid n}}
 \Li(n)\Li(n+hd)
\end{equation}
and
\begin{equation}\label{eq:S2-definition}
 S_{2,h}=\sum_{n\in I_N}
 \sum_{d=p_1\cdots p_J\in\cD}
 \Li(n)\Li(n+hd)
 \prod_{j=1}^J\left(\1_{p_j\mid n}-\frac1{p_j}\right).
\end{equation}
We first compare the uncentred and centred expressions and then estimate the
centred one by Input~\ref{input:tt-decoupling}.

\subsection{Subset expansion}

Expanding the product in \eqref{eq:S2-definition} gives the exact identity
\begin{multline}\label{eq:subset-expansion}
 S_{2,h}-S_{1,h}
 =\sum_{\emptyset\ne I\subseteq[J]}(-1)^{|I|}
 \sum_{(p_1,\ldots,p_J)}
 \left(\prod_{i\in I}\frac1{p_i}\right)\\
 \times\sum_{n\in I_N}
 \left(\prod_{i\notin I}\1_{p_i\mid n}\right)
 \Li(n)\Li(n+hp_1\cdots p_J).
\end{multline}
For a fixed tuple, put
\[
 d_0=\prod_{i\notin I}p_i,\qquad
 d'=\prod_{i\in I}p_i,
\]
and write \(n=d_0m\).  Complete multiplicativity yields
\begin{equation}\label{eq:lambda-rescaling}
 \Li(d_0m)\Li(d_0m+hd_0d')
 =\Li(d_0)^2\Li(m)\Li(m+hd')
 =\Li(m)\Li(m+hd').
\end{equation}
The remaining product \(d'\) carries the reciprocal weight required in
Section~\ref{sec:uncentring}.  The outer scale is now \(N/d_0\), and
\begin{equation}\label{eq:rescaled-outer-log}
 \log(N/d_0)\ge R^{2.1}-R
\end{equation}
because \(d_0\le H=e^R\).

Apply Proposition~\ref{prop:uniform-local-uncentring} with local base
\(Z=N/d_0\).  Its lower hypothesis is exactly
\eqref{eq:rescaled-outer-log}; the remaining flexible-Menon hypotheses are
uniform at this base by Section~\ref{sec:parameter-discharge}.  The
proposition contributes a factor \(V^J(N/d_0)\).  Summing over the fixed
factors uses
\begin{equation}\label{eq:fixed-factor-mass}
 \sum_{d_0\in\cD_{[J]\setminus I}}\frac1{d_0}
 =\prod_{i\notin I}V_i\le V^J.
\end{equation}
Thus all nonempty subsets cost at most \(2^JV^{2J}\).
Combining \eqref{eq:subset-loss} with \eqref{eq:uncentring-final} gives
\begin{equation}\label{eq:centered-uncentered}
 |S_{1,h}-S_{2,h}|
 \ll_A N R^{1/40-3/10}\ll NR^{-1/4}.
\end{equation}
The saving exponent is absolute.

\subsection{The centred term}

Since \(h\le R^{3A}<H_0\), every block prime is coprime to \(h\).  Split
\(S_{2,h}\) according to \(n=r+hq\).  For each \(p\in\cP\),
\begin{equation}\label{eq:residue-change}
 p\mid r+hq
 \quad\Longleftrightarrow\quad
 q\equiv-rh^{-1}\pmod p.
\end{equation}
Set \(g_r(q)=\Li(r+hq)\).  Then
\[
 g_r(q+d)=\Li(r+hq+hd).
\]
The \(q\)-interval \(I_r\) corresponding to \(n\in I_N\) has cardinality
\(L_r\), and
\begin{equation}\label{eq:Lr-lower}
 \log L_r\ge R^{2.1}-3A\log R+o(1)>R^{2.01}
\end{equation}
for sufficiently large \(R\).  Apply Input~\ref{input:tt-decoupling} with
\(I'=[J]\), \(\sigma=+1\), the residue assignments in
\eqref{eq:residue-change}, and \(g_1=g_2=g_r\).  Triangle inequality over
the prime tuples gives
\begin{equation}\label{eq:one-residue-centered}
 |S_{2,h}^{(r)}|\ll V^{0.51J}L_r.
\end{equation}
The residue intervals partition \(I_N\), so
\(\sum_{r\pmod h}L_r=|I_N|\).  Consequently
\begin{equation}\label{eq:centered-bound}
 |S_{2,h}|\ll V^{0.51J}N.
\end{equation}
Combining \eqref{eq:centered-uncentered} and
\eqref{eq:centered-bound}, we obtain the dyadic estimate
\begin{proposition}\label{prop:dyadic-correlation}
Uniformly for \(h\le R^{3A}\) and \(\log N\ge R^{2.1}\),
\begin{equation}\label{eq:S1-final}
 |S_{1,h}|\ll_A V^{0.51J}N+NR^{-1/4}.
\end{equation}
\end{proposition}

\section{Parameter discharge and proof of the main theorem}
\label{sec:main-proof}

\subsection{The flexible-Menon hypotheses}\label{sec:parameter-discharge}

In the subset expansion of Section~\ref{sec:centering}, the local base is
\(Z=N/d_0\), where \(d_0\le H=e^R\).  Thus
\begin{equation}\label{eq:local-base-log}
 \log Z\ge R^{2.1}-R,
 \qquad K=\lceil hM\rceil,
 \qquad H_0\le M\le H.
\end{equation}
We check Proposition~\ref{prop:flexible-menon} with the parameters in
\eqref{eq:circle-cutoff}.  For every fixed \(A>0\), as \(R\to\infty\),
\begin{equation}\label{eq:lower-circle-hyp}
 (\log K)^5\le R^{12A+6}=\WM.
\end{equation}
Also
\begin{equation}\label{eq:power-circle-hyp}
 \log\WM=O_A(\log R)=o(\log K),
 \qquad\text{hence}\qquad
 \WM\le K^{1/(2C_M)}.
\end{equation}
The choice \(B_0=6A+3\) gives
\begin{equation}\label{eq:log-circle-hyp}
 \WM=R^{12A+6}
 <(R^{2.1}-R)^{6A+3}
 \le(\log Z)^{B_0}.
\end{equation}
Moreover,
\begin{equation}\label{eq:K-less-Z}
 \log K\le R+O_A(\log R)<R^{2.1}-R\le\log Z.
\end{equation}
Finally, the truncation scale in \eqref{eq:flexible-scales} satisfies
\begin{equation}\label{eq:K0-large}
 \log K_0\asymp\frac{\log Z}{\log\log Z}
 \gg\frac{R^{2.1}}{\log R}\gg\log K.
\end{equation}
Therefore \(\widetilde K=K\),
\(Q_1=K/\WM^3\), and every hypothesis of
Proposition~\ref{prop:flexible-menon} holds at every local scale used in the
proof.  The same estimates show \(h<H_0\), as was used in
Section~\ref{sec:centering}.

\subsection{Logarithmic dilation}

For
\[
 C_h(x)=\sum_{n\le x}\frac{\Li(n)\Li(n+h)}n,
\]
complete multiplicativity gives, for every \(d\in\cD\),
\[
 C_h(x)=d\sum_{\substack{m\le dx\\d\mid m}}
 \frac{\Li(m)\Li(m+hd)}m.
\]
After multiplication by \(1/d\) and summation over \(d\),
\begin{equation}\label{eq:logarithmic-dilation}
 \WP C_h(x)
 =\sum_{d\in\cD}\sum_{\substack{m\le dx\\d\mid m}}
 \frac{\Li(m)\Li(m+hd)}m.
\end{equation}
The part with \(x<m\le dx\) is bounded by
\begin{equation}\label{eq:dilation-tail}
 \sum_{d\in\cD}\sum_{\substack{x<m\le dx\\d\mid m}}\frac1m
 \ll\sum_{d\in\cD}\frac{\log d+1}{d}
 \ll\WP\log H.
\end{equation}

Define the unweighted prefix
\[
 \mathcal A_h(M)=\sum_{d\in\cD}\sum_{\substack{m\le M\\d\mid m}}
 \Li(m)\Li(m+hd).
\]
Decompose \((1,M]\) into intervals of the form \((N,2N]\).  On the pieces
with \(\log N\ge R^{2.1}\), apply
Proposition~\ref{prop:dyadic-correlation}; below that threshold use the
trivial divisor bound.  The lengths of the pieces in each range form a
geometric series, so, uniformly for every \(M\ge1\),
\begin{equation}\label{eq:all-prefixes}
 |\mathcal A_h(M)|\ll_A
 \WP\min\{M,e^{R^{2.1}}\}
 +\left(V^{0.51J}+R^{-1/4}\right)M.
\end{equation}
Partial summation in \eqref{eq:all-prefixes} gives
\[
 \left|\sum_{d\in\cD}\sum_{\substack{m\le x\\d\mid m}}
 \frac{\Li(m)\Li(m+hd)}m\right|
 \ll_A \WP R^{2.1}
 +\left(V^{0.51J}+R^{-1/4}\right)\log x.
\]
Dividing \eqref{eq:logarithmic-dilation} by \(\WP\), adding
\eqref{eq:dilation-tail}, and using
\(\WP\asymp_{\eps_1}V^J\to\infty\) from
\eqref{eq:block-masses}, gives, after the harmless weakening
\(R^{-1/4}/\WP\le R^{-1/4}\),
\begin{equation}\label{eq:pre-final-bound}
 |C_h(x)|
 \ll_A R^{2.1}+R+
 \left(R^{-1/4}+\frac{V^{0.51J}}{\WP}\right)\log x.
\end{equation}
Here the first term comes from the prefix below the threshold and the second
from \eqref{eq:dilation-tail}; the final factor \(\log x\) is the integral
of the linear term in \eqref{eq:all-prefixes} under partial summation.

By \eqref{eq:block-masses},
\begin{equation}\label{eq:decoupling-ratio}
 \frac{V^{0.51J}}{\WP}
 \ll_{\eps_1}V^{-0.49J}
 \ll R^{-c_{\mathrm{dec}}}
\end{equation}
for an absolute \(c_{\mathrm{dec}}>0\), once \(\eps_1\) has been fixed.

\subsection{Choice of the outer scale}

First suppose
\begin{equation}\label{eq:xR-relation}
 \log x=R^3.
\end{equation}
Then \(h\le(\log x)^A\) is exactly \(h\le R^{3A}\), and
\eqref{eq:pre-final-bound} becomes
\begin{equation}\label{eq:final-exponents}
 |C_h(x)|\ll_A
 (\log x)^{0.7}
 +(\log x)^{1-1/12}
 +(\log x)^{1-c_{\mathrm{dec}}/3}.
\end{equation}
Hence \eqref{eq:main-theorem} holds for any fixed
\begin{equation}\label{eq:absolute-c}
 0<c<\min\left\{\frac3{10},\frac1{12},
                         \frac{c_{\mathrm{dec}}}{3}\right\}.
\end{equation}
This \(c\) is independent of \(A\).

To cover every sufficiently large real \(x\), take
\[
 H=\left\lceil\exp((\log x)^{1/3})\right\rceil,
 \qquad x_H=\exp((\log H)^3).
\]
Then
\begin{equation}\label{eq:all-real-x}
 |\log x_H-\log x|\ll\frac{(\log H)^2}{H}=o(1).
\end{equation}
The harmonic contribution between \(x\) and \(x_H\) is \(o(1)\).
Moreover, \(\log H\ge(\log x)^{1/3}\), so the shift range proved at \(x_H\)
contains the desired range at \(x\).  This completes the proof of
Theorem~\ref{thm:main}.

\section{Progression-of-starts Fourier uniformity}
\label{sec:progression-fourier}

The fixed-polylogarithmic argument in the preceding sections translates an
ordinary interval through length \(hM\).  That translation is harmless when
\(h\) is a fixed power of \(R\), but it is not suitable when
\(\log h\) itself is a power of \(R\).  We instead keep the starts of the
progressions as the averaging variable.  For \(Z,M,h\ge1\), put
\begin{equation}\label{eq:apf-definition}
 \mathcal F_{h,M}(\alpha;Z)
 :=\sum_{Z<y\le2Z}
 \left|\sum_{1\le t\le M}\Li(y+ht)e(\alpha t)\right|,
 \qquad e(u):=e^{2\pi iu}.
\end{equation}
The frequency \(\alpha\) in \eqref{eq:apf-definition} is common to all
starting points.  This feature is essential in the later Fourier inversion.

Fix constants
\begin{equation}\label{eq:super-parameter-region}
 \begin{gathered}
 a>0,\qquad 0<\rho<5,\qquad 2.01<\gamma<\beta,\qquad \rho<\gamma,\\
 a<\frac{1-\varepsilon_b}{200},\qquad
 \rho<\gamma-200a.
 \end{gathered}
\end{equation}
Here \(\varepsilon_b>0\) is a fixed Pilatte block parameter, independent of
\(\eps_1\) fixed in Section~\ref{sec:inputs}.  The second block system
controlled by \(\varepsilon_b\) is introduced explicitly in
Section~\ref{sec:superpolylog-transfer}; its parameter is allowed to depend
on the fixed exponent \(\eta\) in Theorem~\ref{thm:super-main}.  Let
\begin{equation}\label{eq:super-local-scales}
 H_R=\lfloor e^{R^\rho}\rfloor,\qquad
 \Omega=R^2,\qquad \epsilon_K=R^{-a},
\end{equation}
and suppose
\begin{equation}\label{eq:apf-local-range}
 1\le h\le H_R,\qquad
 e^{R^{1-\varepsilon_b}}\le M\le e^R,\qquad
 \log Z\ge R^\gamma-R.
\end{equation}

\subsection{The joint Ramar\'e minor arc}

The first new point is that the minor arcs admit a power saving independent
of \(h\) throughout \eqref{eq:apf-local-range}.  Put
\begin{equation}\label{eq:apf-first-prime-range}
 P_1=\Omega^{C_M},\qquad Q_1=M/\Omega^3,\qquad
 \mathcal P_h=\{p\in[P_1,Q_1]:p\nmid h\},
\end{equation}
where \(C_M\) is the fixed large constant used in
Section~\ref{sec:flexible-minor}.  Put \(X_0=(4Z)^{1/2}\), and let
\(\cT_h=\cT(P_1,Q_1,X_0)\) be the typical set from
Definition~\ref{def:typical-set}, with the first prime factor required to
belong to \(\mathcal P_h\).  Excluding primes which divide \(h\) changes
the density by at most
\begin{equation}\label{eq:shared-prime-density}
 \sum_{\substack{p\mid h\\p\ge P_1}}\frac1p+O(P_1^{-1})
 \le \frac{\log h}{P_1\log P_1}+O(P_1^{-1})
 \ll R^{-100}.
\end{equation}
Together with Input~\ref{input:sieve-complement} and translation of the
physical interval, this gives
\begin{equation}\label{eq:apf-typical-complement}
 \sum_{Z<y\le2Z}\sum_{t\le M}1_{\cT_h^c}(y+ht)
 \ll MZ\frac{\log P_1}{\log Q_1}+R^{-100}MZ+hM^2.
\end{equation}
The last term is smaller than \(R^{-A}MZ\) for every fixed \(A>0\), since
\(
 \log(hM/Z)\le R^\rho+2R-R^\gamma\to-\infty.
\)

\begin{lemma}[Joint minor-arc estimate]\label{lem:apf-minor}
Suppose that \(a_0/q\) is reduced and
\begin{equation}\label{eq:apf-minor-approximant}
 \Omega<q\le M/\Omega,\qquad
 |\alpha-a_0/q|\le\frac{\Omega}{Mq}.
\end{equation}
Then, uniformly in \eqref{eq:apf-local-range},
\begin{equation}\label{eq:apf-minor-conclusion}
 \mathcal F_{h,M}(\alpha;Z)\ll R^{-1/5}MZ.
\end{equation}
\end{lemma}

\begin{proof}
Dualizing the outer \(\ell^1\)-norm gives
\begin{equation}\label{eq:apf-dual}
 \mathcal F_{h,M}(\alpha;Z)
 =\sup_{|\vartheta_y|\le1}
 \left|\sum_{Z<y\le2Z}\vartheta_y
       \sum_{t\le M}\Li(y+ht)e(\alpha t)\right|.
\end{equation}
Write
\(
 \widetilde\vartheta_y=\vartheta_y e(-\alpha y/h)
\)
and define the sparse window
\begin{equation}\label{eq:sparse-window}
 J_h(u)=
 \sum_{\substack{Z<y\le2Z\\u-hM\le y<u\\y\equiv u\, (h)}}
 \widetilde\vartheta_y.
\end{equation}
Since \(e(\alpha t)=e(\alpha(y+ht)/h)e(-\alpha y/h)\), the typical part
of \eqref{eq:apf-dual} may be expanded by Ramar\'e's identity on the
first prime range.  On a dyadic block \(p\sim P\) this gives
\begin{equation}\label{eq:apf-ramare-form}
 \mathcal L_P=
 \sum_m a_m\sum_{\substack{p\sim P\\p\in\mathcal P_h}}
 \Li(p)e(\alpha mp/h)J_h(mp),\qquad |a_m|\le1.
\end{equation}
Indeed, before absorbing the \(m\)-factors, the exact denominator is
\(
 1_{p\nmid m}+\omega_{\mathcal P_h}(m).
\)
Replacing \(1_{p\nmid m}\) by \(1\) costs
\begin{equation}\label{eq:apf-prime-square}
 O(MZ/P+hM^2/P),
\end{equation}
and all later prime ranges are disjoint from the first, so their indicators
depend only on \(m\).  The physical support remains in \(J_h(mp)\), not in
a hidden \(p\)-dependent coefficient.

H\"older's inequality bounds \eqref{eq:apf-ramare-form} by
\begin{equation}\label{eq:apf-holder}
 |\mathcal L_P|
 \le (Z/P)^{3/4}
 \left(\sum_{m\ll Z/P}
 \left|\sum_{p\sim P}\Li(p)e(\alpha mp/h)J_h(mp)\right|^4
 \right)^{1/4}.
\end{equation}
We now expand the fourth power before estimating the sparse windows.  For
fixed \(p_1,\ldots,p_4\sim P\), a contributing tuple satisfies
\begin{equation}\label{eq:sparse-tuple-conditions}
 y_i<mp_i\le y_i+hM,\qquad y_i\equiv mp_i\pmod h.
\end{equation}
Fixing \(y_1\), each of \(y_2,y_3,y_4\) lies in an interval of length
\(O(hM)\) and in one residue class modulo \(h\).  There are therefore
\(O(M)\) choices for each, and the number of compatible start tuples is
\begin{equation}\label{eq:sparse-tuple-count}
 O(ZM^3).
\end{equation}
For a fixed tuple, write \(m=m_0+h\ell\).  The common \(m\)-interval has
length at most \(hM/P\), and if
\(
 r=p_1+p_2-p_3-p_4,
\)
the remaining geometric sum is bounded by
\begin{equation}\label{eq:sparse-geometric-sum}
 \min\left\{M/P+1,\|r\alpha\|^{-1}\right\}.
\end{equation}
This includes repeated primes and \(r=0\).

The standard upper-bound sieve estimate
\begin{equation}\label{eq:four-prime-apf}
 \#\{p_1+p_2-p_3-p_4=r:p_i\sim P\}
 \ll P^3/(\log P)^4
\end{equation}
and \eqref{eq:sparse-tuple-count}--\eqref{eq:sparse-geometric-sum}
give
\begin{multline}\label{eq:apf-fourth-expanded}
 \sum_m\left|\sum_{p\sim P}\Li(p)e(\alpha mp/h)J_h(mp)\right|^4\\
 \ll \frac{ZM^3P^3}{(\log P)^4}
 \sum_{|r|\ll P}
 \min\left\{M/P+1,\|r\alpha\|^{-1}\right\}.
\end{multline}
Vinogradov's lemma, \eqref{eq:apf-minor-approximant}, and
\(P\le Q_1=M/\Omega^3\) imply
\begin{equation}\label{eq:apf-vinogradov}
 \sum_{|r|\ll P}
 \min\left\{M/P+1,\|r\alpha\|^{-1}\right\}
 \ll M\log M/\Omega.
\end{equation}
Substitution in \eqref{eq:apf-holder}, followed by dyadic summation in
\(P\), yields
\begin{equation}\label{eq:apf-minor-raw}
 \mathcal F^{\rm typ}_{h,M}(\alpha;Z)
 \ll MZ\frac{(\log M)^{1/4}\log\log M}{\Omega^{1/4}}.
\end{equation}
Since \(\Omega=R^2\) and \(\log M\le R\), the right side is
\(O(R^{-1/5}MZ)\).  Equations
\eqref{eq:apf-typical-complement} and \eqref{eq:apf-prime-square}
are smaller after summation, proving \eqref{eq:apf-minor-conclusion}.
\end{proof}

\subsection{Decoupling a moving residue from its endpoint}

For a 1-bounded sequence \(f\), a real number \(x\), and integers
\(a,L,Q\), write
\begin{equation}\label{eq:progression-sum-def}
 S_f(x,a;L,Q)=
 \sum_{\substack{x<n\le x+L\\n\equiv a\, (Q)}}f(n).
\end{equation}

\begin{lemma}[Endpoint--residue decoupling]\label{lem:endpoint-residue}
Let \(a_y\equiv y+c\pmod Q\).  Then
\begin{equation}\label{eq:endpoint-residue}
 \sum_{Z<y\le2Z}|S_f(y,a_y;L,Q)|
 \le \frac1Q\int_{Z-Q}^{2Z}
       \sum_{a\bmod Q}|S_f(x,a;L,Q)|\,\dd x+O(Z).
\end{equation}
The implicit constant is absolute.
\end{lemma}

\begin{proof}
For \(0\le s<Q\), moving the left endpoint from \(y\) to \(y-s\), while
keeping \(a_y\) fixed, changes at most two terms.  Average the resulting
inequality over \(s\in[0,Q\)).  For fixed \((x,a)\), the equations
\(
 x=y-s, a\equiv y+c\pmod Q
\)
determine exactly one integer \(y\in[x,x+Q\)), apart from null endpoints.
This gives \eqref{eq:endpoint-residue}; the two endpoint terms contribute
\(O(Z)\).
\end{proof}

Lemma~\ref{lem:endpoint-residue} is the mechanism which avoids a
\(\sqrt{hq}\) loss on the major arcs: it turns the diagonal family of
residue classes selected by the starts \(y\) into the full residue-class
variance before Cauchy--Schwarz is applied.

\section{Major arcs and one exceptional modulus}
\label{sec:exceptional-modulus}

We complete the progression Fourier estimate on the major arcs.  The
published input is the real-valued hybrid variance theorem of Klurman,
Mangerel, and Ter\"av\"ainen \(KMT\), whose Corollary~1.6 retains a
possible real-character main term
\citep[Corollaries~1.6--1.7]{klurman2023shortaps}.  Suppose
\begin{equation}\label{eq:kmt-specialized-range}
 X\ge L\ge10Q\ge10,\qquad
 (\log(L/Q))^{-1/200}\le\epsilon\le1,\qquad
 Q\le X^{\epsilon^{200}},
\end{equation}
and \(Q\) is \((L/Q)^{\epsilon^2}\)-typical in the sense of KMT.  After
the Liouville-specific deletion proved in Lemma~\ref{lem:kmt-liouville},
KMT's variance estimate gives
\begin{equation}\label{eq:kmt-specialized-variance}
 \int_X^{2X}\sum_{b\bmod Q}^{*}
 |S_{\Li}(u,b;L,Q)|^2\,\dd u
 \ll \epsilon\varphi(Q)X(L/Q)^2.
\end{equation}
The superscript star restricts to reduced residue classes.  KMT's
Corollary~1.7 performs the deletion for \(\mu\), except possibly on moduli
induced by one primitive real character.  We now adapt that proof to
\(\Li\), and then show that one character works throughout the outer
argument.

The cap \(Q\leq X^{\epsilon^{200}}\) above is used only to obtain
automatic good-modulus membership in the subpower-range argument of this
section.  Sections~\ref{sec:general-good-deletion}--
\ref{sec:mellin-localized-transfer} instead use KMT's
general-good-modulus form and retain the complementary bad strata with
their exact linear weights.

\subsection{The Liouville main term}

\begin{lemma}[KMT main-term deletion for \(\Li\)]
\label{lem:kmt-liouville}
Let fixed \(a,\rho,\gamma\) satisfy
\begin{equation}\label{eq:kmt-liouville-parameters}
 a>0,\qquad 0<\rho<\gamma-200a,\qquad \epsilon=R^{-a}.
\end{equation}
Suppose \(q\le e^{R^\rho}R^2\),
\(\log X\ge R^\gamma-R^\rho-O(R)\), and the remaining hypotheses of
\eqref{eq:kmt-specialized-range} hold.  Let \(q_0\) be the conductor of
the possible Landau--Page primitive real character in the KMT proof.  If
no such conductor exists, or if \(q_0\nmid q\), then every character
\(\chi\pmod q\) satisfies
\begin{equation}\label{eq:liouville-character-mean}
 \left|\sum_{n\le3X}\Li(n)\chi(n)\right|
 \ll \epsilon^{1/2}\frac{\varphi(q)}qX.
\end{equation}
Consequently the KMT character main term is absorbed by the square-root
variance error.
\end{lemma}

\begin{proof}
For a nonprincipal character, KMT's Hal\'asz reduction (Lemma~7.4) reduces
the claim to a lower bound for the pretentious distance.  The relevant prime
sum, uniformly for the twists \(|t|\leq\log X\), is
\begin{equation}\label{eq:kmt-prime-distance}
 \sum_{X^{\epsilon^{10}}\le p\le X}
 \frac{1+\Re(\chi(p)p^{it})}{p}.
\end{equation}
It is unchanged when \(\mu\) is replaced by \(\lambda\), since
\(\mu(p)=\lambda(p)=-1\) at every prime.  KMT Lemmas~7.4 and~7.9 and
Remark~7.2 therefore give the same distance lower bound for \(\lambda\),
apart from branches induced by the possible exceptional primitive character.
The hypothesis \(q_0\nmid q\) excludes all those induced branches, not merely
the primitive character itself.  More explicitly, for every nonprincipal
\(\chi\pmod q\), the distance identity
\[
 \mathbb D_q(\lambda,\chi(n)n^{it};3X)^2
 =\mathbb D_q(\mu,\chi(n)n^{it};3X)^2
\]
holds because the distance is a prime sum and \(\lambda(p)=\mu(p)=-1\).
The auxiliary lower endpoint in KMT's Lemma~7.9 may be taken as
\(y=X^{\epsilon^{10}}\).  Its condition
\(q^{1/\epsilon^{100}}\le y\) follows here from
\(\rho+110a<\gamma\), which is implied by the stronger hypothesis
\(\rho<\gamma-200a\); the required uniformity in \(t\) includes
\(|t|\le\log X\).
Consequently,
\[
 \inf_{|t|\leq\log X}
 \mathbb D_q(\lambda,\chi(n)n^{it};3X)^2
 \geq\log(1/\epsilon),
\]
and KMT's Hal\'asz inequality (Lemma~7.4), with its parameter
\(T=\log X\), gives the required bound.  Indeed, its quantity
\(M_q(T)\) is at least \(\log(1/\epsilon)\), so
\[
 (M_q(T)+1)e^{-M_q(T)}+T^{-1/2}+(\log X)^{-1/4}
 \ll \epsilon^{1/2}
\]
under the fixed parameter inequalities.  Multiplying by
\((\varphi(q)/q)X\) gives the claimed estimate.  KMT Lemma~7.9 is stated
for arbitrary characters modulo \(q\), so it already includes imprimitive
characters and their induced primitive conductors.  In KMT's notation the
only exceptional alternative is that the inducing conductor is divisible by
\(q_0\), which is excluded by \(q_0\nmid q\).

For the principal character \(\chi_0\pmod q\), inclusion--exclusion gives
\begin{equation}\label{eq:principal-liouville-ie}
 \sum_{n\le3X}\Li(n)\chi_0(n)
 =\sum_{d\mid\operatorname{rad}(q)}
   \mu(d)\Li(d)\sum_{m\le3X/d}\Li(m).
\end{equation}
The classical zero-free-region prime number theorem for Liouville,
obtained from
\(
 \sum\Li(n)n^{-s}=\zeta(2s)/\zeta(s),
\)
is uniform over these divisors because
\(
 \log(X/d)=\log X+O(R^\rho+\log R)\asymp\log X.
\)
It yields
\begin{equation}\label{eq:principal-liouville-pnt}
 \sum_{m\le3X/d}\Li(m)
 \ll \frac Xd
 \exp\{-c(\log X)^{3/5}(\log\log X)^{-1/5}\}.
\end{equation}
Put
\[
 E_X:=\exp\{-c(\log X)^{3/5}(\log\log X)^{-1/5}\}.
\]
After summing \(d\mid\operatorname{rad}(q)\),
\eqref{eq:principal-liouville-ie}--
\eqref{eq:principal-liouville-pnt} give the correct factor
\[
 \left|\sum_{n\le3X}\Li(n)\chi_0(n)\right|
 \ll XE_X\prod_{p\mid q}\left(1+\frac1p\right).
\]
Moreover,
\[
 \prod_{p\mid q}\left(1+\frac1p\right)\frac q{\varphi(q)}
 \ll (1+\log\log(3q))^2\ll_\rho(\log R)^2.
\]
Here the last inequality uses \(q\le e^{R^\rho}R^2\).  Since the lower
bound for \(\log X\) is a fixed positive power of \(R\), the factor
\(E_X(\log R)^2\) is smaller than \(R^{-C}\) for every fixed \(C>0\).
Taking \(C>a/2\) proves
\eqref{eq:liouville-character-mean}.  In KMT Corollary~1.6 the resulting
main term per reduced residue is \(O(\epsilon^{1/2}L/q)\), which is no
larger than the square-root variance contribution.
The real-valued form of KMT Corollary~1.6 reduces its displayed main term
to the untwisted estimate above; the same Hal\'asz reduction is uniform in
the twists used in the proof of that corollary.
\end{proof}

\subsection{Exact rational frequencies}

\begin{lemma}[Rational major arc]\label{lem:apf-rational-major}
Assume \eqref{eq:super-parameter-region}--\eqref{eq:apf-local-range}, and
let \(a_0/q\) be reduced with \(q\le\Omega\).  If no KMT main term survives
for any divisor of \(hq\), then
\begin{equation}\label{eq:rational-major-conclusion}
 \mathcal F_{h,M}(a_0/q;Z)\ll R^{-a/2}MZ.
\end{equation}
The same estimate holds for every consecutive \(t\)-interval of length
\(U\ge M/\Omega^4\), uniformly in its starting point.
\end{lemma}

\begin{proof}
Split \(t\) into its \(q\) residue classes.  For a representative
\(b\in\{0,\ldots,q-1\}\), there is the exact identity
\begin{equation}\label{eq:exact-t-residue-identity}
 \sum_{\substack{1\le t\le M\\t\equiv b\, (q)}}\Li(y+ht)
 =S_{\Li}(y+hb-hq,y+hb;hM,hq)+O(1).
\end{equation}
The shifted interval differs from the desired residue progression by at most
two endpoint terms.  After summing over starts and residue classes this costs
\(O(qZ+qhM)\).  Apply Lemma~\ref{lem:endpoint-residue} with
\(Q=hq\).  Group the residues by \(d=(c,Q)\), divide by \(d\), and use
complete multiplicativity.  The reduced data are
\begin{equation}\label{eq:gcd-reduced-kmt-data}
 Q_d=Q/d,\qquad X_d=Z/d,\qquad L_d=hM/d,\qquad
 \frac{L_d}{Q_d}=\frac Mq.
\end{equation}

We verify the KMT range, including the stated uniformity for partial
lengths \(U\ge M/\Omega^4\).  One has
\begin{equation}\label{eq:kmt-lower-check}
 \log(U/q)\ge R^{1-\varepsilon_b}-O(\log R),
\end{equation}
so the first strict inequality in \eqref{eq:super-parameter-region}
implies
\(
 (\log(U/q))^{-1/200}\le R^{-a}.
\)
Moreover
\begin{align}
 \log(X_d^{\epsilon_K^{200}})
   &\ge R^{\gamma-200a+o(1)}>R^\rho+O(\log R),
   \label{eq:kmt-good-check}\\
 \log y_{\rm typ},\qquad y_{\rm typ}:=(U/q)^{\epsilon_K^2},
   &\ge R^{1-\varepsilon_b-2a+o(1)}.
   \label{eq:kmt-typical-check}
\end{align}
The exponent in \eqref{eq:kmt-typical-check} is positive by
\(a<(1-\varepsilon_b)/200\).  Hence
\[
 y_{\rm typ}=\exp\{R^{1-\varepsilon_b-2a+o(1)}\}
 \gg R^\rho+\log R\gg\log(hq).
\]
The first inequality is exactly the size condition
\(y_{\rm typ}\ge1000\log(hq)\) in KMT's automatic-typicality criterion;
it is the value of \(y_{\rm typ}\), not its logarithm, that is compared
with \(\log(hq)\).  KMT's Lemma~9.1 therefore implies that every divisor
\(Q_d\le hq\) is \(y_{\rm typ}\)-typical.  Together with
\eqref{eq:kmt-good-check}, this places \(Q_d\) in both the automatic
good-modulus and typical ranges.  In particular,
\[
 Q_d\le hq\le e^{R^\rho}R^2\le (X')^{\epsilon_K^{200}}
\]
for the corresponding translated base \(X'\), and every primitive
character inducing a character modulo \(Q_d\) has conductor at most this
same bound.

Apply \eqref{eq:kmt-specialized-variance}, then Cauchy--Schwarz in \(u\)
and the reduced residue classes.  After the change \(x=du\), the integration
interval is \([Z/d-Q_d,2Z/d]\).  Since
\[
 \frac{hq}{Z},\ \frac{hM}{Z}
 \le \exp(R^\rho+2R-R^\gamma)=o(1),
\]
it is contained in
\([Z/(2d),2Z/d]\) for large \(R\), which is the union of two intervals of
the form \([X,2X]\).  Applying KMT on both pieces changes only the absolute
constant, and gives
\begin{equation}\label{eq:kmt-l1-reduced}
 d\int_{Z/d-Q_d}^{2Z/d}\sum_{c\bmod Q_d}^{*}
 |S_{\Li}(u,c;L_d,Q_d)|\,\dd u
 \ll R^{-a/2}\varphi(Q_d)\frac{ZM}{q}.
\end{equation}
Lemma~\ref{lem:endpoint-residue} contributes the outer factor \(1/Q\).
Summing \(d\mid Q\) and using
\(
 \sum_{d\mid Q}\varphi(Q/d)=Q
\)
shows that one \(b\pmod q\) costs
\begin{equation}\label{eq:one-b-major}
 \ll R^{-a/2}\frac{ZM}{q}+O(Z+hM).
\end{equation}
Summing all \(q\) classes gives
\begin{equation}\label{eq:all-b-major}
 \mathcal F_{h,M}(a_0/q;Z)
 \ll R^{-a/2}ZM+qZ+qhM.
\end{equation}
Both endpoint errors are negligible: \(q/M=o(R^{-A})\) and
\(qh/Z=o(R^{-A})\) for every fixed \(A\), by
\eqref{eq:apf-local-range}.  This proves
\eqref{eq:rational-major-conclusion}.  Replacing \(M\) by \(U\) and the
physical base \(Z\) by \(Z+hT\) gives the same argument for every partial
interval, since \eqref{eq:kmt-lower-check}--\eqref{eq:kmt-typical-check}
are unchanged.
\end{proof}

\subsection{One conductor on the whole outer scale}

Set \(\log x=R^\beta\), and consider all KMT base points generated by
\begin{equation}\label{eq:outer-kmt-family}
 \begin{gathered}
 e^{R^\gamma}\le N\le e^Rx,\qquad d_0\le e^R,\qquad Z=N/d_0,\\
 0\le T\le e^R,\qquad 0\le s<hq,\qquad d\mid hq,\qquad
 X'=\frac{Z+hT-s}{d}.
 \end{gathered}
\end{equation}
They lie in the deterministic band
\begin{equation}\label{eq:outer-kmt-band}
 \exp\{R^\gamma-R^\rho-O(R)\}\le X'\le2e^Rx.
\end{equation}
Indeed, the lower bound uses \(d\le hq\le e^{R^\rho}R^2\), while the
upper bound follows from \(\rho<\gamma<\beta\).  In particular, uniformly
over the whole family and for all sufficiently large \(R\),
\begin{equation}\label{eq:outer-kmt-log-band}
 \tfrac12R^\gamma\le\log X'\le2R^\beta.
\end{equation}
Only the lower endpoint of this band is used in the zero-window argument
below.

\begin{lemma}[Global exceptional conductor]\label{lem:global-conductor}
Under \eqref{eq:super-parameter-region}, there is at most one primitive
real character, of conductor
\begin{equation}\label{eq:q0-upper}
 q_0\le e^{R^\rho}R^2,
\end{equation}
whose main term can survive in any KMT call from
\eqref{eq:outer-kmt-family}.  If it exists, then for every fixed \(A>0\),
ineffectively for all sufficiently large \(R\),
\begin{equation}\label{eq:q0-lower}
 q_0\gg_A R^A.
\end{equation}
\end{lemma}

\begin{proof}
In the proof of KMT Corollary~1.7, a real character can survive only if its
\(L\)-function has a real zero \(\vartheta\) satisfying
\begin{equation}\label{eq:kmt-exceptional-zero-original}
 1-\vartheta\ll\frac{1}{\epsilon_K^{200}\log X'}.
\end{equation}
This is the condition \(\vartheta>1-c_0/\log(X'^{\epsilon_K^{200}})\)
in KMT's exceptional-modulus construction.  Uniformly over
\eqref{eq:outer-kmt-band}, it gives
\begin{equation}\label{eq:kmt-bad-zero-window}
 1-\vartheta\ll R^{-(\gamma-200a)}.
\end{equation}
All relevant primitive conductors are at most
\(
 \mathcal Q_R=e^{R^\rho}R^2.
\)
Indeed, every character in a KMT call is induced from a primitive character
whose conductor divides the working modulus \(Q_d\mid hq\), so its conductor
is at most \(h\Omega\le e^{R^\rho}R^2\).  The KMT exceptional-zero condition
is therefore a condition on this same primitive conductor, independent of
the choice of inducing modulus.
Because
\begin{equation}\label{eq:lp-window-inclusion}
 R^{-(\gamma-200a)}\log\mathcal Q_R
 \ll R^{-(\gamma-200a-\rho)}\longrightarrow0,
\end{equation}
the KMT window lies inside the Landau--Page interval
\(1-c_{\rm LP}/\log\mathcal Q_R<\vartheta<1\) for all sufficiently large
\(R\).  Landau--Page uniqueness then gives at most one primitive real
character in the whole family.  Thus the entire window
\eqref{eq:kmt-bad-zero-window} lies inside one
Landau--Page uniqueness window.  Hence at most one primitive real
character works for all calls, including distinct parent intervals and
subset-rescaled bases.

If its exceptional real zero is \(\vartheta_0\), then
\(
 1-\vartheta_0\ll R^{-(\gamma-200a)}.
\)
Siegel's theorem gives
\(
 1-\vartheta_0\gg_\delta q_0^{-\delta}
\)
for every fixed \(\delta>0\).  Choosing \(\delta\) after an arbitrary
fixed \(A\) gives \eqref{eq:q0-lower}.  The resulting threshold is
ineffective.
\end{proof}

If the character of Lemma~\ref{lem:global-conductor} exists, define
\begin{equation}\label{eq:exceptional-shift-set}
 \cE_R=
 \{1\le h\le H_R:q_0\mid hq\text{ for some }1\le q\le\Omega\};
\end{equation}
otherwise put \(\cE_R=\varnothing\); in that branch all subsequent
exceptional-set assertions are immediate.  When \(q_0\) exists,
equivalently,
\begin{equation}\label{eq:exceptional-shift-structure}
 \cE_R=\left\{1\le h\le H_R:
              \frac{q_0}{(q_0,h)}\le\Omega\right\}.
\end{equation}
The elementary gcd sum
\begin{equation}\label{eq:gcd-sum}
 \sum_{q\le\Omega}(q_0,q)
 \le\Omega\tau(q_0)
\end{equation}
and \eqref{eq:q0-lower} imply, for every fixed \(B>0\),
\begin{equation}\label{eq:exceptional-shift-count}
 |\cE_R\cap[1,H']|
 \le H'\Omega\frac{\tau(q_0)}{q_0}
 \ll_B H'R^{-B}
 \qquad(1\le H'\le H_R).
\end{equation}
Here we used the standard bound \(\tau(q_0)=q_0^{o(1)}\) together with
the ineffective estimate \(q_0\gg_A R^A\), choosing \(A\) after the desired
power \(B\).  The same calculation is valid for every prefix \(H'\), which
is the source of the prefix-sparse form of the exceptional set.
There is at most one prime in \(\cE_R\).  Indeed, if two distinct primes
\(p_1,p_2\) belonged to it, then for large \(R\), when \(q_0>\Omega\),
\eqref{eq:exceptional-shift-structure} would give
\(
 p_i\mid q_0, q_0/p_i\le\Omega.
\)
Thus \(q_0\ge p_1p_2\ge q_0^2/\Omega^2\), so
\(q_0\le\Omega^2=R^4\), contradicting \eqref{eq:q0-lower} with \(A>4\).

\subsection{Nearby frequencies and the APF theorem}

\begin{lemma}[Blockwise Abel summation]\label{lem:blockwise-abel}
Suppose every consecutive \(t\)-interval of length
\(U\ge U_0=M/\Omega^4\) satisfies the rational-frequency estimate
\begin{equation}\label{eq:abel-input}
 \sum_{Z<y\le2Z}
 \left|\sum_{T<t\le T+U}\Li(y+ht)e(a_0t/q)\right|
 \le\delta ZU,\qquad 0\le\delta\le1.
\end{equation}
If
\(
 \alpha=a_0/q+\theta,\qquad |\theta|\le\Omega/(Mq),
\)
then
\begin{equation}\label{eq:abel-output}
 \mathcal F_{h,M}(\alpha;Z)
 \ll(\delta+\Omega^{-3})ZM.
\end{equation}
\end{lemma}

\begin{proof}
Take \(B=M\) when \(\theta=0\), and otherwise
\(
 B=\min\{M,\lfloor|\theta|^{-1}\rfloor\}.
\)
Partition \([1,M]\) into blocks of length \(B\).  On a block
\((T,T+L]\), discrete Abel summation gives
\begin{align}\label{eq:discrete-abel}
 \sum_{u=1}^{L}a_y(u)e(\theta(T+u))
 &=A_y(L)e(\theta(T+L))\\
 &\quad+\sum_{v=1}^{L-1}A_y(v)
   \{e(\theta(T+v))-e(\theta(T+v+1))\},\nonumber
\end{align}
where \(A_y(v)=\sum_{u\le v}a_y(u)\).  Use
\eqref{eq:abel-input} for \(v\ge U_0\), the trivial \(Zv\) bound below
\(U_0\), and

\(
 |e(\theta(t+1))-e(\theta t)|\le2\pi|\theta|.
\)
Since \(|\theta|B\le1\), the long portions sum to \(O(\delta ZM)\).
There are \(O(1+M|\theta|)=O(\Omega)\) blocks.  Since each short
prefix has length at most \(U_0=M/\Omega^4\), their total trivial
contribution is \(O(ZM\Omega^{-3})\).  This proves
\eqref{eq:abel-output}.
\end{proof}

\begin{theorem}[Outer-uniform progression Fourier estimate]
\label{thm:outer-apf}
Assume \eqref{eq:super-parameter-region} and set
\(\log x=R^\beta\).  For every fixed \(B>0\), all sufficiently large
\(R\) admit one set \(\cE_R\subseteq[1,H_R]\) satisfying
\eqref{eq:exceptional-shift-count} such that, simultaneously for
\begin{equation}\label{eq:outer-apf-uniformity}
 \begin{gathered}
 h\in[1,H_R]\setminus\cE_R,\qquad
 e^{R^\gamma-R}\le Z\le e^Rx,\\
 e^{R^{1-\varepsilon_b}}\le M\le e^R,\qquad
 \alpha\in\T,
 \end{gathered}
\end{equation}
one has
\begin{equation}\label{eq:outer-apf-conclusion}
 \mathcal F_{h,M}(\alpha;Z)\ll R^{-a/5}MZ.
\end{equation}
If the exceptional character does not exist, then \(\cE_R=\varnothing\).
\end{theorem}

\begin{proof}
Dirichlet approximation at denominator \(M/\Omega\) gives either the
minor alternative \eqref{eq:apf-minor-approximant}, handled by
Lemma~\ref{lem:apf-minor}, or a major arc \(q\le\Omega\).  For
\(h\notin\cE_R\), Lemmas~\ref{lem:kmt-liouville} and
\ref{lem:global-conductor} delete every character main term, uniformly
over all rational partial intervals.  Lemma~\ref{lem:apf-rational-major}
then supplies \eqref{eq:abel-input} with
\(\delta\ll R^{-a/2}\), and Lemma~\ref{lem:blockwise-abel} gives
\begin{equation}\label{eq:major-apf-general-a}
 \mathcal F_{h,M}(\alpha;Z)
 \ll(R^{-a/2}+R^{-6})MZ.
\end{equation}
Because \eqref{eq:super-parameter-region} implies \(a<1/200\), both
\eqref{eq:major-apf-general-a} and the minor bound may be weakened to
\eqref{eq:outer-apf-conclusion}.  The complement
\eqref{eq:apf-typical-complement} is smaller.  The conductor in
Lemma~\ref{lem:global-conductor} was selected before every variable in
\eqref{eq:outer-apf-uniformity}, so the same \(\cE_R\) works throughout.
\end{proof}

\section{Transfer to logarithmic Chowla correlations}
\label{sec:superpolylog-transfer}

We now combine Theorem~\ref{thm:outer-apf} with a fresh copy of the
Pilatte block architecture.  The essential difference from
Sections~\ref{sec:uncentring}--\ref{sec:main-proof} is that quotient starts
are translated before Fourier inversion.  This removes the algebraic factor
\(h\) which would otherwise make a prescribed super-polylogarithmic range
inaccessible.

\subsection{A second block system and coprime pruning}

The block parameter used from now on is \(\varepsilon_b\), independent of
the absolute parameter \(\eps_1\) fixed in Section~\ref{sec:inputs}.  For
this fixed \(\varepsilon_b>0\), re-run Pilatte's construction with
\begin{equation}\label{eq:super-block-scales}
 H_0=\exp(R^{1-\varepsilon_b}),\qquad
 J=\lfloor\varepsilon_b^2\log R\rfloor.
\end{equation}
Write \(\cP_1,\ldots,\cP_J\), \(V_j\), \(V\), \(\cD\), and \(\WP\) for
this second system.  Thus the analogues of
\eqref{eq:block-geometry}--\eqref{eq:block-masses} hold with
\(\eps_1\) replaced by \(\varepsilon_b\), and Pilatte's fourth-moment
estimate is available for its divisor blocks.  In particular,
\begin{equation}\label{eq:super-block-mass}
 \WP\asymp_{\varepsilon_b}V^J,
 \qquad
 V^J\asymp_{\varepsilon_b}R^{c_0(\varepsilon_b)}.
\end{equation}
The point of using a second system is that \(\varepsilon_b\) may be chosen
after the exponent \(\eta\) in Theorem~\ref{thm:super-main} has been fixed.
The following records explicitly that the centred input may be re-run for
this system.

\begin{lemma}[Centred interface for the second block system]
\label{lem:second-system-centered-interface}
For every sufficiently small fixed \(\varepsilon_b>0\),
Input~\ref{input:tt-decoupling} remains valid, with constants allowed to
depend on \(\varepsilon_b\), after its original Pilatte system is replaced
by the second system in \eqref{eq:super-block-scales}.
\end{lemma}

\begin{proof}
Apply \citet[Theorem~3.3]{taoteravainen2026quantitative} with its fixed block
parameter equal to \(\varepsilon_b\).  That theorem permits every
sufficiently small fixed value of this parameter, every nonempty selection
of the resulting blocks, every interval of the stated length, arbitrary
residue assignments, and arbitrary \(1\)-bounded functions.  Pilatte's
construction at \(\varepsilon_b\) is exactly the system required by that
application, with the separation, disjointness, location in \((H_0,H)\),
and harmonic-mass bounds recorded above.  The asserted uniformity follows.
\end{proof}

\begin{lemma}[Coprime Pilatte pruning]\label{lem:coprime-block-pruning}
Fix \(B_0>0\).  Uniformly for \(1\le h\le e^{R^{B_0}}\), define
\[
 \cP_j(h)=\{p\in\cP_j:p\nmid h\},\qquad
 V_j(h)=\sum_{p\in\cP_j(h)}\frac1p,
\]
and form \(\cD(h)\), \(\cD_I(h;M)\), \(V_I(h;M)\), and
\(Q_{I,h,M}\) from these pruned blocks in the evident way.  Put
\(
 W_{\mathrm{Pil},h}=\prod_{j=1}^J V_j(h).
\)
Then for every fixed \(A>0\), uniformly in such \(h\),
\begin{equation}\label{eq:pruned-mass-loss}
 \max_{j\le J}|V_j(h)-V_j|\ll_{A,\varepsilon_b}R^{-A},
 \qquad
 W_{\mathrm{Pil},h}=(1+O_{A,\varepsilon_b}(R^{-A}))\WP.
\end{equation}
In particular \(W_{\mathrm{Pil},h}\asymp_{\varepsilon_b}V^J\).  Moreover, for every
nonempty \(I\subseteq[J]\) and dyadic \(M\),
\begin{equation}\label{eq:pruned-fourth-moment}
 \int_0^1|Q_{I,h,M}(\alpha)|^4\,\dd\alpha
 \ll_{\varepsilon_b}\frac{V^{4J}}{M(\log M)^4},
\end{equation}
and
\begin{equation}\label{eq:pruned-block-mass-sum}
 \sum_M V_I(h;M)\le\prod_{i\in I}V_i(h)\le V^{|I|}.
\end{equation}
\end{lemma}

\begin{proof}
Every block prime satisfies \(p>H_0\).  Hence
\begin{align*}
 0\le V_j-V_j(h)
 &\le\sum_{\substack{p\mid h\\p>H_0}}\frac1p
 \le \frac1{H_0\log H_0}\sum_{p\mid h}\log p\\
 &\le \frac{\log h}{H_0\log H_0}
 \le \frac{R^{B_0}}{\exp(R^{1-\varepsilon_b})R^{1-\varepsilon_b}}
 \ll_{A,\varepsilon_b}R^{-A}.
\end{align*}
The block-mass asymptotic gives \(V_j\asymp_{\varepsilon_b}1\).  Apply
the first estimate with exponent \(A+2\); since
\(J=O_{\varepsilon_b}(\log R)\), multiplying the factors then proves the
second claim in \eqref{eq:pruned-mass-loss} with the stated exponent \(A\).

For the fourth moment, orthogonality gives the exact positive expansion
\[
 \int_0^1|Q_{I,h,M}(\alpha)|^4\,\dd\alpha
 =\sum_{\substack{d_1+d_2=d_3+d_4\\
                   d_i\in\cD_I(h;M)}}
   \frac1{d_1d_2d_3d_4}.
\]
Since \(\cD_I(h;M)\subseteq\cD_I(M)\), this is at most the corresponding
unpruned fourth moment, proving \eqref{eq:pruned-fourth-moment}.
Equation~\eqref{eq:pruned-block-mass-sum} is immediate from disjointness of
the dyadic blocks.
\end{proof}

\subsection{A parameter-stable transfer theorem}

\begin{proposition}[APF-to-Chowla transfer]\label{prop:apf-transfer}
Fix
\begin{equation}\label{eq:transfer-fixed-parameters}
 \sigma>0,\qquad 0<\rho<5,\qquad 2.01<\gamma<\beta,
 \qquad \rho<\gamma.
\end{equation}
Choose a fixed \(\varepsilon_b>0\) and its second Pilatte block system so
that
\begin{equation}\label{eq:transfer-subset-hypotheses}
 2^JV^{2J}\le R^{\sigma/20}
\end{equation}
for all sufficiently large \(R\).  Put
\(
 \log x=R^\beta, H_R=\lfloor e^{R^\rho}\rfloor.
\)
Suppose one set \(\cE_R\subseteq[1,H_R]\) satisfies, for every fixed
\(B>0\) and every \(1\le H'\le H_R\),
\begin{equation}\label{eq:transfer-exception-hypothesis}
 |\cE_R\cap[1,H']|\ll_B H'R^{-B},
\end{equation}
and suppose the conclusion of Theorem~\ref{thm:outer-apf}, with exponent
\(\sigma\), holds outside this same set for every parent interval,
subset-rescaled base, divisor block, frequency, and prefix endpoint used in
the logarithmic dilation argument.  Then there is \(c_*>0\) such that
\begin{equation}\label{eq:transfer-final}
 \max_{\substack{1\le h\le
  \exp((\log x)^{\rho/\beta})\\h\notin\cE_R}}
 \left|\sum_{n\le x}\frac{\Li(n)\Li(n+h)}n\right|
 \ll(\log x)^{1-c_*}.
\end{equation}
One may take any fixed
\begin{equation}\label{eq:transfer-saving}
 0<c_*\le
 \min\left\{1-\frac\gamma\beta,
             \frac\sigma{6\beta},
             \frac{c_{\rm dec}}\beta\right\},
\end{equation}
where \(c_{\rm dec}=c_{\rm dec}(\varepsilon_b)>0\) is the centred-decoupling
saving from the second Pilatte block system.
\end{proposition}

\begin{proof}
Fix \(h\in[1,H_R]\setminus\cE_R\) and use the pruned family from
Lemma~\ref{lem:coprime-block-pruning}.  In this proof we suppress the
\(h\)-dependence from \(\cD_I(h;M)\), \(V_I(h;M)\), and
\(Q_{I,h,M}\).

Fix a nonempty subset \(I\subseteq[J]\) and one dyadic divisor block
\(\cD_I(M)\).  Write \(n=r+hq\), \(0\le r<h\), and let \(I_r\) be the
induced interval of quotient starts.  Averaging translations in \(q\)
gives
\begin{multline}\label{eq:quotient-translation}
 T_{M,h}=\frac1M\sum_{t\le M}\sum_{r\bmod h}\sum_{q\in I_r}
 \sum_{d\in\cD_I(M)}
 \frac{\Li(r+h(q+t))\Li(r+h(q+t+d))}{d}\\
 +O(hMV_I[M]).
\end{multline}
Define
\begin{align}\label{eq:quotient-fourier-factors}
 F_{r,q}(\alpha)&=\sum_{t\le M}\Li(r+h(q+t))e(\alpha t),\nonumber\\
 G_{r,q}(\alpha)&=\sum_{k\le2M}\Li(r+h(q+k))e(-\alpha k),\\
 Q_{I,M}(\alpha)&=\sum_{d\in\cD_I(M)}e(\alpha d)/d.\nonumber
\end{align}
Fourier orthogonality in \(t+d=k\) turns the main term in
\eqref{eq:quotient-translation} into
\begin{equation}\label{eq:quotient-fourier-identity}
 \frac1M\sum_{r\bmod h}\sum_{q\in I_r}
 \int_\T Q_{I,M}(\alpha)F_{r,q}(\alpha)G_{r,q}(\alpha)\,\dd\alpha.
\end{equation}
Here \(\|G_{r,q}\|_\infty\le2M\), with no factor \(h\).

The frequency transfer has a direct fourth-root form.  Put
\[
 \mathcal S(\alpha)=\sum_{r\bmod h}\sum_{q\in I_r}
 |F_{r,q}(\alpha)G_{r,q}(\alpha)|.
\]
The APF hypothesis and \(\lVert G_{r,q}\rVert_\infty\le2M\) give
\begin{equation}\label{eq:apf-product-linfty}
 \lVert\mathcal S\rVert_\infty\ll R^{-\sigma}M^2N.
\end{equation}
For each start, Parseval and Cauchy--Schwarz give
\(
 \int_\T|F_{r,q}G_{r,q}|\ll M.
\)
There are \(O(N)\) quotient starts in total, and hence
\begin{equation}\label{eq:apf-product-lone}
 \lVert\mathcal S\rVert_1\ll MN.
\end{equation}
Interpolation between \eqref{eq:apf-product-linfty} and
\eqref{eq:apf-product-lone} yields
\[
 \lVert\mathcal S\rVert_{4/3}
 \ll R^{-\sigma/4}M^{5/4}N.
\]
Hölder's inequality and \eqref{eq:pruned-fourth-moment} now bound
\eqref{eq:quotient-fourier-identity} directly, without an exceptional
frequency split, and give
\begin{equation}\label{eq:apf-one-divisor-block}
 |T_{M,h}|
 \ll NR^{-\sigma/4}\frac{V^J}{\log M}
 +O(hMV_I[M]).
\end{equation}
The boundary term is \(o(NV_I[M]R^{-A})\) for every fixed \(A\), since
\(hM\le e^{R^\rho+R}\), every relevant local base has logarithm at least
\(R^\gamma-R\), and \(\gamma>\rho\).

Since the divisor blocks are dyadic and
\(e^{R^{1-\varepsilon_b}}\le M\le e^R\),
\begin{equation}\label{eq:apf-dyadic-log-sum}
 \sum_M\frac1{\log M}\ll_{\varepsilon_b}\log R.
\end{equation}
Consequently, after summing \(M\) and then the nonempty subset expansion,
\eqref{eq:transfer-subset-hypotheses} and
\(\log R\le R^{\sigma/40}\) give
\begin{equation}\label{eq:apf-uncentring-final}
 |S_{1,h}-S_{2,h}|\ll NR^{-\sigma/6}.
\end{equation}
Indeed,
\(
 -\sigma/4+\sigma/20+\sigma/40=-7\sigma/40<-\sigma/6.
\)
The subset expansion is unchanged: every pruned harmonic mass is at most
its unpruned counterpart, so its total cost is still bounded by
\(2^JV^{2J}\).

Every retained block prime is coprime to \(h\).  Thus, after writing
\(n=r+hq\), its divisibility condition is a genuine residue condition
\[
 p\mid r+hq\quad\Longleftrightarrow\quad
 q\equiv-rh^{-1}\pmod p.
\]
To apply the centred decoupling theorem, assign arbitrary residues to the
finitely many deleted primes.  The restricted outer sum over pruned prime
tuples is bounded by the full nonnegative outer sum in
Lemma~\ref{lem:second-system-centered-interface}.  Hence
\begin{equation}\label{eq:super-centered-bound}
 |S_{2,h}|\ll V^{0.51J}N.
\end{equation}
The corresponding quotient intervals satisfy
\begin{equation}\label{eq:super-centered-length}
 \log L_r\ge R^\gamma-R-R^\rho+O(1)>R^{2.01},
\end{equation}
again because \(\gamma>\rho\) and \(\gamma>2.01\).

It remains to normalize the dilation.  Complete multiplicativity gives,
for the pruned divisor family,
\begin{equation}\label{eq:pruned-logarithmic-dilation}
 \begin{aligned}
 W_{\mathrm{Pil},h} C_h(x)
 &=\sum_{d\in\cD(h)}\sum_{\substack{m\le dx\\d\mid m}}
   \frac{\Li(m)\Li(m+hd)}m,\\
 C_h(x)&:=\sum_{n\le x}\frac{\Li(n)\Li(n+h)}n.
 \end{aligned}
\end{equation}
The tail \(x<m\le dx\) costs \(O(W_{\mathrm{Pil},h}R)\), exactly as in
\eqref{eq:dilation-tail}.  Decompose every remaining prefix into intervals
of multiplicative width at most two.  Intervals below \(e^{R^\gamma}\) are
treated trivially; all larger intervals use
\eqref{eq:apf-uncentring-final}--\eqref{eq:super-centered-bound}.  The same
\(\cE_R\) was fixed before all prefix endpoints, so no exceptional-set
union is introduced.  Partial summation, division by
\(W_{\mathrm{Pil},h}\asymp_{\varepsilon_b}V^J\), and
Lemma~\ref{lem:coprime-block-pruning} give
\begin{equation}\label{eq:super-pre-final}
 |C_h(x)|
 \ll R^\gamma+R+
 \left(R^{-\sigma/6}+\frac{V^{0.51J}}{W_{\mathrm{Pil},h}}\right)\log x.
\end{equation}
By \eqref{eq:super-block-mass},
\(
 V^{0.51J}/W_{\mathrm{Pil},h}\ll R^{-c_{\rm dec}}
\)
for some fixed \(c_{\rm dec}(\varepsilon_b)>0\).  Since
\(\log x=R^\beta\), equations \eqref{eq:super-pre-final} and
\eqref{eq:transfer-saving} prove \eqref{eq:transfer-final}.
\end{proof}

\begin{corollary}[Maximal terminal-point transfer]
\label{cor:maximal-apf-transfer}
Under the hypotheses of Proposition~\ref{prop:apf-transfer}, the same
exceptional set satisfies
\[
 \max_{\substack{1\le h\le
  \exp((\log x)^{\rho/\beta})\\h\notin\cE_R}}
 \sup_{1\le y\le x}
 \left|\sum_{n\le y}\frac{\Li(n)\Li(n+h)}n\right|
 \ll (\log x)^{1-c_*}
\]
for every fixed \(c_*\) allowed by
\eqref{eq:transfer-saving}.
\end{corollary}

\begin{proof}
The exceptional set in Proposition~\ref{prop:apf-transfer} is selected
before any parent interval or prefix endpoint is used.  For a terminal point
\(y\ge e^{R^\gamma}\), repeat its dilation proof with \(y\) in place of
\(x\).  Every local base and every translated base then lies in the range
\(e^{R^\gamma-R}\le Z\le e^Rx\) of
Theorem~\ref{thm:outer-apf}, and the same partial-summation estimates only
decrease when the terminal point is shortened.  Thus
\[
 \left|\sum_{n\le y}\frac{\Li(n)\Li(n+h)}n\right|
 \ll R^\gamma+R+R^{-\sigma/6}\log x
       +R^{-c_{\rm dec}}\log x.
\]
If \(y<e^{R^\gamma}\), the trivial harmonic bound is
\(O(1+R^\gamma)\).  Since \(c_*\le1-\gamma/\beta\), both bounds are
\(O((\log x)^{1-c_*})\), uniformly in \(y\).  The prefix estimate for
\(\cE_R\) is inherited from \eqref{eq:exceptional-shift-count}.
\end{proof}

\subsection{The optimized exponent family}

We now prove the principal new result stated in the introduction.

\begin{proof}[Proof of Theorem~\ref{thm:super-main}]
Fix
\begin{equation}\label{eq:eta-open-range}
 0<\eta<1.
\end{equation}
Choose
\begin{equation}\label{eq:eta-parameter-construction}
 \beta=5,\qquad
 \rho=5\eta,\qquad
 \gamma=\frac52(1+\eta),\qquad
 a=\frac{1-\eta}{1000}.
\end{equation}
Then
\begin{equation}\label{eq:rho-identity}
 0<\rho<5,\qquad
 \frac\rho\beta=\eta,\qquad
 \beta-\gamma=\gamma-\rho=\frac52(1-\eta)>0,
\end{equation}
and \(\gamma>5/2>2.01\).  Moreover
\begin{equation}\label{eq:good-parameter-discharge}
 \gamma-200a-\rho
 =\frac{23}{10}(1-\eta)>0.
\end{equation}

After \(a\) has been fixed, choose the independent Pilatte parameter
\(\varepsilon_b>0\) sufficiently small that
\begin{equation}\label{eq:epsb-choice}
 \varepsilon_b<\frac1{100},\qquad
 2^JV^{2J}\le R^{a/100}
\end{equation}
for all sufficiently large \(R\).  This is possible because the exponent
of \(R\) in \(2^JV^{2J}\) tends to zero with \(\varepsilon_b\), exactly as
in the argument following \eqref{eq:subset-loss}.  Since \(a<1/1000\),
we also have
\[
 a<\frac{1-\varepsilon_b}{200}.
\]
Thus every condition in \eqref{eq:super-parameter-region} holds.  Notice
that no comparison between \(\rho\) and \(1-\varepsilon_b\) is required:
KMT typicality is supplied by Lemma~9.1 of KMT as used in
\eqref{eq:kmt-typical-check}, while block primes dividing \(h\) are removed
by Lemma~\ref{lem:coprime-block-pruning}.

Theorem~\ref{thm:outer-apf} applies with
\(
 \sigma=a/5=(1-\eta)/5000.
\)
Moreover \eqref{eq:epsb-choice} is exactly the subset condition
\(2^JV^{2J}\le R^{\sigma/20}\) in
Proposition~\ref{prop:apf-transfer}.  Corollary~\ref{cor:maximal-apf-transfer}
therefore gives some fixed \(c_\eta>0\).  For example, one may take any positive value
below
\begin{equation}\label{eq:explicit-c-eta}
 \min\left\{
   \frac{1-\eta}{2},
   \frac{1-\eta}{150000},
   \frac{c_{\rm dec}(\varepsilon_b)}5
 \right\}.
\end{equation}
By \eqref{eq:rho-identity}, the local shift cap becomes
\begin{equation}\label{eq:eta-shift-conversion}
 e^{R^\rho}=\exp((\log x)^{\rho/\beta})
 =\exp((\log x)^\eta).
\end{equation}
For a prescribed \(C>0\), choose \(B=5C\) in
\eqref{eq:exceptional-shift-count}; then, uniformly for
\(1\le H'\le\exp((\log x)^\eta)\),
\begin{equation}\label{eq:eta-exception-conversion}
 |\cE_R\cap[1,H']|\ll_{\eta,C}
 H'(\log x)^{-C}.
\end{equation}
The prime assertion follows from the argument after
\eqref{eq:exceptional-shift-count}.  If no exceptional character exists,
the construction gives \(\cE_R=\varnothing\), proving the conditional
all-shifts assertion as well.
\end{proof}

The endpoint \(\eta=1\) is not included.  In the present power-scale
architecture the final exponent is \(\eta=\rho/\beta\), while the transfer
requires the strict chain \(\rho<\gamma<\beta\).  Thus \(\eta=1\) would
force \(\rho=\beta\) and destroy the scale gap.  The endpoint is a genuine
scale-separation boundary of this argument.

\section{Maximal moments in long shift windows}\label{sec:maximal-moments}

The pointwise progression-of-starts argument stops at the strict endpoint
$\eta=1$.  At that endpoint a different, much simpler mechanism becomes
available.  The maximal truncations of the harmonic Liouville convolution
operators form a uniformly bounded family on every $\ell^p(\mathbb Z)$,
$1<p<\infty$.  This gives all fixed moments over shift windows of length at
least the summation range.  Localizing the high physical blocks and using a
published almost-all Fourier theorem extends the same all-moment conclusion
to every $H\ge x^\theta$ with $\theta>1/3$.  For still shorter polynomial
shift windows we combine the multiplier bound below the shift scale with
Menon's averaged Chowla theorem above it.

We use the classical Davenport estimate in the Liouville form
\begin{equation}\label{eq:davenport-liouville}
 \sup_{\alpha\in\mathbb R/\mathbb Z}
 \left|\sum_{n\le Y}\Li(n)e(\alpha n)\right|
 \ll_A \frac{Y}{(\log(2Y))^A}
 \qquad(A>0),
\end{equation}
uniformly for $Y\ge2$.  For completeness, this follows directly from
Davenport's M\"obius estimate \citep{davenport1937infiniteII}.  Indeed,
using
\[
 \Li(n)=\sum_{d^2\mid n}\mu(n/d^2),
\]
the left side before taking the supremum becomes
\[
 \sum_{d\le\sqrt Y}\sum_{m\le Y/d^2}\mu(m)e(\alpha d^2m).
\]
For $d\le Y^{1/4}$, Davenport's bound applies uniformly to the inner sum,
and its length is at least $Y^{1/2}$; after increasing the logarithmic
exponent and summing $d^{-2}$ this contributes
$O_A(Y(\log(2Y))^{-A})$.  The complementary range is bounded trivially by
\[
 \sum_{d>Y^{1/4}}\frac{Y}{d^2}\ll Y^{3/4},
\]
which is smaller than the same target for sufficiently large $Y$; bounded
$Y$ is absorbed into the constant.  This proves
\eqref{eq:davenport-liouville}.  The associated harmonic Liouville series
also appears in \citet{bateman1963trigonometrical}.

\begin{lemma}[Dyadic harmonic multiplier decay]
\label{lem:dyadic-harmonic-multiplier}
Let $N\ge2$ and let $I\subseteq(N,2N]$ be an integer interval.  Then for
every fixed $A>0$,
\begin{equation}\label{eq:dyadic-harmonic-fourier}
 \sup_{\alpha\in\mathbb R/\mathbb Z}
 \left|\sum_{n\in I}\frac{\Li(n)}n e(\alpha n)\right|
 \ll_A (\log(2N))^{-A}.
\end{equation}
\end{lemma}

\begin{proof}
Write $S(t,\alpha)=\sum_{n\le t}\Li(n)e(\alpha n)$.  Partial summation on
$I=(a,b]\subseteq(N,2N]$ gives
\[
 \sum_{a<n\le b}\frac{\Li(n)}n e(\alpha n)
 =\frac{S(b,\alpha)}b-\frac{S(a,\alpha)}a
   +\int_a^b\frac{S(t,\alpha)}{t^2}\,dt.
\]
Insert \eqref{eq:davenport-liouville}; throughout the integral
$\log(2t)\asymp\log(2N)$.  This proves
\eqref{eq:dyadic-harmonic-fourier}.
\end{proof}

For $X\ge1$ define the one-sided convolution operator
\begin{equation}\label{eq:harmonic-convolution-operator}
 (T_Xf)(h):=\sum_{1\le n\le X}\frac{\Li(n)}n f(h+n),
 \qquad f\in\ell^p(\mathbb Z).
\end{equation}
Its maximal truncation is
\begin{equation}\label{eq:maximal-harmonic-convolution-operator}
 (T_X^*f)(h):=\sup_{1\le Y\le X}|(T_Yf)(h)|.
\end{equation}

\begin{lemma}[Uniform maximal $\ell^p$ harmonic Liouville multiplier]
\label{lem:uniform-lp-multiplier}
For every fixed $1<p<\infty$ there is $M_p<\infty$ such that
\begin{equation}\label{eq:uniform-lp-operator}
 \sup_{X\ge1}\|T_X^*\|_{\ell^p(\mathbb Z)\to\ell^p(\mathbb Z)}\le M_p.
\end{equation}
\end{lemma}

\begin{proof}
For $j\ge0$ put
\[
 \mathcal M_jf(h):=
 \sup_{2^j<Y\le2^{j+1}}
 \left|\sum_{2^j<n\le Y}\frac{\Li(n)}n f(h+n)\right|.
\]
We first prove rapid decay of $\mathcal M_j$ on $\ell^2$.  Partition the
integer shell $(2^j,2^{j+1}]$ into binary cells.  At level $\ell$ there
are $O(2^\ell)$ cells, each of length $O(2^{j-\ell})$, and every initial
subinterval of the shell is a disjoint union of at most one cell at each
level.  If $B$ is a level-$\ell$ cell and $T_B$ is convolution with
$\Li(n)n^{-1}1_B(n)$, then Plancherel, Lemma
\ref{lem:dyadic-harmonic-multiplier} with exponent $2A$, and the trivial
harmonic mass of $B$ give
\begin{equation}\label{eq:binary-cell-l2}
 \|T_B\|_{2\to2}
 \ll_A\min\{(1+j)^{-2A},2^{-\ell}\}.
\end{equation}
The binary decomposition and Cauchy--Schwarz over the cells at each level
therefore imply
\begin{align}
 \|\mathcal M_jf\|_2
 &\ll\sum_{\ell=0}^{j+1}
       \left(\sum_{B\text{ at level }\ell}\|T_Bf\|_2^2\right)^{1/2}
 \notag\\
 &\ll_A\sum_{\ell=0}^{j+1}2^{\ell/2}
       \min\{(1+j)^{-2A},2^{-\ell}\}\,\|f\|_2
 \ll_A(1+j)^{-A}\|f\|_2.
 \label{eq:block-maximal-l2}
\end{align}
The last sum is split where $2^{-\ell}$ crosses $(1+j)^{-2A}$.

Pointwise,
\[
 \mathcal M_jf(h)
 \le\sum_{2^j<n\le2^{j+1}}\frac{|f(h+n)|}{n},
\]
so $\mathcal M_j$ has uniformly bounded strong type $(1,1)$ and
$(\infty,\infty)$.  Interpolation for sublinear operators applied to
\eqref{eq:block-maximal-l2} gives
\begin{equation}\label{eq:block-lp-op}
 \|\mathcal M_j\|_{p\to p}
 \ll_{A,p}(1+j)^{-A\vartheta_p},
 \qquad
 \vartheta_p:=\frac{2}{\max\{p,p/(p-1)\}}>0.
\end{equation}
Choose $A>1/\vartheta_p$.  Every prefix $[1,Y]$ contains complete dyadic
shells and at most one truncated final shell, whence pointwise
\[
 T_X^*f\le\sum_{j:\,2^j\le X}\mathcal M_jf+O(|f(\,\cdot+1)|).
\]
The operator norms in \eqref{eq:block-lp-op} are summable, uniformly in
$X$, and \eqref{eq:uniform-lp-operator} follows.
\end{proof}

\begin{proof}[Proof of Theorem~\ref{thm:long-shift-moments}]
Fix $L\ge0$ and an arbitrary sequence $b:\mathbb Z\to\mathbb C$ with
$|b(m)|\le1$.  Let
\[
 f(m):=b(m)1_{[L+1,L+H+x]}(m).
\]
For every integer $L<h\le L+H$, one has exactly
\[
 (T_x^*f)(h)=
 \sup_{1\le y\le x}\left|\sum_{n\le y}
 \frac{\Li(n)b(n+h)}n\right|.
\]
Lemma~\ref{lem:uniform-lp-multiplier} gives
\[
 \sum_{L<h\le L+H}
 \sup_{1\le y\le x}
 \left|\sum_{n\le y}\frac{\Li(n)b(n+h)}n\right|^p
 \le \|T_x^*f\|_{\ell^p}^p
 \ll_p \|f\|_{\ell^p}^p
 \ll_p H+x.
\]
Division by $H$ proves \eqref{eq:long-shift-moment-intro}.
Taking $b=\Li$ gives the displayed correlation theorem.
\end{proof}

\begin{corollary}[Arbitrarily sparse endpoint exceptions]
\label{cor:long-shift-tail}
Fix $B\ge0$ and $\varepsilon,C>0$.  Uniformly for
\begin{equation}\label{eq:near-endpoint-window}
 H\ge \frac{x}{(\log x)^B}
\end{equation}
and $L\ge0$, all sufficiently large $x$ admit a set
$\mathcal E_{x,L,H}\subseteq(L,L+H]\cap\mathbb Z$ such that
\begin{equation}\label{eq:long-tail-size}
 |\mathcal E_{x,L,H}|\ll_{B,\varepsilon,C}H(\log x)^{-C},
\end{equation}
and
\begin{equation}\label{eq:long-tail-bound}
 \max_{\substack{L<h\le L+H\\h\notin\mathcal E_{x,L,H}}}
 \sup_{1\le y\le x}|C_h(y)|\le(\log x)^\varepsilon.
\end{equation}
In particular, for $H=x$ this gives the endpoint $\eta=1$ with an
arbitrarily power-logarithmically sparse exceptional set.
\end{corollary}

\begin{proof}
Choose a fixed $p>1$ with $p\varepsilon>B+C$.  By
Theorem~\ref{thm:long-shift-moments} and \eqref{eq:near-endpoint-window},
\[
 \frac1H\sum_{L<h\le L+H}
 \sup_{1\le y\le x}|C_h(y)|^p\ll_p 1+(\log x)^B.
\]
Markov's inequality at height $(\log x)^\varepsilon$ gives
\[
 \frac{|\mathcal E_{x,L,H}|}{H}
 \ll_p (\log x)^{B-p\varepsilon}
 \ll (\log x)^{-C},
\]
which proves the claim.
\end{proof}

We next use a published almost-all short-interval theorem to remove the
factor $x/H$ throughout polynomial windows beyond exponent $1/3$.  We use
the maximal progression norm recorded before
Input~\ref{input:almost-all-mobius-fourier}.

\begin{lemma}[Almost-all maximal Fourier uniformity for Liouville]
\label{lem:almost-all-liouville-fourier}
Fix $\delta>0$ and $A,B>0$.  Uniformly for integers $N,K$ satisfying
\[
 N^{1/3+\delta}\le K\le N^{1-\delta},
\]
there is a set $\mathcal Y_{N,K}\subseteq[N,2N]\cap\mathbb Z$ with
\begin{equation}\label{eq:liouville-bad-start-count}
 |\mathcal Y_{N,K}|\ll_{A,B,\delta}N(\log N)^{-B}
\end{equation}
such that, for $y\notin\mathcal Y_{N,K}$,
\begin{equation}\label{eq:liouville-local-maximal-fourier}
 \sup_{\alpha\in\mathbb R/\mathbb Z}
 \left|\sum_{y<n\le y+K}\Li(n)e(\alpha n)\right|^*
 \ll_{A,B,\delta}K(\log N)^{-A}.
\end{equation}
\end{lemma}

\begin{proof}
Input~\ref{input:almost-all-mobius-fourier} gives the analogue of
\eqref{eq:liouville-local-maximal-fourier} for $\mu$, with an arbitrary
power of $\log N$ both in the bound and in the exceptional measure.
Although the source records a real-measure
exceptional set, an integer interval of integer length has constant
integer support as its real left endpoint varies in $[u,u+1)$.  Thus the
same estimate, up to harmless endpoints, counts exceptional integer
starts.

We transfer this statement from $\mu$ to $\Li$ using the exact identity
\begin{equation}\label{eq:liouville-square-divisor-identity}
 \Li(n)=\sum_{d^2\mid n}\mu(n/d^2).
\end{equation}
Choose
\[
 D=(\log N)^{A+B+10}.
\]
For a fixed $d\le D$, split the integer starts $y$ by their residue
$s\pmod {d^2}$ and write $y=d^2u+s$.  The integers $m$ with
$y<d^2m\le y+K$ then occupy an interval of one of the two integer lengths
$K/d^2+O(1)$, depending only on $s$.  Moreover, the intersection with any
arithmetic progression in the original $n$-interval maps to an arithmetic
progression in $m$, while the phase becomes $e(\alpha d^2m)$.

For $d\le D$, the powers of $\log N$ are absorbed by the fixed exponent
margin: the transformed ambient scale and length obey
\[
 (N/d^2)^{1/3+\delta/2}
 \le K/d^2+O(1)
 \le (N/d^2)^{1-\delta/2}
\]
for all sufficiently large $N$.  Indeed, uniformly for $d\le D$, the
lower- and upper-endpoint ratios are respectively bounded by
\[
 \frac{K/d^2}{(N/d^2)^{1/3+\delta/2}}
 \ge N^{\delta/2}d^{-4/3+\delta}\longrightarrow\infty,
 \qquad
 \frac{N^{1-\delta}/d^2}{(N/d^2)^{1-\delta/2}}
 =N^{-\delta/2}d^{-\delta}\longrightarrow0.
\]
Thus the $O(1)$ endpoint rounding is harmless.  Apply the published $\mu$ theorem at
this scale (covering the transformed ambient range by a bounded number of
dyadic bands), with a logarithmic exponent chosen larger than
$2A+2B+30$.  For each fixed residue $s$, the number of exceptional $u$ is
$O((N/d^2)(\log N)^{-2A-2B-30})$.  Summing over the $d^2$ residues and
then over $d\le D$ leaves $O(N(\log N)^{-B-10})$ exceptional starts.
Outside them, summing the bounds over $d\le D$ costs only
\[
 \sum_{d\le D}\left(\frac K{d^2}+1\right)
 \ll K,
\]
and gives the right side of
\eqref{eq:liouville-local-maximal-fourier}.

It remains to control $d>D$.  Put
\[
 R_D(y):=\sum_{d>D}\#\{y<n\le y+K:d^2\mid n\}.
\]
This dominates the tail of \eqref{eq:liouville-square-divisor-identity}
for every phase and every subprogression.  Interchanging $y,n,d$ gives
\[
 \sum_{N\le y\le2N}R_D(y)
 \ll K\sum_{d>D}\#\{n\le3N:d^2\mid n\}
 \ll KN\sum_{d>D}d^{-2}
 \ll \frac{KN}{D}.
\]
Markov's inequality at height $K(\log N)^{-A}$ discards at most
$O(N(\log N)^{-B-10})$ further starts.  Combining the two exceptional
sets proves the lemma.
\end{proof}

\begin{lemma}[Localized Toeplitz estimate]
\label{lem:localized-toeplitz}
Fix $\theta>1/3$, $1<p<\infty$, and $D_0>0$.  Let
$x^\theta\le H<x$, let $L_0\ge0$, and let $b:\mathbb Z\to\mathbb C$ be
$1$-bounded.  If $H\le N\le x$ and $I_N\subseteq(N,2N]$ is an integer
interval, then
\begin{equation}\label{eq:localized-toeplitz-bound}
 \left(\frac1H\sum_{L_0<h\le L_0+H}
 \sup_{J\subseteq I_N}
 \left|\sum_{n\in J}\frac{\Li(n)b(n+h)}n\right|^p\right)^{1/p}
 \ll_{p,\theta,D_0}(\log x)^{-D_0}.
\end{equation}
Here $J$ ranges over integer subintervals of $I_N$.
\end{lemma}

\begin{proof}
Choose
\[
 0<\rho<\min\left\{\frac{\theta-1/3}{4},\frac1{12}\right\},
 \qquad
 K_N=\left\lfloor\min\{H/4,N^{1-\rho}\}\right\rfloor.
\]
For all sufficiently large $x$,
\begin{equation}\label{eq:local-length-window}
 N^{1/3+\rho}\le K_N\le N^{1-\rho}.
\end{equation}
Indeed, \(H\ge x^\theta\ge N^\theta\), so the first candidate is at least
\(N^\theta/4\).  Since \(\rho<\theta-1/3\), this is at least
\(N^{1/3+\rho}\) for sufficiently large \(N\); the second candidate is
already \(N^{1-\rho}\).  This also covers the first shell \(N=H\).

Apply Lemma~\ref{lem:almost-all-liouville-fourier} at length $K_N$, with
Fourier exponent $2A_1$ and exceptional-start exponent $B_1$, to be fixed
below.  Average over the $K_N$ possible
translations of the grid of consecutive $K_N$-blocks.  There is a
translation for which the number of full blocks whose left endpoint is
bad is
\begin{equation}\label{eq:bad-local-block-count}
 \ll \frac{N}{K_N}(\log N)^{-B_1}.
\end{equation}
That grid covers $I_N$ apart from at most two boundary intervals of length
$K_N$.

For a full good block $I=(y,y+K_N]$, the star norm in
\eqref{eq:liouville-local-maximal-fourier} controls every subinterval.
Maximal partial summation therefore gives
\begin{equation}\label{eq:good-local-harmonic-multiplier}
 \sup_{J\subseteq I}\sup_\alpha
 \left|\sum_{n\in J}\frac{\Li(n)}n e(\alpha n)\right|
 \ll \frac{K_N}{N}(\log N)^{-2A_1}.
\end{equation}
Define the sublinear block operator
\[
 \mathcal M_Ib(h):=\sup_{J\subseteq I}
 \left|\sum_{n\in J}\frac{\Li(n)b(h+n)}n\right|.
\]
Run the binary-cell argument of
\eqref{eq:binary-cell-l2}--\eqref{eq:block-maximal-l2} inside $I$.
An arbitrary subinterval uses at most two cells at each binary level,
which changes only the absolute constant.
At level $\ell$, a cell has trivial harmonic mass
$O((K_N/N)2^{-\ell})$, whereas
\eqref{eq:good-local-harmonic-multiplier} gives
$O((K_N/N)(\log N)^{-2A_1})$.  Consequently
\[
 \|\mathcal M_I\|_{2\to2}
 \ll \frac{K_N}{N}(\log N)^{-A_1}.
\]
The strong $(1,1)$ and $(\infty,\infty)$ norms are $O(K_N/N)$.
Sublinear interpolation therefore gives
\begin{equation}\label{eq:good-local-lp-operator}
 \|\mathcal M_I\|_{p\to p}
 \ll_p\frac{K_N}{N}(\log N)^{-A_1\vartheta_p},
 \qquad
 \vartheta_p=\frac2{\max\{p,p/(p-1)\}}.
\end{equation}
Every bad or boundary block has the trivial maximal-operator bound
$O(K_N/N)$.
Summing \eqref{eq:good-local-lp-operator},
\eqref{eq:bad-local-block-count}, and the two boundary contributions shows
that the sum of the block operator norms is
\begin{equation}\label{eq:local-operator-norm-sum}
 \ll_p (\log N)^{-A_1\vartheta_p}
       +(\log N)^{-B_1}+\frac{K_N}{N}.
\end{equation}

For $L_0<h\le L_0+H$ and a block $I$, the values $b(n+h)$ involved lie
in an interval of at most $H+K_N\le5H/4$ integers.  Restricting $b$ to
that interval and using the global $\ell^p$ operator norm gives
\[
 \left(\frac1H\sum_{L_0<h\le L_0+H}|\mathcal M_Ib(h)|^p\right)^{1/p}
 \ll_p\|\mathcal M_I\|_{p\to p}.
\]
Every subinterval of $I_N$ intersects consecutive grid blocks, and its
two partial end blocks are dominated by their corresponding maximal block
operators.  Minkowski's inequality therefore permits summation of the
maximal operator norms in
\eqref{eq:local-operator-norm-sum}.  Since
$K_N/N\le N^{-\rho}$ and $\log N\asymp_\theta\log x$, choosing
$A_1,B_1$ sufficiently large in terms of $p,\theta,D_0$ proves
\eqref{eq:localized-toeplitz-bound}.
\end{proof}

\begin{proof}[Proof of Theorem~\ref{thm:one-third-moments}]
Fix $\theta>1/3$.  If $H\ge x$, the result is
Theorem~\ref{thm:long-shift-moments}, and the high tail is empty.  Suppose
$x^\theta\le H<x$ and
split the sum at $n=H$.  Theorem~\ref{thm:long-shift-moments}, with
summation range and shift-window length both $H$, gives an $O_p(1)$
normalized $L^p$ norm for the maximal prefix
$\sup_{1\le y\le H}|C_h(y)|$.

Decompose $(H,x]$ into $O(\log x)$ intervals contained in dyadic shells
$(N,2N]$.  Given any fixed $A>0$, apply
Lemma~\ref{lem:localized-toeplitz} to each shell with $D_0=A+2$.
Every partial sum of the high tail consists of complete earlier shells and
one subinterval of the terminal shell.  Minkowski's inequality therefore
makes the maximal high tail
$O_{p,\theta,A}((\log x)^{-A})$ in normalized $L^p$.  This proves
\eqref{eq:one-third-high-tail-intro}.  Adding the prefix proves
\eqref{eq:one-third-moment-intro}, uniformly in the bounded target $b$
and in the translated shift window.
\end{proof}

\begin{corollary}[Rapid tails in polynomial shift windows]
\label{cor:one-third-tail}
Fix $\theta>1/3$ and $P>0$.  Uniformly for $H\ge x^\theta$,
$L_0\ge0$, every $T\ge1$, and every $1$-bounded sequence $b$, one has
\begin{multline}\label{eq:one-third-distribution}
 \#\left\{L_0<h\le L_0+H:
 \sup_{1\le y\le x}
 \left|\sum_{n\le y}\frac{\Li(n)b(n+h)}n\right|>T\right\}\\
 \ll_{\theta,P}HT^{-P}.
\end{multline}
In particular, for fixed $\varepsilon,C>0$, all sufficiently large $x$
admit
$\mathcal E_{x,L_0,H}\subseteq(L_0,L_0+H]\cap\mathbb Z$ such that
\begin{equation}\label{eq:one-third-tail-size}
 |\mathcal E_{x,L_0,H}|\ll_{\theta,\varepsilon,C}
 H(\log x)^{-C}
\end{equation}
and, outside this set,
\begin{equation}\label{eq:one-third-tail-bound}
 \sup_{1\le y\le x}
 \left|\sum_{n\le y}\frac{\Li(n)\Li(n+h)}n\right|
 \le(\log x)^\varepsilon.
\end{equation}
The same conclusion holds with the second Liouville factor replaced by
any $1$-bounded sequence.
\end{corollary}

\begin{proof}
Choose any fixed $p>\max\{1,P\}$.  Markov's inequality and
Theorem~\ref{thm:one-third-moments} give, uniformly for $T\ge1$,
\[
 \frac1H\#\left\{h:
 \sup_{1\le y\le x}
 \left|\sum_{n\le y}\frac{\Li(n)b(n+h)}n\right|>T\right\}
 \ll_{p,\theta}T^{-p}\le T^{-P},
\]
which proves \eqref{eq:one-third-distribution}.  Now take
$T=(\log x)^\varepsilon$ and choose $P>C/\varepsilon$.  Then
\[
 \frac{|\mathcal E_{x,L_0,H}|}{H}
 \ll_{P,\theta}(\log x)^{-P\varepsilon}
 \ll(\log x)^{-C}.
\]
\end{proof}

\section{General-good-modulus deletion and sparse modulus scores}
\label{sec:general-good-deletion}

We now prepare the major-arc input for shift bands which are genuinely
polynomial in the local summation scale.  Let
\[
 L=\log x,
 \qquad H_{\mathrm{lo}}=\exp(L^{1/2}),
 \qquad H_{\mathrm{hi}}=x^{2/5},
\]
and fix a dyadic shift band
\[
 H<h\leq 2H,
 \qquad H_{\mathrm{lo}}<H<H_{\mathrm{hi}}.
\]
Put \(S=\log H\).  Throughout the middle-band argument we use the fixed
choices
\begin{equation}\label{eq:middle-parameters}
 r=\frac15,
 \qquad a=\frac1{500},
 \qquad R=S^r,
 \qquad \epsilon=R^{-a},
 \qquad \Omega=R^2,
\end{equation}
and the high-tail threshold
\begin{equation}\label{eq:middle-tail-threshold}
 Y=H^{3/2}\exp(10R^{2.01}).
\end{equation}
The Pilatte system used below has a fixed sufficiently small block parameter
\(\varepsilon_b>0\), and its active divisor blocks satisfy
\begin{equation}\label{eq:middle-active-divisor-range}
 \exp(R^{1-\varepsilon_b})\leq M\leq \exp(R).
\end{equation}
In particular,
\begin{equation}\label{eq:middle-scale-separation}
 \frac{h\exp(2R)}Y
 \ll H^{-1/2}\exp(-R^{2.01}),
 \qquad
 \log(Y/h)=\frac12S+10R^{2.01}\gg R^{2.01}.
\end{equation}
The first estimate absorbs all translation and endpoint errors, while the
second supplies the interval length required by the centred estimate.

The notation \(\mathcal Q_{X,\delta,M_0}\) below is the good-modulus set of
\citet{klurman2023shortaps}.  There is a small but important parameter
conversion at this point.  Their Proposition~9.4 uses
\(\mathcal Q_{X,\delta^6,\delta^{-80}}\), and its deduction of KMT
Theorem~1.5 substitutes \(\delta=\epsilon^{1.1}\); KMT Corollary~1.6 inherits
that setup.  Thus the exact set which interfaces with the corollary is
\begin{equation}\label{eq:middle-exact-kmt-good-set}
 \mathcal Q_{X,\epsilon^{6.6},\epsilon^{-88}}.
\end{equation}
We record the source-side ranges that will be used repeatedly.  Put
\(\zeta=\epsilon^{6.6}\) and \(M_0=\epsilon^{-88}\).  Uniformly for
every actual base \(X'\) in the middle-band family,
\begin{equation}\label{eq:middle-kmt-source-ranges}
 (\log X')^{-1/20}<\zeta,\qquad
 \frac1{\log\log X'}\leq M_0
 \leq\frac{\zeta^{20}\log X'}{20\log\log X'}.
\end{equation}
Indeed, \(\log X'\geq S/3\), \(6.6ar<1/20\), and the last inequality
reduces to
\(\epsilon^{-220}\ll\log X'/\log\log X'\), which follows from
\(220ar<1\).  These are the range hypotheses in KMT Lemma~8.1.
Moreover,
\[
 M_0=\epsilon^{-88}\geq\zeta^{-6}=\epsilon^{-39.6}.
\]
The defining zero-free condition is monotone in this parameter, so
membership in \(\mathcal Q_{X',\zeta,M_0}\) implies the membership required
by KMT Lemma~8.2.

\subsection{Deleting the Liouville character main term}

We first isolate the only real character which can survive in one robust
physical base band.  The twist range in the next lemma contains the entire
range used later after Mellin localization, with a fixed margin inside the
prime-sum window of KMT.

\begin{lemma}[Bandwise Liouville deletion]
\label{lem:bandwise-liouville-deletion}
Let \([X,2X]\) be a dyadic base band with \(\log X\geq S/3\), and let
\(X'\in[X,2X]\) be an actual KMT base.  There is at most one primitive
real conductor \(q_0(X)\leq 2X\) with the following property.  Suppose
that \(m=hq/d\) is a working reduced modulus with
\[
 \begin{gathered}
 h\in(H,2H],\qquad q\leq\Omega,\qquad d\mid hq,\qquad m\leq X'/10,\\
 \frac m{X'}\ll H^{-1/2}e^{-9R^{2.01}},\qquad \log m\ll S,
 \end{gathered}
\]
and that \(m\) is typical at the KMT scale and
\[
 m\in\mathcal Q_{X',\epsilon^{6.6},\epsilon^{-88}}.
\]
If \(q_0(X)\) does not exist, or if \(q_0(X)\nmid m\), then every
character \(\chi\pmod m\) which can occur in the KMT minimizing main term
satisfies
\begin{equation}\label{eq:middle-bandwise-character-bound}
 \left|\sum_{n\leq3X'}\Li(n)\chi(n)n^{iu}\right|
 \ll
 \epsilon^{1/2}\frac{\varphi(m)}mX'
 \qquad (|u|\leq 3X'/2).
\end{equation}
The implied constant is absolute and uniform over the robust band.
\end{lemma}

\begin{proof}
Write \(q_*\) for the conductor of the primitive character inducing
\(\chi\).  We separate the high- and low-conductor branches at the cutoff
\begin{equation}\label{eq:middle-conductor-cutoff}
 (\epsilon^{6.6})^{20}=\epsilon^{132}.
\end{equation}

Suppose first that \(q_*>(X')^{\epsilon^{132}}\).  Apply KMT
Lemma~8.2(i) with source parameter \(\zeta=\epsilon^{6.6}\).  Its lower
parameter condition follows from
\begin{equation}\label{eq:middle-kmt-high-lower-parameter}
 \zeta=S^{-6.6ar}\geq(\log X')^{-1/50},
 \qquad 6.6ar<\frac1{50}.
\end{equation}
The lemma gives prime-twist cancellation with saving
\(\zeta^{10}=\epsilon^{66}\).  We use it on the prime range
\begin{equation}\label{eq:middle-prime-range}
 (X')^{\epsilon^7}<p\leq X'.
\end{equation}
This is contained in the permitted range, since
\(\epsilon^7\geq\zeta^{5.5}=\epsilon^{36.3}\).  Partial summation with
weight \(1/(p\log p)\) gives
\begin{equation}\label{eq:middle-high-conductor-harmonic-error}
 \sum_{(X')^{\epsilon^7}<p\leq X'}
 \frac{\chi(p)p^{it}}p
 \ll \epsilon^{66}\log(1/\epsilon)=o(1)
\end{equation}
uniformly in the frequency range required by the Hal\'asz reduction.  On
the other hand, Mertens' theorem gives
\begin{equation}\label{eq:middle-prime-harmonic-mass}
 \sum_{(X')^{\epsilon^7}<p\leq X'}\frac1p
 =7\log(1/\epsilon)+O(1).
\end{equation}
Primes dividing the inducing modulus remove only \(o(1)\) from this mass.
Indeed, \(\log m\ll S\ll\log X'\), and hence
\begin{equation}\label{eq:middle-inducing-prime-loss}
 \sum_{\substack{p\mid m\\p\geq X^{\epsilon^7}}}\frac1p
 \leq
 \frac{\log m}{X^{\epsilon^7}\epsilon^7\log X}=o(1).
\end{equation}
It follows from \eqref{eq:middle-high-conductor-harmonic-error}--
\eqref{eq:middle-inducing-prime-loss} that the KMT pretentious distance is
at least \(5\log(1/\epsilon)\).  KMT Lemma~7.4 then gives
\eqref{eq:middle-bandwise-character-bound} in the high-conductor branch.

Now suppose that
\begin{equation}\label{eq:middle-low-conductor-cap}
 q_*\leq(2X)^{\epsilon^{132}}.
\end{equation}
On every prime scale \(P\geq X^{\epsilon^7}\), apply KMT Lemma~7.9 with
\begin{equation}\label{eq:middle-lemma79-parameter}
 \vartheta=2\epsilon^{125}.
\end{equation}
The conductor hypothesis of that lemma holds because
\begin{equation}\label{eq:middle-125-plus-7}
 (2X)^{\epsilon^{132}}
 \leq X^{2\epsilon^{132}}
 \leq P^{2\epsilon^{125}}=P^\vartheta
\end{equation}
for all sufficiently large \(X\); the exponent identity behind this
inequality is \(125+7=132\).  Its remaining lower parameter condition,
\(\vartheta\geq(\log P)^{-0.4}\), follows from
\(127.8ar<0.4\).  Its frequency window also contains every twist used
here.  Indeed,
\begin{equation}\label{eq:middle-lemma79-twist-window}
 \frac{\log\!\left(P^{(\log P)^{1/25}}\right)}{\log X'}
 \geq \epsilon^{182/25}(\log X')^{1/25}
 \gg S^{(1-182ar)/25}\longrightarrow\infty.
\end{equation}
Thus the external range \(|u|\leq3X'/2\), enlarged by the
\(O(\epsilon^{-2})\) Hal\'asz window below, lies inside the range of KMT
Lemma~7.9 for all sufficiently large \(S\).  More explicitly, since
\(\epsilon^{-2}=S^{1/1250}=o(X')\),
\[
 \frac32X'+\epsilon^{-2}<2.1X'
\]
for all sufficiently large \(S\).  The nonexceptional terms in
that lemma are
\begin{equation}\label{eq:middle-lemma79-errors}
 O\!\left(\epsilon^{125}\log^3(1/\epsilon)P
       +P(\log P)^{-0.3}\right).
\end{equation}
The passage from the primitive character to its induced character modulo
\(m\) again loses only the quantity in
\eqref{eq:middle-inducing-prime-loss}.

It remains to localize the exceptional real-zero term in KMT's explicit
formula.  Let \(\beta\) be a real zero of the primitive character in
\eqref{eq:middle-low-conductor-cap}.  On a smooth prime shell of scale
\(P\), equation~(27) of KMT writes its residue as
\[
 P^{\beta-it}\widetilde h(\beta-it).
\]
The Mellin transform is bounded in the required strip.  Consequently the
residue is \(O(\epsilon^{20}P)\) unless
\begin{equation}\label{eq:middle-dangerous-zero-window}
 1-\beta
 \ll \frac{\log(1/\epsilon)}{\epsilon^7\log X}.
\end{equation}
For completeness, the same conclusion at the sharp harmonic level does
not lose a factor \((1-\beta)^{-1}\).  Put
\[
 \lambda_0=1-\beta,
 \qquad v_0=\epsilon^7\log X,
 \qquad v_1=\log X.
\]
Stieltjes partial summation bounds the exceptional contribution to the
harmonic prime sum by
\begin{equation}\label{eq:middle-exceptional-stieltjes}
 \int_{v_0}^{v_1}\frac{e^{-\lambda_0v}}v\,\dd v.
\end{equation}
Outside \eqref{eq:middle-dangerous-zero-window}, enlarge the absolute
constant there so that
\(\lambda_0v_0\geq20\log(1/\epsilon)\).  The change of variables
\(z=\lambda_0v\) gives
\begin{equation}\label{eq:middle-exceptional-integral-bound}
 \int_{v_0}^{v_1}\frac{e^{-\lambda_0v}}v\,\dd v
 \leq
 \int_{\lambda_0v_0}^{\infty}\frac{e^{-z}}z\,\dd z
 \leq
 \frac{e^{-\lambda_0v_0}}{\lambda_0v_0}
 \ll\epsilon^{20}.
\end{equation}
Thus the denominator which appears after partial summation is
\(\lambda_0v_0\), already bounded away from zero, rather than a naked
factor \(1/(1-\beta)\).

The window \eqref{eq:middle-dangerous-zero-window} lies inside one
Landau--Page uniqueness window for all primitive conductors satisfying
\eqref{eq:middle-low-conductor-cap}, since
\begin{equation}\label{eq:middle-lp-window-comparison}
 \frac{\log(1/\epsilon)}{\epsilon^7\log X}
 =o\!\left(\frac1{\epsilon^{132}\log X}\right).
\end{equation}
Hence at most one primitive real conductor \(q_0(X)\) survives throughout
the robust band.  For every other low-conductor character, the
nonexceptional zero and contour terms in KMT Lemma~7.9, followed by the
same Stieltjes partial summation and geometric log-shell summation, have
total size
\begin{equation}\label{eq:middle-low-conductor-total-error}
 O\!\left(
  \epsilon^{125}\log^4(1/\epsilon)
  +(\epsilon^7S)^{-0.3}
 \right)=o(1).
\end{equation}
KMT Remark~7.2 permits the terminal sharp cutoff.  Its exceptional term is
\(O(P/(1+|t|))\), so the range \(|t|\geq\epsilon^{-20}\) contributes
\(O(\epsilon^{20}P)\); the complementary range is covered by
\eqref{eq:middle-exceptional-stieltjes}--
\eqref{eq:middle-exceptional-integral-bound}.  The external twist in the
statement has size at most \(3X'/2\), and the auxiliary Hal\'asz window is
\(o(X')\).  It therefore remains inside the \(2.1X'\) frequency window of
the KMT prime estimate.  Equations
\eqref{eq:middle-prime-harmonic-mass} and
\eqref{eq:middle-low-conductor-total-error} again give a distance of at
least \(5\log(1/\epsilon)\), and KMT Lemma~7.4 proves
\eqref{eq:middle-bandwise-character-bound}.

For the principal character at twist zero, inclusion--exclusion and the
classical zero-free-region prime number theorem for Liouville give
\begin{equation}\label{eq:middle-principal-liouville}
 \sum_{n\leq3X'}\Li(n)\1_{(n,m)=1}
 \ll X'\exp\!\left\{-cS^{3/5}(\log S)^{-1/5}\right\}
       \prod_{p\mid m}\left(1+\frac1p\right).
\end{equation}
This estimate is uniform over all inclusion--exclusion divisors.  For
every working KMT pair in the middle-band argument,
\begin{equation}\label{eq:middle-principal-residual-scale}
 \frac m{X'}\leq\frac{hq}{Z}
 \ll H^{-1/2}\exp(-9R^{2.01}),
\end{equation}
so \(\log(X'/d)\asymp\log X'\) for every
\(d\mid\operatorname{rad}(m)\).  The product in
\eqref{eq:middle-principal-liouville}, together with
\(m/\varphi(m)\ll\log\log(3m)\), is only polylogarithmic and is absorbed
by the stretched-exponential factor.  This proves
\eqref{eq:middle-bandwise-character-bound} for the principal character at
zero twist.

It remains to treat the principal character at a nonzero twist without
discarding its pole term.  Put \(P_0=(X')^{\epsilon^7}\), take the Hal\'asz
window \(T=\epsilon^{-2}\), and let \(w=u-v\), where \(|v|\leq T\) is the
auxiliary Hal\'asz frequency.  The relevant distance has the lower bound
\begin{equation}\label{eq:middle-principal-twisted-distance}
 \mathbb D_m(\Li(n)n^{iu},n^{iv};X')^2
 \geq
 \sum_{\substack{P_0<p\leq X'\\p\nmid m}}
 \frac{1+\cos(w\log p)}p.
\end{equation}
The first term on the right is
\(7\log(1/\epsilon)+O(1)\), and the omitted primes have total mass
\(o(1)\) by \eqref{eq:middle-inducing-prime-loss}.  For the oscillatory
term, the principal specialization of KMT Lemma~7.9 and sharp-cutoff
Remark~7.2, followed by partial summation, gives uniformly in the window
\eqref{eq:middle-lemma79-twist-window}
\begin{equation}\label{eq:middle-principal-cosine-integral}
 \Re\sum_{P_0<p\leq X'}p^{-1-iw}
 =\int_{\epsilon^7\log X'}^{\log X'}\frac{\cos(wz)}z\,\dd z+O(1)
 \geq-O(1).
\end{equation}
The last inequality is uniform: below \(|w|z=1\) the integrand has
nonnegative integral, while above that point Dirichlet's test bounds the
negative part absolutely.  Hence
\eqref{eq:middle-principal-twisted-distance} is at least
\(5\log(1/\epsilon)\).  With \(T=\epsilon^{-2}\), the errors
\(T^{-1/2}\) and \((\log X')^{-1/4}\) in KMT Lemma~7.4 are
\(o(\epsilon^{1/2})\).  This proves
\eqref{eq:middle-bandwise-character-bound} for every principal twist and
completes the lemma.
\end{proof}

\subsection{A robust bad-modulus set and its linear score}

We next replace a collection of pointwise good-modulus assertions by one
set which is uniform over a dyadic band of physical bases.  Define
\begin{equation}\label{eq:middle-robust-zero-width}
 \Delta_X
 :=2\epsilon^{-88}
   \max_{X\leq T\leq2X}\frac{\log\log T}{\log T},
 \qquad \sigma_X=1-\Delta_X.
\end{equation}
Let \(\mathcal G_X\) be the set of primitive conductors
\(b\in[\tfrac12X^{\epsilon^{132}},2X]\) for which a primitive
Dirichlet \(L\)-function has a zero in
\[
 \Re s>\sigma_X,
 \qquad |\Im s|\leq12X,
\]
and define the zero-bad moduli by
\begin{equation}\label{eq:middle-robust-bad-set}
 \mathcal B_X
 :=\{m\leq2X:b\mid m\text{ for some }b\in\mathcal G_X\}.
\end{equation}
Typicality is not included in \(\mathcal B_X\); it will be verified
uniformly for every working modulus below.

\begin{lemma}[Robust bad-modulus score]
\label{lem:global-bad-score}
For every robust band \([X,2X]\) occurring in the middle-band high-tail
argument, one has
\begin{equation}\label{eq:middle-robust-generator-count}
 |\mathcal G_X|\ll(\log X)^{20\epsilon^{-88}},
 \qquad
 |\mathcal B_X\cap[1,U]|\ll U\eta_X
 \quad(U\geq1),
\end{equation}
where
\begin{equation}\label{eq:middle-eta-x}
 \eta_X:=\exp\{-S^{1-133ar}\}.
\end{equation}
Moreover, for every actual base \(X'\in[X,2X]\) and every working reduced
modulus \(m=hq/d\leq X'/10\) from the middle-band family,
\begin{equation}\label{eq:middle-robust-good-bridge}
 m\notin\mathcal B_X
 \quad\Longrightarrow\quad
 m\in\mathcal Q_{X',\epsilon^{6.6},\epsilon^{-88}}.
\end{equation}

For \(q\leq\Omega\), put
\begin{equation}\label{eq:middle-linear-score}
 W_{X,q}(h)
 :=\frac1{hq}
   \sum_{\substack{m\mid hq\\m\in\mathcal B_X}}\varphi(m).
\end{equation}
Then
\begin{equation}\label{eq:middle-linear-score-average}
 \frac1H\sum_{H<h\leq2H}W_{X,q}(h)
 \ll q\eta_X.
\end{equation}
\end{lemma}

\begin{proof}
Apply the log-free zero-density estimate used in KMT Lemma~8.1 with
conductor cap \(2X\), height \(12X\), and left edge \(\sigma_X\).
Since
\[
 1-\sigma_X
 \leq \frac{2\epsilon^{-88}\log\log X}{\log X}(1+o(1)),
\]
the number of primitive conductor generators is at most
\begin{equation}\label{eq:middle-zero-density-calculation}
 \big((2X)^2(12X)\big)^{(5/2)(1-\sigma_X)}
 \ll(\log X)^{20\epsilon^{-88}},
\end{equation}
which proves the first assertion in
\eqref{eq:middle-robust-generator-count}.  Every generator is at least
\(\tfrac12X^{\epsilon^{132}}\).  If \(U\) is smaller than the least
generator, then \(\mathcal B_X\cap[1,U]\) is empty.  Otherwise only
generators \(b\leq U\) contribute, and their number of multiples is
exactly \(\lfloor U/b\rfloor\).  Hence
\begin{align}
 |\mathcal B_X\cap[1,U]|
 &\leq \sum_{\substack{b\in\mathcal G_X\\b\leq U}}
       \left\lfloor\frac Ub\right\rfloor\notag\\
 &\ll U X^{-\epsilon^{132}}(\log X)^{20\epsilon^{-88}}
 \leq U\exp\{-S^{1-133ar}\}.
 \label{eq:middle-generator-density}
\end{align}
The last inequality uses \(\log X\geq S/3\) and
\begin{equation}\label{eq:middle-generator-entropy-gap}
 1-132ar>88ar,
 \qquad ar=\frac1{2500}.
\end{equation}

The robust definition also proves
\eqref{eq:middle-robust-good-bridge}.  Indeed, the good set on the
right tests primitive conductors above \((X')^{\epsilon^{132}}\), zeros
up to height \(3X'\), and zero-free width
\(\epsilon^{-88}\log\log X'/\log X'\).  The definition
\eqref{eq:middle-robust-bad-set} starts at the smaller threshold
\(\tfrac12X^{\epsilon^{132}}\), reaches height \(12X\), and uses the
maximum width on the whole band.  Each of its three tests is therefore
stronger.

It remains to prove \eqref{eq:middle-linear-score-average}.  Put
\(g=(m,q)\).  Prime by prime,
\begin{equation}\label{eq:middle-divisibility-pullback}
 m\mid hq
 \quad\Longleftrightarrow\quad
 \frac{m}{(m,q)}\mid h.
\end{equation}
Since \(h\asymp H\), the number of such shifts satisfies
\begin{equation}\label{eq:middle-divisibility-count}
 \#\{H<h\leq2H:m\mid hq\}
 \leq \frac{Hg}{m}+1.
\end{equation}
Using \(1/h\leq1/H\) and \(\varphi(m)\leq m\), we obtain
\begin{align}
 \frac1H\sum_{H<h\leq2H}W_{X,q}(h)
 &\leq
 \frac1{H^2q}
 \sum_{\substack{m\leq2Hq\\m\in\mathcal B_X}}
 \varphi(m)\left(\frac{H(m,q)}m+1\right)\notag\\
 &\leq
 \frac1{Hq}\sum_{\substack{m\leq2Hq\\m\in\mathcal B_X}}(m,q)
 +\frac1{H^2q}
  \sum_{\substack{m\leq2Hq\\m\in\mathcal B_X}}m.
 \label{eq:middle-score-two-terms}
\end{align}
The first term is at most
\(|\mathcal B_X\cap[1,2Hq]|/H\ll q\eta_X\), since
\((m,q)\leq q\).  Partial summation in the second term and
\eqref{eq:middle-robust-generator-count} give
\[
 \sum_{\substack{m\leq2Hq\\m\in\mathcal B_X}}m
 \ll (Hq)^2\eta_X,
\]
so the second term is also \(O(q\eta_X)\).  This proves
\eqref{eq:middle-linear-score-average}.
\end{proof}

Define, before choosing any prefix, rational denominator, Mellin frequency,
or sparsity exponent,
\[
 \mathcal X_H:=
 \left\{2^j:\frac12e^{S/3}\leq2^j\leq8xe^{2R}\right\}.
\]
Every actual KMT base in the middle-band high-tail argument lies in one of
these robust dyadic bands: the lower bound is the uniform
\(X_{\min}\geq e^{S/3}\), while all parent and subset-rescaled variables
are at most \(O(xe^{O(R)})\).  Consequently
\begin{equation}\label{eq:middle-number-physical-bands}
 |\mathcal X_H|\ll L.
\end{equation}
Define the total score
\begin{equation}\label{eq:middle-total-score}
 \mathfrak W_H(h)
 :=\sum_{X\in\mathcal X_H}\sum_{q\leq\Omega}W_{X,q}(h).
\end{equation}
Equations \eqref{eq:middle-linear-score-average} and
\eqref{eq:middle-number-physical-bands} imply
\begin{equation}\label{eq:middle-total-score-average}
 \frac1H\sum_{H<h\leq2H}\mathfrak W_H(h)
 \ll L\Omega^2\exp\{-S^{1-133ar}\}
 \ll \exp\{-S^{c_1}\}
\end{equation}
for some absolute \(c_1>0\).

For every \(X\in\mathcal X_H\) for which the exceptional conductor in
Lemma~\ref{lem:bandwise-liouville-deletion} exists, put
\begin{equation}\label{eq:middle-lp-trace}
 E_{X,H}^{\mathrm{LP}}
 :=\{H<h\leq2H:q_0(X)\mid hq
                 \text{ for some }q\leq\Omega\}.
\end{equation}
The dangerous-zero window
\eqref{eq:middle-dangerous-zero-window} and Siegel's theorem imply,
ineffectively, that
\begin{equation}\label{eq:middle-q0-siegel}
 q_0(X)\gg_A S^A
\end{equation}
for every fixed \(A>0\).  Indeed, the upper bound for \(1-\beta\) is a
fixed negative power of \(S\), while Siegel's lower bound may be used with
an arbitrarily small fixed exponent.  The exact gcd count is
\begin{align}
 |E_{X,H}^{\mathrm{LP}}|
 &\leq
 \sum_{q\leq\Omega}
 \left(\frac{H(q_0(X),q)}{q_0(X)}+1\right)\notag\\
 &\leq
 H\Omega\frac{\tau(q_0(X))}{q_0(X)}+\Omega.
 \label{eq:middle-lp-trace-count}
\end{align}
After summing over \(\mathcal X_H\),
\eqref{eq:middle-q0-siegel}, the bound
\(\tau(n)=n^{o(1)}\), and \(S\geq L^{1/2}\) give
\begin{equation}\label{eq:middle-all-lp-traces}
 \left|\bigcup_{X\in\mathcal X_H}E_{X,H}^{\mathrm{LP}}\right|
 \ll_C HL^{-C}
\end{equation}
for every fixed \(C>0\).

Fix once and for all a sufficiently small absolute \(b>0\), and define
\begin{equation}\label{eq:global-apf-exception-set}
 E_{\mathrm{APF}}(H)
 :=
 \bigcup_{X\in\mathcal X_H}E_{X,H}^{\mathrm{LP}}
 \ \cup\
 \{H<h\leq2H:\mathfrak W_H(h)>R^{-b}\}.
\end{equation}
This set is defined before a sparsity exponent is chosen.  Markov's
inequality, \eqref{eq:middle-total-score-average}, and
\eqref{eq:middle-all-lp-traces} give
\begin{equation}\label{eq:global-apf-exception-size}
 |E_{\mathrm{APF}}(H)|\ll_C HL^{-C}
\end{equation}
for every fixed \(C>0\).  This is the exceptional set used by the scored
major-arc estimate in the next section.

\section{A scored progression Fourier estimate in the middle range}
\label{sec:scored-middle-apf}

For a consecutive integer interval \((T,T+U]\), define
\begin{equation}\label{eq:middle-partial-apf-definition}
 \mathcal F_{h,U}^{[T]}(\alpha;Z)
 :=\sum_{Z<y\leq2Z}
   \left|\sum_{T<t\leq T+U}
     \Li(y+ht)e(\alpha t)\right|.
\end{equation}
Thus \(\mathcal F_{h,M}^{[0]}(\alpha;Z)\) is the quantity in
\eqref{eq:apf-definition}.  The purpose of the score
\eqref{eq:middle-linear-score} is that it retains the exact coefficient of
every bad gcd stratum instead of discarding the entire shift.

\begin{proposition}[Scored middle-band progression Fourier estimate]
\label{prop:middle-scored-apf}
Use the parameters \eqref{eq:middle-parameters}--
\eqref{eq:middle-active-divisor-range}, and let
\(h\in(H,2H]\setminus E_{\mathrm{APF}}(H)\).  Uniformly over the
high-tail parent intervals and subset-rescaled bases occurring for
\(Z\geq Y/e^R\), over every \(M\) satisfying
\eqref{eq:middle-active-divisor-range}, and over \(\alpha\in\T\), one has
\begin{equation}\label{eq:global-scored-apf-conclusion}
 \mathcal F_{h,M}^{[0]}(\alpha;Z)
 \ll R^{-b_1}MZ,
 \qquad
 b_1:=\min\left\{\frac a2,b,\frac15\right\}>0.
\end{equation}
Let \(X_{\min}\) be the least physical KMT base in this complete family.
The same estimate, with the same exceptional set and constants, holds with
\(\Li(n)\) replaced by \(\Li(n)n^{it_0}\), uniformly for
\begin{equation}\label{eq:middle-twist-window}
 |t_0|\leq X_{\min}/2.
\end{equation}
\end{proposition}

\begin{proof}
We first dispose of the minor arcs.  Dirichlet approximation supplies a
reduced fraction \(a_0/q\) such that
\begin{equation}\label{eq:middle-dirichlet-approximation}
 1\leq q\leq M/\Omega,
 \qquad
 \left|\alpha-\frac{a_0}q\right|
 \leq\frac{\Omega}{Mq}.
\end{equation}
When \(q>\Omega\), use the joint Ramar\'e argument of
Lemma~\ref{lem:apf-minor}.  We record why its two shift-dependent inputs
remain uniform in the present polynomial shift band.  With
\[
 P_1=\Omega^{C_M},
 \qquad Q_1=M/\Omega^3,
 \qquad \mathcal P_h=\{p\in[P_1,Q_1]:p\nmid h\},
\]
removing the primes which divide \(h\) changes the first-range harmonic
mass by at most
\begin{equation}\label{eq:middle-shared-prime-loss}
 \sum_{\substack{p\mid h\\p\geq P_1}}\frac1p+O(P_1^{-1})
 \leq\frac{\log h}{P_1\log P_1}+O(P_1^{-1})
 \ll R^{-100}.
\end{equation}
The translated-endpoint term is bounded by \(hM^2\), and
\begin{equation}\label{eq:middle-minor-boundary}
 \frac{hM}{Z}\leq\frac{2He^{2R}}{Y}
 \ll H^{-1/2}\exp(-R^{2.01}).
\end{equation}
After dualizing the outer \(\ell^1\)-norm, the fourth-power expansion in
the proof of Lemma~\ref{lem:apf-minor} fixes one start and leaves only
\(O(M)\) choices for each of the other three starts; hence the start-tuple
count is \(O(ZM^3)\), not \(O(Zh^3M^3)\).  The remaining progression
sum is
\[
 \min\{M/P+1,\|r\alpha\|^{-1}\},
 \qquad r=p_1+p_2-p_3-p_4.
\]
The four-prime sieve estimate and Vinogradov's lemma therefore give the
same bound as \eqref{eq:apf-minor-raw}, namely
\begin{equation}\label{eq:middle-minor-raw}
 \mathcal F_{h,M}^{[0]}(\alpha;Z)
 \ll MZ\frac{(\log M)^{1/4}\log\log M}{\Omega^{1/4}}
 \ll R^{-1/5}MZ.
\end{equation}
Equations \eqref{eq:middle-shared-prime-loss} and
\eqref{eq:middle-minor-boundary} also absorb the typical-complement and
prime-square errors.  This proves
\eqref{eq:global-scored-apf-conclusion} on the minor arcs.

It remains to treat \(q\leq\Omega\).  We begin at the exact rational
frequency \(a_0/q\) and retain an arbitrary consecutive partial interval
\((T,T+U]\) with
\begin{equation}\label{eq:middle-long-partial-length}
 U\geq U_0:=M/\Omega^4.
\end{equation}
Split \(t\) into its \(q\) residue classes and write \(Q=hq\).  For each
class, Lemma~\ref{lem:endpoint-residue} contributes the factor \(1/Q\)
and replaces the diagonal family selected by \(y\) with all residues
modulo \(Q\).  Group those residues according to \(d=(c,Q)\).  After
division by \(d\), the exact reduced data are
\begin{equation}\label{eq:middle-gcd-reduced-data}
 m=Q/d,
 \qquad
 X_d=\frac{Z+O(hM+hq)}d,
 \qquad
 L_d=\frac{hU}d,
 \qquad
 \frac{L_d}{m}=\frac Uq.
\end{equation}
The harmless \(O(hM+hq)\) displacement only chooses a neighboring actual
base \(X'\) in one of the robust bands from \(\mathcal X_H\).

At this point we use KMT Theorem~1.5 and Corollary~1.6 in their
general-good-modulus form
\citep[Theorem~1.5, Corollary~1.6, and Proposition~9.4]
{klurman2023shortaps}.  In KMT notation, for each actual base
\(X'\asymp X_d\), we take
\[
 X_{\mathrm{KMT}}=X',
 \qquad h_{\mathrm{KMT}}=L_d,
 \qquad Q_{\mathrm{KMT}}=q_{\mathrm{KMT}}=m.
\]
Thus the source hypotheses reduce to
\[
 X'\geq L_d\geq10m,
 \qquad
 \bigl(\log(L_d/m)\bigr)^{-1/200}\leq\epsilon^{1.1},
\]
together with typicality; on each good stratum we additionally use
membership in \(\mathcal Q_{X',\epsilon^{6.6},\epsilon^{-88}}\), as
verified below.
No condition \(m\leq(X')^{\epsilon^{200}}\) is used here: that is only a
sufficient condition for automatic good-modulus membership, whereas the
present argument retains the general-good strata and scores their
complement.

We verify every KMT size condition at the shortest permitted value of
\(U\).  Since \(q\leq\Omega=R^2\),
\begin{equation}\label{eq:middle-partial-kmt-lower}
 \log(U/q)
 \geq R^{1-\varepsilon_b}-O(\log R),
\end{equation}
and hence
\begin{equation}\label{eq:middle-partial-kmt-conditions}
 X_d\geq L_d\geq10m,
 \qquad
 (\log(U/q))^{-1/200}\leq\epsilon^{1.1}\leq\epsilon.
\end{equation}
The stronger first inequality follows from
\(1.1a<(1-\varepsilon_b)/200\); it is the lower endpoint required before
the \(\delta=\epsilon^{1.1}\) substitution in KMT Proposition~9.4.
The typicality parameter satisfies
\begin{equation}\label{eq:middle-partial-typicality}
 y_{\mathrm{typ}}
 :=(U/q)^{\epsilon^2},
 \qquad
 \log y_{\mathrm{typ}}
 \geq R^{1-\varepsilon_b-2a+o(1)}.
\end{equation}
Thus \(y_{\mathrm{typ}}\gg1000\log(hq)\).  KMT Lemma~9.1 implies that
every divisor \(m\mid hq\) in
\eqref{eq:middle-gcd-reduced-data} is automatically typical.  If
\(m\notin\mathcal B_X\) for its robust band, then
\eqref{eq:middle-robust-good-bridge} places it in the exact set
\(\mathcal Q_{X',\epsilon^{6.6},\epsilon^{-88}}\).  Moreover, the
exceptional conductor of Lemma~\ref{lem:bandwise-liouville-deletion}
cannot divide \(m\): otherwise it would divide \(hq\), placing \(h\) in
the trace \eqref{eq:middle-lp-trace}.  Consequently the character main
term is deleted on every good stratum.

For clarity, this deletion has exactly the scale required in the KMT
variance.  Lemma~\ref{lem:bandwise-liouville-deletion}, inserted into the
minimizing-character term of KMT Corollary~1.6, bounds its value in one reduced
residue class by \(O(\epsilon^{1/2}L_d/m)\).  Hence its squared
\(L^2\)-mass over the \(\varphi(m)\) reduced classes and an ambient start
interval of length \(O(X_d)\) is
\begin{equation}\label{eq:middle-character-main-mass}
 \ll \epsilon\,\varphi(m)X_d\left(\frac{L_d}{m}\right)^2.
\end{equation}
The complex-valued theorem gives the identical estimate after the external
twist, because the oscillatory integral has absolute value at most \(L_d\).
Thus adding the character term back preserves the KMT variance scale.

Let
\[
 S_{\Li}(u,c;L_d,m)
 :=\sum_{\substack{u<n\leq u+L_d\\n\equiv c\ (m)}}\Li(n).
\]
KMT Corollary~1.6 and Cauchy--Schwarz in the start variable and in the
reduced residue classes give
\begin{align}
 &d\int\sum_{c\bmod m}^{*}
   |S_{\Li}(u,c;L_d,m)|\,\dd u
 \notag\\
 &\quad\ll
 d\big(\varphi(m)X_d\big)^{1/2}
 \left(
   \epsilon\varphi(m)X_d\left(\frac Uq\right)^2
 \right)^{1/2}
 \ll
 \epsilon^{1/2}\varphi(m)\frac{ZU}{q}.
 \label{eq:middle-good-stratum-l1}
\end{align}
Here and below the integral may be split into two adjacent dyadic base
intervals; the robust good set is valid on both.  The endpoint--residue
factor is \(1/Q\).  Summing
\eqref{eq:middle-good-stratum-l1} over the good divisors and using
\begin{equation}\label{eq:middle-totient-divisor-identity}
 \sum_{d\mid Q}\varphi(Q/d)=Q
\end{equation}
shows that one Fourier residue class \(b\pmod q\) costs at most
\(\epsilon^{1/2}UZ/q\).  Summing the \(q\) classes leaves
\begin{equation}\label{eq:middle-all-good-strata}
 O(\epsilon^{1/2}UZ).
\end{equation}

On a bad stratum no cancellation is used.  A reduced progression contains
at most \(U/q+O(1)\) terms.  Before the endpoint--residue factor, the
integration over starts and the \(\varphi(m)\) reduced residue classes
therefore cost \(ZU\varphi(m)/q\).  Multiplication by \(1/Q\), followed
by the sum over the \(q\) Fourier residue classes, leaves exactly
\begin{equation}\label{eq:middle-one-bad-stratum}
 UZ\frac{\varphi(m)}{hq}.
\end{equation}
Thus all bad strata assigned to the robust band \(X\) contribute
\(UZ W_{X,q}(h)\).  Summing over the physical base bands only enlarges
this to
\begin{equation}\label{eq:middle-bad-strata-total}
 UZ\sum_{X\in\mathcal X_H}W_{X,q}(h)
 \leq UZ\mathfrak W_H(h).
\end{equation}
This calculation is the reason for the linear weight \(\varphi(m)/(hq)\)
in \eqref{eq:middle-linear-score}; no divisor-counting or inverse-totient
factor is introduced.

The residue splitting and moving endpoints contribute
\(O(qZ+qhU)\).  Relative to \(UZ\), these errors satisfy
\begin{equation}\label{eq:middle-major-endpoint-ratios}
 \frac qU
 \ll R^2\exp\{-R^{1-\varepsilon_b}+O(\log R)\},
 \qquad
 \frac{qh}{Z}
 \ll H^{-1/2}\exp(-9R^{2.01}),
\end{equation}
and are smaller than any fixed power of \(R^{-1}\).  Combining
\eqref{eq:middle-all-good-strata},
\eqref{eq:middle-bad-strata-total}, and
\eqref{eq:middle-major-endpoint-ratios} gives the uniform rational-prefix
bound
\begin{equation}\label{eq:middle-rational-prefix-bound}
 \mathcal F_{h,U}^{[T]}(a_0/q;Z)
 \ll
 \big(\epsilon^{1/2}+\mathfrak W_H(h)+R^{-10}\big)UZ
 \qquad(U\geq U_0).
\end{equation}
Both the robust bad sets and the score are independent of \(T\) and
\(U\).  Since \(h\notin E_{\mathrm{APF}}(H)\), the coefficient on the
right is
\begin{equation}\label{eq:middle-rational-delta}
 \delta_H
 \ll R^{-a/2}+R^{-b}+R^{-10}.
\end{equation}

We now pass from \(a_0/q\) to the nearby frequency
\(\alpha=a_0/q+\theta\).  If \(\theta=0\), take one block of length
\(M\); otherwise put
\[
 B=\min\{M,\lfloor|\theta|^{-1}\rfloor\}
\]
and divide \([1,M]\) into consecutive blocks of length at most \(B\).
There are \(O(1+M|\theta|)=O(\Omega)\) such blocks.  On one block,
write
\[
 A_y(v)=\sum_{T<t\leq T+v}
 \Li(y+ht)e(a_0t/q).
\]
Discrete Abel summation gives
\begin{align}
 \sum_{u=1}^{B}A_y(u)-A_y(u-1)\,e(\theta(T+u))
 &=A_y(B)e(\theta(T+B))\notag\\
 &\quad+(1-e(\theta))
 \sum_{v<B}A_y(v)e(\theta(T+v)).
 \label{eq:middle-discrete-abel}
\end{align}
After summing over \(y\), use
\eqref{eq:middle-rational-prefix-bound} for \(v\geq U_0\) and the
trivial bound \(Zv\) below \(U_0\).  Since \(|\theta|B\leq1\), the long
parts contribute \(O(\delta_H ZM)\).  The short prefixes contribute at
most
\begin{equation}\label{eq:middle-short-prefix-abel}
 O(ZU_0\Omega)=O(ZM\Omega^{-3}).
\end{equation}
It follows that every major arc satisfies
\begin{equation}\label{eq:middle-nearby-major-bound}
 \mathcal F_{h,M}^{[0]}(\alpha;Z)
 \ll
 (R^{-a/2}+R^{-b}+R^{-6})MZ.
\end{equation}
Together with the minor-arc estimate
\eqref{eq:middle-minor-raw}, this proves
\eqref{eq:global-scored-apf-conclusion}.

Finally replace \(\Li(n)\) by \(\Li(n)n^{it_0}\), with
\eqref{eq:middle-twist-window}.  The joint Ramar\'e argument only acquires
unit-modulus prime and integer coefficients, so
\eqref{eq:middle-minor-raw} is unchanged.  On a major arc, apply the
complex-valued KMT hybrid theorem afresh for each \(t_0\).  If
\(v\in[-X',X']\) is its minimizing twist at the actual base \(X'\), the
twist in the Liouville character sum is \(u=v-t_0\) up to an irrelevant
choice of sign.  Since \(X'\geq X_{\min}\),
\begin{equation}\label{eq:middle-combined-twist-window}
 |u|\leq X'+X_{\min}/2\leq3X'/2.
\end{equation}
Lemma~\ref{lem:bandwise-liouville-deletion} therefore deletes the main
term.  The robust good sets, conductor traces, and bad-modulus score do not
depend on \(t_0\), so the same exceptional set works throughout.  This
proves the twisted assertion.
\end{proof}

\section{Mellin localization and the middle-band transfer}
\label{sec:mellin-localized-transfer}

We now transfer the scored progression estimate of
Proposition~\ref{prop:middle-scored-apf} to the logarithmically weighted
correlation.  The point requiring care is that logarithmic dilation does
not preserve a fixed physical interval: a divisor \(d\) replaces
\((Y,y]\) by \((dY,dy]\).  We retain this divisor dependence until the
active divisor has been placed in one of its existing dyadic blocks, and
then separate it by Mellin inversion.

Throughout this section \(H<h\leq 2H\) is a middle shift band and
\begin{equation}\label{eq:mellin-middle-parameters}
 L=\log x,\qquad S=\log H,\qquad
 R=S^{1/5},\qquad \epsilon=R^{-1/500},\qquad
 \Omega=R^2,
\end{equation}
while
\begin{equation}\label{eq:mellin-tail-threshold}
 Y=H^{3/2}\exp(10R^{2.01}).
\end{equation}
We use the fixed Pilatte block system with upper prime scale \(e^R\),
pruned by deleting the primes dividing \(h\), as in
Lemma~\ref{lem:coprime-block-pruning}.  Its divisor family and harmonic
normalization are denoted by \(\cD(h)\) and
\(W_{\mathrm{Pil},h}\), respectively.

For a smooth function \(w\) supported in \([1/2,4]\), set
\begin{equation}\label{eq:mellin-localized-correlation}
 C_{h,w}(N):=
 \sum_{n\geq1}\frac{\Li(n)\Li(n+h)}n\,w(n/N).
\end{equation}
Our Mellin convention is
\begin{equation}\label{eq:mellin-transform-convention}
 \widehat w(t):=\int_0^\infty w(u)u^{-it}\frac{\dd u}{u},
 \qquad
 w(u)=\frac1{2\pi}\int_{\mathbb R}\widehat w(t)u^{it}\,\dd t.
\end{equation}
In particular, the sign of every divisor phase below is fixed by
\eqref{eq:mellin-transform-convention}.

\begin{proposition}[Mellin-localized transfer]
\label{prop:mellin-localized-transfer}
There is an absolute \(b_2>0\) with the following property.  Suppose
\(h\notin E_{\mathrm{APF}}(H)\), where \(E_{\mathrm{APF}}(H)\) is the
fixed exceptional set in Proposition~\ref{prop:middle-scored-apf}.  If
\(w(n/N)\) is one of the smooth dyadic or one-sided terminal cutoffs used
in the decomposition of \((Y,x]\), and all its subset-rescaled bases belong
to the finite family \(\mathcal X_H\) of Section~\ref{sec:general-good-deletion},
then
\begin{equation}\label{eq:mellin-smooth-transfer}
 |C_{h,w}(N)|
 \ll R^{-b_2}\int_{\mathbb R}|\widehat w(t)|\,\dd t,
\end{equation}
uniformly over this declared cutoff family.  The same assertion holds for
the one-sided smoothed terminal cutoffs defined below, with their displayed
Mellin \(L^1\)-norm.
Consequently, for an absolute \(c_2>0\),
\begin{equation}\label{eq:middle-high-tail}
 \sup_{Y\leq y\leq x}
 \left|\sum_{Y<n\leq y}
 \frac{\Li(n)\Li(n+h)}n\right|
 \ll L^{1-c_2}.
\end{equation}
One may take \(c_2=b_2/20\).
\end{proposition}

\begin{proof}
Complete multiplicativity gives the exact localized dilation identity
\begin{equation}\label{eq:mellin-localized-dilation}
 W_{\mathrm{Pil},h}C_{h,w}(N)
 =\sum_{d\in\cD(h)}
   \sum_{d\mid m}\frac{\Li(m)\Li(m+hd)}m
   w\left(\frac{m}{dN}\right).
\end{equation}
Indeed, for fixed \(d\), putting \(m=dn\) changes the inner sum into
\(d^{-1}C_{h,w}(N)\), and
\(\sum_{d\in\cD(h)}d^{-1}=W_{\mathrm{Pil},h}\).

Denote the right side of \eqref{eq:mellin-localized-dilation} by
\(S_{1,h,w}\), and introduce its centred analogue
\begin{multline*}
 S_{2,h,w}:=
 \sum_{(p_1,\ldots,p_J)}\sum_m
 \frac{\Li(m)\Li(m+h p_1\cdots p_J)}m
 w\!\left(\frac{m}{p_1\cdots p_JN}\right)\\
 {}
 \prod_{j=1}^J\left(\1_{p_j\mid m}-\frac1{p_j}\right).
\end{multline*}
Expanding the product gives \(S_{2,h,w}-S_{1,h,w}\) as a sum over
nonempty \(I\subseteq[J]\).  For a fixed tuple write
\[
 d_0=\prod_{i\notin I}p_i,\qquad d'=\prod_{i\in I}p_i,
\]
and rescale \(m=d_0k\).  Complete multiplicativity gives the exact
\(I\)-term
\begin{equation}\label{eq:mellin-exact-subset-term}
 (-1)^{|I|}
 \sum_{d_0\in\cD_{[J]\setminus I}(h)}\frac1{d_0}
 \sum_{d'\in\cD_I(h)}\frac1{d'}
 \sum_k\frac{\Li(k)\Li(k+hd')}{k}
 w\!\left(\frac{k}{d'N}\right).
\end{equation}
This display retains both the complementary harmonic weight \(1/d_0\) and
the active weight \(1/d'\).  Place \(d'\) in its existing dyadic block
\begin{equation}\label{eq:mellin-active-divisor-block}
 d'\in\cD_I(h;M),\qquad M/2<d'\leq M.
\end{equation}
The cutoff in \eqref{eq:mellin-localized-dilation} is then exactly
\begin{equation}\label{eq:mellin-divisor-separation}
 w\left(\frac{k}{d'N}\right)
 =\frac1{2\pi}\int_{\mathbb R}\widehat w(t)
   \left(\frac{k}{N}\right)^{it}d'^{-it}\,\dd t.
\end{equation}
If the left side is nonzero, the support of \(w\) and
\eqref{eq:mellin-active-divisor-block} imply
\[
 MN/4<k\leq4MN.
\]
We shall use the harmless common envelope
\begin{equation}\label{eq:mellin-common-k-range}
 k\in[MN/4,8MN].
\end{equation}
Thus every integer sum in \eqref{eq:mellin-divisor-separation} is finite
on a range independent of \(d'\).  Since
\(\widehat w\in L^1(\mathbb R)\), the exchange of the finite sums with
the Mellin integral is absolutely justified.

For a fixed Mellin frequency \(t\), the active divisor polynomial becomes
\begin{equation}\label{eq:mellin-divisor-polynomial}
 Q_{I,h,M,t}(\alpha)
 :=\sum_{d'\in\cD_I(h;M)}
 \frac{d'^{-it}e(\alpha d')}{d'}.
\end{equation}
The new coefficient has modulus one.  Expanding the fourth moment and
taking absolute values term by term gives
\begin{align}
 \int_{\mathbb T}|Q_{I,h,M,t}(\alpha)|^4\,\dd\alpha
 &\leq
 \sum_{\substack{d_1+d_2=d_3+d_4\\
                   d_i\in\cD_I(h;M)}}
 \frac1{d_1d_2d_3d_4}                                      \notag\\
 &\ll_{\varepsilon_b}\frac{V^{4J}}{M(\log M)^4}.
 \label{eq:mellin-fourth-moment}
\end{align}
Hence the divisor fourth moment is uniform in \(t\).  The other Mellin
factor, \(k^{it}\), is absorbed into one of the two \(1\)-bounded
multiplicative functions in the quotient correlation.

We next make the frequency range compatible with the constrained twist in
the KMT theorem.  Let \(X_{\min}\) be the smallest KMT base among all
physical blocks which occur in the localized high-tail transfer.  The
middle-band margins give
\begin{equation}\label{eq:mellin-smallest-kmt-base}
 \log X_{\min}\geq S/3.
\end{equation}
Before applying KMT, truncate every Mellin integral globally to
\begin{equation}\label{eq:mellin-global-frequency-cutoff}
 |t|\leq X_{\min}/2.
\end{equation}
For a fixed smooth dyadic cutoff the omitted tail decreases faster than
any power of \(X_{\min}\).  For the terminal cutoff \(w_\eta\) constructed
below, two integrations by parts give
\begin{equation}\label{eq:mellin-terminal-pointwise-decay}
 |\widehat w_\eta(t)|
 \ll\min\{1,|t|^{-1},(\eta t^2)^{-1}\},
\end{equation}
and therefore, because \(X_{\min}\gg\eta^{-1}\),
\begin{equation}\label{eq:mellin-frequency-tail}
 \int_{|t|>X_{\min}/2}|\widehat w_\eta(t)|\,\dd t
 \ll(\eta X_{\min})^{-1}.
\end{equation}
We take \(\eta=R^{-A_0}\) with \(A_0\) a sufficiently large fixed
constant.  Equations \eqref{eq:mellin-smallest-kmt-base} and
\eqref{eq:mellin-frequency-tail} make the omitted tail smaller than
every fixed negative power of \(R\).  This estimate remains unchanged
after the dilation is normalized: the trivial harmonic mass on the
right of \eqref{eq:mellin-localized-dilation} is
\(O(W_{\mathrm{Pil},h})\), and division by the same
\(W_{\mathrm{Pil},h}\) cancels it.

Fix now one retained \(t\).  Let \([X_0,2X_0]\) be a robust band and
\(X'\in[X_0,2X_0]\) the actual KMT base.  Apply the complex-valued KMT
theorem afresh to
\begin{equation}\label{eq:mellin-twisted-liouville}
 f_t(n):=\Li(n)n^{it}.
\end{equation}
Let \(v\in[-X',X']\) be the minimizing twist supplied by this particular
application.  Since \(X'\geq X_{\min}\),
\begin{equation}\label{eq:mellin-fresh-minimizer-range}
 |v-t|\leq X'+X_{\min}/2\leq3X'/2.
\end{equation}
The character mean is therefore a Liouville--character mean with twist
\(v-t\) (up to the immaterial choice of conjugating the character), and
Lemma~\ref{lem:bandwise-liouville-deletion} deletes it outside the same
conductor trace.  Notice that \(v\) is chosen anew for each \(t\); no
translation of a previously constrained minimizer is asserted.  The
robust good set and the bad-modulus score are independent of \(t\).
On the minor arcs all inserted coefficients have modulus one, while the
centred theorem of Tao--Ter\"av\"ainen permits arbitrary \(1\)-bounded
functions.  Consequently both Proposition~\ref{prop:middle-scored-apf}
and the centred estimate hold uniformly for every retained Mellin slice.

We now retain every normalization in the frequency transfer.  Fix a common
physical interval \(J_Z\subseteq[Z,2Z]\) from
\eqref{eq:mellin-common-k-range}, and let \(J_Z(u)\) be any initial
subinterval.  Define the unweighted Mellin-slice sum
\begin{equation}\label{eq:mellin-unweighted-slice}
 T_{I,M,t}(J_Z(u)):=
 \sum_{d'\in\cD_I(h;M)}\frac{d'^{-it}}{d'}
 \sum_{k\in J_Z(u)}
 \Li(k)\Li(k+hd')\left(\frac{k}{N}\right)^{it}.
\end{equation}
Writing \(k=r+hq\) and averaging translations by \(1\le s\le M\) gives
the exact quotient identity
\begin{multline}\label{eq:mellin-quotient-translation}
 T_{I,M,t}(J_Z(u))
 =\frac1M\sum_{s\le M}\sum_{r\bmod h}\sum_{q\in I_r(u)}
 \sum_{d'\in\cD_I(h;M)}\frac{d'^{-it}}{d'}\\
 \times \Li(r+h(q+s))
 \left(\frac{r+h(q+s)}N\right)^{it}
 \Li(r+h(q+s+d'))
 +O(hM V_I(h;M)),
\end{multline}
where \(I_r(u)\) is the quotient-start interval induced by \(J_Z(u)\).
Fourier orthogonality in \(s+d'=\ell\) turns the main term into
\begin{equation}\label{eq:mellin-quotient-fourier}
 \frac1M\sum_{r\bmod h}\sum_{q\in I_r(u)}
 \int_{\mathbb T}Q_{I,h,M,t}(\alpha)
 F_{r,q,t}(\alpha)G_{r,q}(\alpha)\,\dd\alpha,
\end{equation}
where
\begin{align*}
 F_{r,q,t}(\alpha)
 &=\sum_{s\le M}\Li(r+h(q+s))
   \left(\frac{r+h(q+s)}N\right)^{it}e(\alpha s),\\
 G_{r,q}(\alpha)
 &=\sum_{\ell\le2M}\Li(r+h(q+\ell))e(-\alpha\ell).
\end{align*}
Thus the first factor has length \(M\), the second has length at most
\(2M\), and every Mellin coefficient has modulus one.

Put
\[
 \mathcal S_t(\alpha)
 :=\sum_{r\bmod h}\sum_{q\in I_r(u)}
 |F_{r,q,t}(\alpha)G_{r,q}(\alpha)|.
\]
If \(b_1>0\) is the exponent in
Proposition~\ref{prop:middle-scored-apf}, that proposition and the
trivial bound \(\|G_{r,q}\|_\infty\le2M\) give
\begin{equation}\label{eq:mellin-product-linfty}
 \|\mathcal S_t\|_\infty\ll R^{-b_1}M^2Z.
\end{equation}
For each start, Parseval followed by Cauchy--Schwarz gives an \(O(M)\)
integral, and the total number of quotient starts is \(O(Z)\); hence
\begin{equation}\label{eq:mellin-product-lone}
 \|\mathcal S_t\|_1\ll MZ.
\end{equation}
Interpolation yields
\begin{equation}\label{eq:mellin-product-four-thirds}
 \|\mathcal S_t\|_{4/3}
 \ll R^{-b_1/4}M^{5/4}Z.
\end{equation}
H\"older's inequality, \eqref{eq:mellin-fourth-moment}, and the outer
factor \(1/M\) in \eqref{eq:mellin-quotient-fourier} therefore give,
uniformly in the partial endpoint \(u\),
\begin{equation}\label{eq:mellin-unweighted-prefix-bound}
 \sup_u|T_{I,M,t}(J_Z(u))|
 \ll R^{-b_1/4}Z\frac{V^J}{\log M}
      +O(hM V_I(h;M)).
\end{equation}

The harmonic weight now supplies the missing local normalization.  Abel
summation on \(J_Z\subseteq[Z,2Z]\) gives
\begin{multline}\label{eq:mellin-harmonic-partial-summation}
 \left|\sum_{d'\in\cD_I(h;M)}\frac{d'^{-it}}{d'}
 \sum_{k\in J_Z}\frac{\Li(k)\Li(k+hd')}{k}
 \left(\frac{k}{N}\right)^{it}\right|\\
 \leq \frac{3}{Z}\sup_u|T_{I,M,t}(J_Z(u))|.
\end{multline}
Thus the \(Z\) in \eqref{eq:mellin-unweighted-prefix-bound} cancels
exactly.  Multiplying by the complementary factor \(1/d_0\) in
\eqref{eq:mellin-exact-subset-term} and summing gives
\begin{equation}\label{eq:mellin-complementary-mass}
 \sum_{d_0\in\cD_{[J]\setminus I}(h)}\frac1{d_0}
 \leq V^{J-|I|}\leq V^J.
\end{equation}
The sum over dyadic \(M\) costs \(O(\log R)\), and all nonempty subsets
cost at most \(2^JV^{2J}\).  Choose the fixed block parameter so that
\begin{equation}\label{eq:mellin-subset-loss}
 2^JV^{2J}\leq R^{b_1/20}.
\end{equation}
Since \(\log R\leq R^{b_1/40}\), division by
\(W_{\mathrm{Pil},h}\asymp V^J\) leaves
\begin{equation}\label{eq:mellin-normalized-uncentred}
 \frac{|S_{1,h,w}-S_{2,h,w}|}{W_{\mathrm{Pil},h}}
 \ll R^{-b_1/6}\int_{\mathbb R}|\widehat w(t)|\,\dd t.
\end{equation}
The translation boundary in \eqref{eq:mellin-unweighted-prefix-bound}
is smaller than every fixed power of \(R^{-1}\) by
\eqref{eq:mellin-boundary-margin} below.

It remains to bound the centred term \(S_{2,h,w}\).
Mellin inversion gives the full divisor tuple the unit coefficient
\(d^{-it}\) and one Liouville factor the bounded twist \((m/N)^{it}\).
For each retained \(t\), split the unweighted partial parent intervals by
their quotient length.  If that length is at least
\(\exp(R^{2.01})\), the Tao--Ter\"av\"ainen theorem bounds the interval by
\(V^{0.51J}Z\).  If it is shorter, the corresponding physical interval has
length \(\Delta<h\exp(R^{2.01})\), and the trivial centred bound is
\(O(V^J\Delta)\).  After the harmonic factor \(1/Z\) and division by
\(W_{\mathrm{Pil},h}\asymp V^J\), the latter branch is at most
\begin{equation}\label{eq:mellin-short-centred-prefix}
 \frac{\Delta}{Z}
 \leq\frac{h\exp(R^{2.01})}{Y/e^R}
 \ll H^{-1/2}\exp(-8R^{2.01}+R),
\end{equation}
which is smaller than every fixed power of \(R^{-1}\).  Thus the supremum
over all partial endpoints has the same centred bound as the long branch.
Applying
\eqref{eq:mellin-harmonic-partial-summation} removes \(Z\), integration
against \(|\widehat w(t)|\), and division by
\(W_{\mathrm{Pil},h}\asymp V^J\) give
\begin{equation}\label{eq:mellin-normalized-centred}
 \frac{|S_{2,h,w}|}{W_{\mathrm{Pil},h}}
 \ll R^{-c_{\mathrm{dec}}}
       \int_{\mathbb R}|\widehat w(t)|\,\dd t.
\end{equation}
Consequently, with
\begin{equation}\label{eq:mellin-transfer-exponent}
 b_2:=\min\{b_1/6,c_{\mathrm{dec}}\}>0,
\end{equation}
equations \eqref{eq:mellin-normalized-uncentred} and
\eqref{eq:mellin-normalized-centred} prove the normalized slice estimate
required in \eqref{eq:mellin-smooth-transfer}.

The two geometric conditions needed in this calculation hold uniformly
after every complementary rescaling.  Namely, every scored-APF base obeys
\begin{equation}\label{eq:mellin-boundary-margin}
 Z\geq Y/e^R,
 \qquad
 \frac{hM}{Z}\leq\frac{2He^{2R}}Y=o(R^{-A})
 \quad(A>0),
\end{equation}
and every centred quotient interval has
\begin{equation}\label{eq:mellin-tt-margin}
 \log\frac{Y}{2He^R}
 \geq\frac12S+9R^{2.01}-R\gg R^{2.01}.
\end{equation}
Hence all translation endpoints are absorbed by an arbitrary negative
power of \(R\), and the Tao--Ter\"av\"ainen interval-length hypothesis is
satisfied.  Integrating the retained slices against
\(|\widehat w(t)|\), adding the tail
\eqref{eq:mellin-frequency-tail}, and dividing by
\(W_{\mathrm{Pil},h}\) proves \eqref{eq:mellin-smooth-transfer}.

It remains to pass from smooth blocks to every terminal point.  Choose a
fixed smooth dyadic partition of unity in the logarithmic variable.
Complete blocks have uniformly bounded Mellin \(L^1\)-norm.  To specify the
terminal profiles, fix \(\psi\in C^\infty(\mathbb R)\) with
\(0\leq\psi\leq1\), \(\psi(s)=0\) for \(s\leq0\), and \(\psi(s)=1\) for
\(s\geq1\).  With \(\eta=R^{-A_0}\), use
\[
 \psi\!\left(\frac{\log n-\log Y}{\eta}\right)
 \quad\hbox{and}\quad
 \psi\!\left(\frac{\log y-\log n}{\eta}\right)
\]
for the lower and upper transitions.  Their collars are one-sided,
respectively \([Y,e^\eta Y]\) and \([e^{-\eta}y,y]\); hence no lower
transition enters the untreated region below \(Y\), and no upper transition
passes the terminal point.  If \(D=u\,d/du\), the resulting cutoff satisfies
\[
 \|w_\eta\|_{L^1(du/u)}\ll1,\qquad
 \|Dw_\eta\|_1\ll1,\qquad
 \|D^2w_\eta\|_1\ll\eta^{-1}.
\]
The collars contain \(O(\eta y)\) integers of size \(\asymp y\), and
therefore change the harmonic sum by \(O(\eta)\).
Equation~\eqref{eq:mellin-terminal-pointwise-decay} gives
\begin{equation}\label{eq:mellin-terminal-lone}
 \int_{\mathbb R}|\widehat w_\eta(t)|\,\dd t
 \ll\log(1/\eta)\ll\log R.
\end{equation}
The fixed lower endpoint adds only the one one-sided transition just
defined.  There are \(O(\log x)=O(L)\)
complete physical blocks.  They contribute \(O(LR^{-b_2})\), whereas the
\(O(1)\) transition profiles contribute \(O(R^{-b_2}\log R)\).  Taking
\(A_0\) large and absorbing \(\log R\) into \(R^{b_2/2}\) gives
\begin{equation}\label{eq:mellin-high-tail-before-L}
 \sup_{Y\leq y\leq x}
 \left|\sum_{Y<n\leq y}
 \frac{\Li(n)\Li(n+h)}n\right|
 \ll L R^{-b_2/2}.
\end{equation}
In a middle band, \(S\geq L^{1/2}\), so
\(R=S^{1/5}\geq L^{1/10}\).  Thus the right side of
\eqref{eq:mellin-high-tail-before-L} is at most
\(L^{1-b_2/20}\), proving \eqref{eq:middle-high-tail} with
\(c_2=b_2/20\).
\end{proof}

\section{The global sparse-exception theorem}
\label{sec:global-u-proof}

For \(h\geq1\), write
\begin{equation}\label{eq:global-correlation-notation}
 C_h(y):=\sum_{n\leq y}\frac{\Li(n)\Li(n+h)}n,
 \qquad
 A_x(h):=\sup_{1\leq y\leq x}|C_h(y)|.
\end{equation}
We first combine the high-tail estimate of the previous section with the
maximal moment theorem below the high-tail threshold.

\begin{lemma}[Evacuation of the middle-band prefix]
\label{lem:middle-prefix-evacuation}
Let \(H<h\leq2H\) be a middle shift band, and define \(Y\) by
\eqref{eq:mellin-tail-threshold}.  Fix an absolute \(0<c_3<1\), and put
\begin{equation}\label{eq:middle-prefix-exception}
 E_{\mathrm{pre}}(H):=
 \left\{H<h\leq2H:
 \sup_{1\leq y\leq Y}|C_h(y)|>L^{1-c_3}\right\}.
\end{equation}
Then \(E_{\mathrm{pre}}(H)\) is independent of every subsequently
requested sparsity exponent \(C\), and, for every fixed \(C>0\),
\begin{equation}\label{eq:middle-prefix-exception-size}
 |E_{\mathrm{pre}}(H)|\ll_C H L^{-C}.
\end{equation}
\end{lemma}

\begin{proof}
With \(S=\log H\) and \(R=S^{1/5}\),
\[
 \log Y=\frac32S+10R^{2.01}.
\]
For the fixed exponent \(\theta_0=3/5>1/3\),
\begin{equation}\label{eq:middle-prefix-moment-range}
 \log H-\theta_0\log Y
 =\frac1{10}S-6R^{2.01}>0
\end{equation}
for all sufficiently large \(H\).  Thus \(H\geq Y^{\theta_0}\), and
Theorem~\ref{thm:one-third-moments}, with summation range \(Y\) and the
translated shift window \((H,2H]\), gives for every fixed \(p>1\)
\begin{equation}\label{eq:middle-prefix-moment}
 \frac1H\sum_{H<h\leq2H}
 \left(\sup_{1\leq y\leq Y}|C_h(y)|\right)^p
 \ll_{p,\theta_0}1.
\end{equation}
The set in \eqref{eq:middle-prefix-exception} has already been defined
using the fixed threshold \(L^{1-c_3}\).  After a fixed \(C>0\) is given,
choose a fixed
\[
 p>\max\left\{1,\frac{C+1}{1-c_3}\right\}.
\]
Markov's inequality applied to \eqref{eq:middle-prefix-moment} gives
\[
 |E_{\mathrm{pre}}(H)|
 \ll_p H L^{-p(1-c_3)}
 \ll_C H L^{-C},
\]
which proves \eqref{eq:middle-prefix-exception-size}.  The moment order is
used only to estimate the already fixed set, so its choice does not alter
the quantifier order.
\end{proof}

\begin{proof}[Proof of Theorem~\ref{thm:global-u}]
Put \(L=\log x\) and introduce the two fixed crossover scales
\begin{equation}\label{eq:global-crossovers}
 H_{\mathrm{lo}}:=\exp(L^{1/2}),
 \qquad
 H_{\mathrm{hi}}:=x^{2/5}.
\end{equation}

For \(1\leq h\leq H_{\mathrm{lo}}\), apply
Theorem~\ref{thm:super-main} with \(\eta=1/2\).  It supplies one
exceptional set \(E_{\mathrm{lo}}\), fixed before the sparsity exponent is
chosen, which is power-logarithmically sparse in every prefix and outside
which
\begin{equation}\label{eq:global-low-correlation}
 A_x(h)\ll L^{1-c_{\mathrm{lo}}}
\end{equation}
for an absolute \(c_{\mathrm{lo}}>0\).

Next fix a dyadic middle band
\begin{equation}\label{eq:global-middle-band}
 H<h\leq2H,
 \qquad H_{\mathrm{lo}}<H<H_{\mathrm{hi}}.
\end{equation}
Define
\begin{equation}\label{eq:global-middle-exception}
 E_{\mathrm{mid}}(H)
 :=E_{\mathrm{APF}}(H)\cup E_{\mathrm{pre}}(H).
\end{equation}
Both sets on the right are defined by fixed score or correlation
thresholds, before \(C\) is chosen.  Proposition~\ref{prop:middle-scored-apf}
and Lemma~\ref{lem:middle-prefix-evacuation} give, for every fixed
\(C>0\),
\begin{equation}\label{eq:global-middle-exception-size}
 |E_{\mathrm{mid}}(H)|\ll_C H L^{-C}.
\end{equation}
If \(h\notin E_{\mathrm{mid}}(H)\) and \(y\leq Y\), the definition of
\(E_{\mathrm{pre}}(H)\) bounds \(C_h(y)\).  If \(Y<y\leq x\), split at
\(Y\) and apply the same prefix bound together with
Proposition~\ref{prop:mellin-localized-transfer}.  After reducing the
saving exponent slightly to absorb fixed implied constants, this gives
\begin{equation}\label{eq:global-middle-correlation}
 A_x(h)\leq L^{1-c_{\mathrm{mid}}}
 \qquad(h\notin E_{\mathrm{mid}}(H))
\end{equation}
for an absolute \(c_{\mathrm{mid}}>0\).

For a dyadic high band \(H<h\leq2H\) with
\(H\geq H_{\mathrm{hi}}\), fix an absolute
\(0<c_{\mathrm{hi}}<1\) and define the set itself by
\begin{equation}\label{eq:global-high-exception}
 E_{\mathrm{hi}}(H):=
 \{H<h\leq2H:A_x(h)>L^{1-c_{\mathrm{hi}}}\}.
\end{equation}
Theorem~\ref{thm:one-third-moments} applies with the fixed exponent
\(\theta=2/5>1/3\).  For every fixed \(p>1\), it gives
\[
 \frac1H\sum_{H<h\leq2H}A_x(h)^p\ll_p1.
\]
After \(C\) is specified, choose a fixed
\(p>(C+1)/(1-c_{\mathrm{hi}})\).  Markov's inequality gives
\begin{equation}\label{eq:global-high-exception-size}
 |E_{\mathrm{hi}}(H)|\ll_C H L^{-C}.
\end{equation}
Again the set in \eqref{eq:global-high-exception} is independent of the
moment used to estimate it.

Partition the positive shifts up to \(x\) into dyadic bands and retain on
each band the corresponding low, middle, or high exceptional part.  If a
dyadic band crosses either scale in \eqref{eq:global-crossovers}, split it
at that fixed crossover and use the two adjacent estimates.  Let \(\cE_x\)
be the union of the resulting disjoint pieces.  Each constituent set is
fixed before \(C\), so the same is true of \(\cE_x\).

Let \(1\leq H\leq x\), and choose \(k\) so that
\(2^{k-1}<H\leq2^k\).  The exceptional pieces meeting \([1,H]\) have
total length bounded by a geometric series.  Using the low prefix estimate
and \eqref{eq:global-middle-exception-size}--
\eqref{eq:global-high-exception-size}, we obtain
\begin{equation}\label{eq:global-dyadic-stitching}
 \begin{aligned}
 |\cE_x\cap[1,H]|
 &\leq\sum_{j\leq k}
 |\cE_x\cap(2^{j-1},2^j]|\\
 &\ll_C L^{-C}\sum_{j\leq k}2^j
 \ll_C H L^{-C}.
 \end{aligned}
\end{equation}
The possible final band is counted in full; its lower endpoint is greater
than \(H/2\), so this does not change the last bound.  No factor counting
the number of bands appears.

Choose the \(c\) in Theorem~\ref{thm:global-u} smaller than
\(c_{\mathrm{lo}},c_{\mathrm{mid}},c_{\mathrm{hi}}\), and decrease it
once more to absorb their fixed implied constants.  The three ranges in
\eqref{eq:global-crossovers} cover every \(1\leq h\leq x\), so
\eqref{eq:global-corr-intro} follows.  Equation
\eqref{eq:global-dyadic-stitching}, together with the fact that every
threshold set was fixed before the estimating moment or Siegel exponent
was selected, proves \eqref{eq:global-prefix-intro}.  In particular, the
quantifier order is
\begin{equation}\label{eq:global-quantifier-order}
 \exists\cE_x\quad\forall C>0,
\end{equation}
not \(\forall C>0\,\exists\cE_{x,C}\).
\end{proof}

The range \(h\leq x\) contains the new pointwise middle-band argument, but
the shift cutoff can be removed by the translated long-shift moments.

\begin{proof}[Proof of Corollary~\ref{cor:all-positive-shifts}]
Let \(\cE_x^{(0)}\subseteq[1,x]\) be the set from
Theorem~\ref{thm:global-u}, put \(N=\lfloor x\rfloor\), and let
\(c_+=c/2\), where \(c\) is the saving exponent in that theorem.  Set
\begin{equation}\label{eq:all-positive-fixed-threshold}
 T_x:=(\log x)^{1-c_+}.
\end{equation}
After increasing the absolute lower threshold for \(x\),
\eqref{eq:global-corr-intro} implies
\begin{equation}\label{eq:all-positive-old-range}
 A_x(h)\leq T_x
 \qquad(1\leq h\leq N,\ h\notin\cE_x^{(0)}).
\end{equation}

For \(j\geq0\), define the translated dyadic bands
\begin{equation}\label{eq:all-positive-dyadic-bands}
 I_j:=(2^jN,2^{j+1}N]\cap\mathbb N,
 \qquad B_j:=|I_j|=2^jN,
\end{equation}
and define, once and for all,
\begin{equation}\label{eq:all-positive-exception-set}
 \cE_x^+:=\cE_x^{(0)}\cup
 \{h>N:A_x(h)>T_x\}.
\end{equation}
This definition contains neither \(C\) nor a moment exponent.

The translated-window form of Theorem~\ref{thm:long-shift-moments},
with window location \(2^jN\) and length \(B_j\), gives for every fixed
\(p>1\)
\begin{equation}\label{eq:all-positive-band-moment}
 \sum_{h\in I_j}A_x(h)^p
 \ll_p B_j\left(1+\frac{x}{B_j}\right)
 \ll_p B_j,
\end{equation}
uniformly in \(j\), since \(B_j\geq N\geq x/2\).  Given a fixed \(C>0\),
choose a fixed
\[
 p>\max\left\{1,\frac{C+1}{1-c_+}\right\}.
\]
Markov's inequality, \eqref{eq:all-positive-fixed-threshold}, and
\eqref{eq:all-positive-band-moment} yield simultaneously for every
\(j\geq0\)
\begin{equation}\label{eq:all-positive-band-size}
 |\cE_x^+\cap I_j|
 \ll_C B_j(\log x)^{-C}.
\end{equation}
The moment chosen here estimates the fixed set
\eqref{eq:all-positive-exception-set}; it does not redefine it.

If a finite \(H\leq N\), then \eqref{eq:all-positive-prefix-intro} is the
finite-range theorem.  If \(H>N\), choose \(J\geq0\) with
\(2^JN<H\leq2^{J+1}N\).  Only \(I_0,\ldots,I_J\) meet \([1,H]\), and
\begin{align*}
 |\cE_x^+\cap[1,H]|
 &\leq |\cE_x^{(0)}|
       +\sum_{j=0}^J|\cE_x^+\cap I_j|\\
 &\ll_C N(\log x)^{-C}
       +(\log x)^{-C}\sum_{j=0}^JB_j\\
 &\ll_C H(\log x)^{-C}.
\end{align*}
Every finite prefix meets only finitely many translated bands, so the
infinite union in \eqref{eq:all-positive-exception-set} creates no
convergence issue.  Equations \eqref{eq:all-positive-old-range} and
\eqref{eq:all-positive-exception-set} give
\eqref{eq:all-positive-corr-intro}.  This proves the corollary.
\end{proof}

\section{A GRH profile at every shift scale}
\label{sec:grh-all-shifts}

The unconditional theorem of the preceding section permits a very sparse
set of exceptional shifts.  Under the generalized Riemann hypothesis the
character obstruction in the KMT variance theorem disappears.  This gives
an estimate for every shift, including shifts comparable with or larger than
the summation scale.  At those scales the natural statement retains the
harmonic mass below the shift.

Throughout this section, GRH means the generalized Riemann hypothesis for
all primitive Dirichlet \(L\)-functions, including the primitive character
of conductor \(1\).

We now prove Theorem~\ref{thm:grh-all-shifts}, including its quantification
over every integer \(h\ge1\) and every prescribed subpower range.  We also
record the following fixed-polynomial consequence.

\begin{corollary}[Polynomial scale profile]
\label{cor:grh-scale-consequences}
Assume GRH.  For every fixed \(0<\delta<1\),
\begin{equation}\label{eq:grh-polynomial-profile}
 \limsup_{x\to\infty}\frac1{\log x}
 \sup_{1\le h\le x^\delta}|C_h(x)|\le\delta.
\end{equation}
\end{corollary}

We prove the theorem through a coefficient-gap progression estimate.  This
is a new parameter regime: the base of each progression is separated from
its step \(h\), rather than being bounded below by a fixed power of the
global Pilatte parameter.  We therefore verify the KMT and centred
decoupling hypotheses at this scale instead of invoking
Theorem~\ref{thm:outer-apf} as a black box.

\subsection{Parameters and automatic typicality}

Put
\begin{equation}\label{eq:grh-parameters}
 L=\log x,\qquad R=L^{1/4},\qquad
 \epsilon_K=R^{-1/500},\qquad \Omega=R^2,
 \qquad G=10R^{2.01}.
\end{equation}
Thus
\begin{equation}\label{eq:grh-gap-size}
 G=10L^{2.01/4}=o(L).
\end{equation}
Fix \(c_1=1/2000\).  Next choose the second Pilatte block parameter
\(0<\varepsilon_b<1/100\), independently of \(x,h,y\), sufficiently small
that its block system from \eqref{eq:super-block-scales} satisfies
\begin{equation}\label{eq:grh-subset-margin}
 2^JV^{2J}\le R^{c_1/20}
\end{equation}
for all sufficiently large \(R\).  Such a choice is possible because the
exponent of \(R\) on the left tends to zero with \(\varepsilon_b\).  This
choice also gives the strict KMT lower-endpoint margin
\begin{equation}\label{eq:grh-kmt-exponent-margin}
 \frac{1-\varepsilon_b}{200}>\frac1{500}.
\end{equation}
All constants selected below are made after this fixed choice and are
therefore absolute.

\begin{lemma}[Automatic typicality at the coefficient gap]
\label{lem:grh-automatic-typicality}
Let
\begin{equation}\label{eq:grh-divisor-scale}
 \exp(R^{1-\varepsilon_b})\le M\le e^R,
 \qquad q\le\Omega,
 \qquad U\ge M/\Omega^4.
\end{equation}
Suppose \(h\le x\), and let \(d\mid hq\).  For the reduced KMT length and
modulus
\begin{equation}\label{eq:grh-reduced-length-modulus}
 \mathsf L_d=\frac{hU}{d},\qquad
 \mathsf Q_d=\frac{hq}{d},
\end{equation}
the modulus \(\mathsf Q_d\) is
\((\mathsf L_d/\mathsf Q_d)^{\epsilon_K^2}\)-typical in the sense of
KMT.
\end{lemma}

\begin{proof}
The cancellation of \(h\) and \(d\) gives the exact ratio
\begin{equation}\label{eq:grh-ratio-identity}
 \frac{\mathsf L_d}{\mathsf Q_d}=\frac Uq.
\end{equation}
At the smallest permitted \(U\) and the largest permitted \(q\),
\[
 \log(U/q)\ge R^{1-\varepsilon_b}-O(\log R).
\]
Consequently the KMT typicality parameter obeys
\begin{equation}\label{eq:grh-ytyp-lower}
 y_{\rm typ}:=(U/q)^{\epsilon_K^2}
 \ge \exp\!\left\{R^{1-\varepsilon_b-1/250+o(1)}\right\}.
\end{equation}
On the other hand,
\[
 \log\mathsf Q_d\le\log(hq)\le L+O(\log R).
\]
The right side of \eqref{eq:grh-ytyp-lower} dominates
\(1000\log\mathsf Q_d\).  KMT Lemma~9.1 therefore makes every reduced
working modulus typical.  Notice that the criterion compares
\(y_{\rm typ}\), rather than \(\log y_{\rm typ}\), with
\(1000\log\mathsf Q_d\).
\end{proof}

Under the GRH clause of \citet[Theorem~1.5]{klurman2023shortaps}, the KMT
variance theorem applies to every typical modulus in its range.  Thus no
general-good-modulus exceptional set is needed here.  The real-valued
specialization \citet[Corollary~1.6]{klurman2023shortaps} can still display
a minimizing-character term; we remove it next.

\subsection{Twisted Liouville sums under GRH}

\begin{lemma}[Uniform GRH character bound]
\label{lem:grh-twisted-liouville}
Assume GRH.  Fix \(0<\eta<1/10\).  If \(\chi^*\) is primitive of conductor
\(q_*\), then, uniformly for \(X\ge2\) and \(u\in\R\),
\begin{equation}\label{eq:grh-primitive-twist}
 \sum_{n\le X}\Li(n)\chi^*(n)n^{iu}
 \ll_\eta X^{1/2+2\eta}
       \bigl(q_*(2+|u|)\bigr)^\eta.
\end{equation}
If \(\chi\pmod Q\) is induced by \(\chi^*\) and \(Q\le X\), then the
same sum with \(\chi\) in place of \(\chi^*\) is
\(O_\eta(X^{1/2+5\eta})\), uniformly for \(|u|\le X\).
\end{lemma}

\begin{proof}
For the primitive character one has, initially in \(\Re s>1\),
\begin{equation}\label{eq:grh-liouville-dirichlet-series}
 \sum_{n\ge1}\frac{\Li(n)\chi^*(n)n^{iu}}{n^s}
 =\frac{L(2s-2iu,(\chi^*)^2)}{L(s-iu,\chi^*)}.
\end{equation}
Under GRH the denominator has no zero in
\(\Re(s-iu)>1/2\).  Apply Perron's formula with height \(X^2\), shift the
contour to \(\Re s=1/2+\eta\), and use the standard GRH bounds for
\(L^{-1}(s-iu,\chi^*)\) there.  The numerator is evaluated in the
half-plane \(\Re(2s-2iu)=1+2\eta\), where its Dirichlet series is
absolutely convergent.  The vertical integral and the two horizontal
segments give \eqref{eq:grh-primitive-twist}; the extra \(X^\eta\) in the
exponent absorbs the Perron truncation and logarithmic factors.

For completeness, if \(\chi\pmod Q\) is induced by \(\chi^*\), then
\begin{equation}\label{eq:grh-imprimitive-series}
 \sum_{n\ge1}\frac{\Li(n)\chi(n)n^{iu}}{n^s}
 =E_{\chi,Q}(s,u)
   \frac{L(2s-2iu,(\chi^*)^2)}{L(s-iu,\chi^*)},
\end{equation}
where
\begin{equation}\label{eq:grh-imprimitive-euler-factor}
 E_{\chi,Q}(s,u)
 =\prod_{\substack{p\mid Q\\p\nmid q_*}}
   \left(1+\chi^*(p)p^{-(s-iu)}\right).
\end{equation}
On \(\Re s\ge1/2+\eta\),
\begin{equation}\label{eq:grh-euler-factor-bound}
 |E_{\chi,Q}(s,u)|
 \le\prod_{p\mid Q}(1+p^{-1/2-\eta})
 \ll_\eta Q^\eta.
\end{equation}
Repeating the preceding contour argument, then using \(q_*\le Q\le X\)
and \(|u|\le X\), proves the asserted \(O_\eta(X^{1/2+5\eta})\)
bound.  The conductor-one character is included in this argument.
\end{proof}

\begin{lemma}[Deletion of the KMT main term]
\label{lem:grh-kmt-main-deletion}
Assume GRH.  Suppose a KMT call has ambient base \(X_d\), modulus
\(\mathsf Q_d\), and arises from a coefficient-gap progression with
physical base \(Z\ge he^{4G}\) and rational denominator \(q\le\Omega\).
Then every minimizing-character term in KMT Corollary~1.6 is absorbed by
the square-root variance error.
\end{lemma}

\begin{proof}
Every such pair satisfies
\begin{equation}\label{eq:grh-modulus-base-gap}
 \frac{\mathsf Q_d}{X_d}
 \ll\frac{hq}{Z}\le R^2e^{-4G}.
\end{equation}
In particular \(\mathsf Q_d\le X_d\), and
\(X_d\gg e^{4G}/R^2\).  The minimizing twist in KMT lies in
\(|u|\le X_d\).  Lemma~\ref{lem:grh-twisted-liouville}, with
\(\eta=1/20\), therefore gives the deliberately weakened estimate
\begin{equation}\label{eq:grh-character-four-fifths}
 \max_{\chi\bmod\mathsf Q_d}\max_{|u|\le X_d}
 \left|\sum_{n\le3X_d}\Li(n)\chi(n)n^{iu}\right|
 \ll X_d^{4/5}.
\end{equation}
Since
\(\mathsf Q_d/\varphi(\mathsf Q_d)\ll\log\log(3\mathsf Q_d)\) and
\(\epsilon_K^{-1/2}=R^{1/1000}\), the lower bound for \(X_d\) gives
\[
 X_d^{4/5}
 =o\!\left(\epsilon_K^{1/2}
       \frac{\varphi(\mathsf Q_d)}{\mathsf Q_d}X_d\right).
\]
This is precisely the threshold required to absorb the character main term
into the square-root of the KMT variance error.
\end{proof}

\subsection{A coefficient-gap progression Fourier estimate}

The next proposition is the point at which the all-shift scale enters.  It
re-proves the progression estimate with \(Z/h\) exponentially large in
\(R^{2.01}\), but without imposing a fixed-power lower bound on \(\log Z\).

\begin{proposition}[Coefficient-gap APF under GRH]
\label{prop:grh-coefficient-gap-apf}
Assume GRH.  Let \(h\le x\), let \(M\) satisfy
\eqref{eq:grh-divisor-scale}, and suppose
\begin{equation}\label{eq:grh-apf-base-range}
 he^{4G}\le Z\le e^Rx.
\end{equation}
Then
\begin{equation}\label{eq:grh-apf-conclusion}
 \sup_{\alpha\in\T}\mathcal F_{h,M}(\alpha;Z)
 \ll R^{-c_1}MZ.
\end{equation}
The same estimate holds for every consecutive subinterval of \([1,M]\) of
length \(U\ge M/\Omega^4\), uniformly in its endpoints, with \(M\) on the
right replaced by \(U\).
\end{proposition}

\begin{proof}
Use the first Ramar\'e range
\begin{equation}\label{eq:grh-ramare-range}
 P_1=\Omega^{C_M},\qquad Q_1=M/\Omega^3,
\end{equation}
and retain only primes \(p\nmid h\).  The deleted harmonic mass is
\begin{equation}\label{eq:grh-deleted-ramare-primes}
 \sum_{\substack{p\mid h\\p\ge P_1}}\frac1p
 \le\frac{\log h}{P_1\log P_1}+O(P_1^{-1})
 \ll R^{-100},
\end{equation}
because \(\log h\le L=R^4\) and \(C_M\) is the fixed large constant in
\eqref{eq:CM-choice}.  The physical boundary introduced by the translated
typical cutoff has relative size
\begin{equation}\label{eq:grh-apf-boundary-gap}
 \frac{hM}{Z}\le e^{R-4G}\ll_A R^{-A}
\end{equation}
for every fixed \(A>0\).

Suppose first that \(\alpha\) has a reduced approximation \(a_0/q\) with
\[
 \Omega<q\le M/\Omega,
 \qquad |\alpha-a_0/q|\le\frac{\Omega}{Mq}.
\]
Dualize the outer \(\ell^1\)-norm in
\eqref{eq:apf-definition} and use the sparse window
\eqref{eq:sparse-window}.  Because every retained prime is invertible
modulo \(h\), the Ramar\'e expansion and fourth-power calculation in the
proof of Lemma~\ref{lem:apf-minor} apply verbatim at the present
coefficient gap.  More explicitly, fixing the first start determines the
relevant residue of the cofactor modulo \(h\), and each of the other three
starts has \(O(M)\) choices; hence the number of compatible quadruples is
\(O(ZM^3)\), as in \eqref{eq:sparse-tuple-count}.  The remaining cofactor
sum is the geometric sum \eqref{eq:sparse-geometric-sum}.  The four-prime
sieve bound and Vinogradov's lemma therefore give
\[
 \mathcal F_{h,M}(\alpha;Z)
 \ll MZ\frac{(\log M)^{1/4}\log\log M}{\Omega^{1/4}}
      +O(R^{-100}MZ+hM^2).
\]
Equations \eqref{eq:grh-parameters} and
\eqref{eq:grh-apf-boundary-gap} make this
\(O(R^{-1/5}MZ)\).  This establishes the minor-arc estimate without using
the old fixed-\(\gamma\) hypothesis.

It remains to treat a major arc.  First take the exact frequency \(a_0/q\)
with \(q\le\Omega\), and fix a consecutive \(t\)-interval of length
\(U\ge M/\Omega^4\).  Splitting \(t\) into its \(q\) residue classes,
applying Lemma~\ref{lem:endpoint-residue}, and grouping physical residues
by \(d=(c,hq)\) give the reduced data
\begin{equation}\label{eq:grh-major-reduced-data}
 \mathsf Q_d=\frac{hq}{d},\qquad
 \mathsf L_d=\frac{hU}{d},\qquad
 \frac{\mathsf L_d}{\mathsf Q_d}=\frac Uq,\qquad
 \mathcal I_d=\left[\frac{Z-hq}{d},\frac{2Z}{d}\right].
\end{equation}
Here \(\mathcal I_d\) is the exact scaled start interval supplied by the
endpoint--residue lemma.  Since \(hq=o(Z)\), it is covered by at most two
standard dyadic intervals, contained respectively in
\([Z/(2d),Z/d]\) and \([Z/d,2Z/d]\).  On either piece let \(X_d\) denote
its ambient base; then \(X_d\asymp Z/d\), and replacing one exact interval
by these two pieces changes only the absolute constant.
We now discharge every KMT scale condition.  From
\eqref{eq:grh-apf-base-range} and \(U\le M\le e^R\),
\[
 X_d\ge\mathsf L_d,
 \qquad \mathsf L_d\ge10\mathsf Q_d.
\]
Furthermore, by \eqref{eq:grh-kmt-exponent-margin},
\begin{equation}\label{eq:grh-kmt-lower-endpoint}
 (\log(U/q))^{-1/200}
 \le R^{-(1-\varepsilon_b)/200+o(1)}
 \le\epsilon_K.
\end{equation}
Lemma~\ref{lem:grh-automatic-typicality} supplies typicality.  The GRH
clause of KMT Theorem~1.5 applies to every one of these typical moduli,
KMT Corollary~1.6 supplies the real-valued form, and
Lemma~\ref{lem:grh-kmt-main-deletion} removes its character term.

KMT variance followed by Cauchy--Schwarz in the base and in the reduced
residue classes gives, before endpoint--residue normalization,
\begin{equation}\label{eq:grh-kmt-l1-bound}
 d\int\sum_{c\bmod\mathsf Q_d}^{*}|S_{\Li}(v,c;\mathsf L_d,\mathsf Q_d)|
 \,\dd v
 \ll \epsilon_K^{1/2}\varphi(\mathsf Q_d)\frac{ZU}{q}.
\end{equation}
The endpoint--residue lemma contributes \(1/(hq)\).  Hence the exact
identity
\begin{equation}\label{eq:grh-gcd-coefficient-cancellation}
 \frac1{hq}\sum_{d\mid hq}\varphi(hq/d)=1
\end{equation}
shows that one \(t\)-residue costs
\(O(\epsilon_K^{1/2}ZU/q)\), and summing the \(q\) residues costs
\(O(\epsilon_K^{1/2}ZU)\).  The endpoint errors are
\(O(qZ+qhU)\); relative to \(ZU\), they are bounded by
\begin{equation}\label{eq:grh-major-endpoint-errors}
 \frac qU\ll R^2e^{-R^{1-\varepsilon_b}+O(\log R)},
 \qquad \frac{qh}{Z}\le R^2e^{-4G}.
\end{equation}
They are smaller than every fixed negative power of \(R\).

Finally let \(\alpha=a_0/q+\theta\), where
\(|\theta|\le\Omega/(Mq)\).  The rational estimate just proved is uniform
for all partial lengths \(U\ge M/\Omega^4\).  Blockwise Abel summation
therefore costs only \(O(MZ\Omega^{-3})\); the omitted initial block has
the same bound.  Combining the major and minor arcs, and using
\(\epsilon_K^{1/2}=R^{-1/1000}\), proves
\eqref{eq:grh-apf-conclusion} with the fixed weakening
\(c_1=1/2000\).  Translating the \(t\)-interval changes the physical base
by \(O(hM)=o(Z)\), so all displayed hypotheses and estimates remain
uniform.
\end{proof}

\subsection{Transfer above an \(h\)-dependent threshold}

For each shift define
\begin{equation}\label{eq:grh-low-threshold}
 N_h=h\exp(4G+R).
\end{equation}
The factor \(e^R\) in this definition is essential: it compensates for the
largest complementary Pilatte divisor in the subset expansion.

\begin{proposition}[Maximal transfer above \(N_h\)]
\label{prop:grh-maximal-transfer}
Assume GRH.  There is an absolute \(c_2>0\) such that, uniformly for
\(h\le x\) and \(N_h\le y\le x\),
\begin{equation}\label{eq:grh-transferred-bound}
 |C_h(y)|\le\log N_h+O(LR^{-c_2}).
\end{equation}
\end{proposition}

\begin{proof}
Use the second Pilatte system fixed after
\eqref{eq:grh-subset-margin}, and prune each block to
\[
 \cP_j(h)=\{p\in\cP_j:p\nmid h\}.
\]
Since every block prime exceeds
\(H_0=\exp(R^{1-\varepsilon_b})\),
\begin{equation}\label{eq:grh-pilatte-pruning}
 \sum_{\substack{p\mid h\\p>H_0}}\frac1p
 \le\frac{L}{H_0\log H_0}\ll_A R^{-A}
\end{equation}
for every fixed \(A>0\).  Lemma~\ref{lem:coprime-block-pruning} therefore
gives
\begin{equation}\label{eq:grh-pruned-normalization}
 W_{\mathrm{Pil},h}=(1+O_A(R^{-A}))W_{\mathrm{Pil}}
 \asymp V^J,
\end{equation}
and the pruned fourth-moment estimate remains valid.

We spell out why the new threshold is compatible with every subset
rescaling.  In the logarithmic-dilation identity
\eqref{eq:pruned-logarithmic-dilation}, split the common physical prefix at
\(N_h\), and decompose the part above \(N_h\) into intervals of
multiplicative width at most two.  For a complementary divisor
\(d_0\le e^R\), every APF base arising from such an interval satisfies
\begin{equation}\label{eq:grh-rescaled-base}
 Z\ge\frac{N_h}{e^R}=he^{4G}.
\end{equation}
Thus Proposition~\ref{prop:grh-coefficient-gap-apf} applies after every
subset rescaling.  Moreover, with every active divisor \(M\le e^R\),
\begin{equation}\label{eq:grh-translation-ratio}
 \frac{hM}{Z}\ll e^{2R-4G},
\end{equation}
so quotient translations and interval endpoints are negligible.  The
quotient interval available to the centred term has logarithmic length at
least
\begin{equation}\label{eq:grh-tt-length}
 \log\frac{N_h}{he^R}=4G>R^{2.01}.
\end{equation}

For clarity, consider one nonempty active subset \(I\subseteq[J]\) and one
dyadic divisor block \(\cD_I(h;M)\).  After writing \(n=r+hq\), translate
the quotient \(q\), not the physical variable \(n\).  Fourier inversion
then uses
\begin{align*}
 F_{r,q}(\alpha)&=\sum_{t\le M}\Li(r+h(q+t))e(\alpha t),\\
 G_{r,q}(\alpha)&=\sum_{k\le2M}\Li(r+h(q+k))e(-\alpha k),\\
 Q_{I,h,M}(\alpha)&=\sum_{d\in\cD_I(h;M)}\frac{e(\alpha d)}d.
\end{align*}
The complementary fixed factors carry their original harmonic weight, and
\begin{equation}\label{eq:grh-complementary-mass}
 \sum_{d_0\in\cD_{[J]\setminus I}(h)}\frac1{d_0}
 \leq V^{J-|I|}\leq V^J.
\end{equation}
In particular \(\lVert G_{r,q}\rVert_\infty\le2M\), with no factor \(h\).
If
\[
 \mathcal S(\alpha)=\sum_{r\bmod h}\sum_q
 |F_{r,q}(\alpha)G_{r,q}(\alpha)|,
\]
then Proposition~\ref{prop:grh-coefficient-gap-apf} and Parseval give
\begin{equation}\label{eq:grh-product-norms}
 \lVert\mathcal S\rVert_\infty\ll R^{-c_1}M^2Z,
 \qquad
 \lVert\mathcal S\rVert_1\ll MZ.
\end{equation}
Interpolation yields
\begin{equation}\label{eq:grh-product-four-thirds}
 \lVert\mathcal S\rVert_{4/3}
 \ll R^{-c_1/4}M^{5/4}Z.
\end{equation}
Pairing this with the pruned divisor fourth moment and the outer factor
\(1/M\) gives the fourth-root uncentring saving for one fixed complementary
divisor.  Multiplying by \(1/d_0\), summing
\eqref{eq:grh-complementary-mass}, then summing dyadic \(M\) and all
nonempty subsets costs at most \(2^JV^{2J}\).  After division by
\(W_{\mathrm{Pil},h}\asymp V^J\), the reserved margin in
\eqref{eq:grh-subset-margin} leaves the normalized contribution
\begin{equation}\label{eq:grh-uncentred-transfer}
 O(R^{-c_1/6}Z).
\end{equation}

For the centred contribution, every retained block prime is coprime to
\(h\), and hence
\begin{equation}\label{eq:grh-centred-residue}
 p\mid r+hq
 \quad\Longleftrightarrow\quad
 q\equiv-rh^{-1}\pmod p.
\end{equation}
Assign arbitrary residues to the deleted primes.  The outer sum over
retained tuples is a sub-sum of the nonnegative outer sum in
Input~\ref{input:tt-decoupling}, so pruning does not require a strengthened
version of that input.  Equation \eqref{eq:grh-tt-length} supplies its
interval hypothesis, and its arbitrary \(1\)-bounded functions absorb the
translated Liouville factors.  Thus the centred term is
\(O(V^{0.51J}Z)\).  After division by
\(W_{\mathrm{Pil},h}\asymp V^J\), this is
\begin{equation}\label{eq:grh-centred-transfer}
 O(R^{-c_{\rm dec}}Z)
\end{equation}
for an absolute \(c_{\rm dec}>0\).

Set
\begin{equation}\label{eq:grh-c2-definition}
 c_2=\min\{c_1/6,c_{\rm dec}\}>0.
\end{equation}
The estimates above are uniform in the terminal endpoint.  In the dilation
identity, the part \(y<m\le dy\) contributes
\(O(W_{\mathrm{Pil},h}R)\).  The remaining common prefix \(m\le y\) is
split at \(N_h\): below it we use the trivial divisor estimate, and above
it we apply \eqref{eq:grh-uncentred-transfer} and
\eqref{eq:grh-centred-transfer} on every multiplicative block.  Partial
summation and \eqref{eq:grh-pruned-normalization} give
\[
 |C_h(y)|\le\log N_h+O(R+LR^{-c_2}).
\]
Since \(R=O(LR^{-c_2})\), this is
\eqref{eq:grh-transferred-bound}.
The use of a common physical prefix \(m\le y\) is important: no
divisor-dependent localized tail, and hence no Mellin separation, is
needed in this conditional transfer.
\end{proof}

\subsection{Proof of the all-shift profile}

\begin{proof}[Proof of Theorem~\ref{thm:grh-all-shifts}]
It is enough to consider sufficiently large \(x\), since the remaining
range is absorbed by enlarging the absolute implied constant.  Fix
\(h\in\N\) and \(1\le y\le x\).

If \(y\le h\), the trivial harmonic estimate gives
\begin{equation}\label{eq:grh-endpoint-low}
 |C_h(y)|\le\sum_{n\le y}\frac1n
 \le\log(2y)+O(1).
\end{equation}
Suppose next that \(h<y<N_h\).  Then
\begin{equation}\label{eq:grh-endpoint-middle}
 |C_h(y)|\le\log y+O(1)
 \le\log h+4G+R+O(1).
\end{equation}
Finally, if \(y\ge N_h\), then necessarily \(h\le x\), and
Proposition~\ref{prop:grh-maximal-transfer} gives
\begin{equation}\label{eq:grh-endpoint-high}
 |C_h(y)|\le\log h+4G+R+O(LR^{-c_2}).
\end{equation}
These three cases are exhaustive.  In particular, when \(h>x\) only the
first case can occur.  When \(xe^{-4G-R}<h\le x\), the nontrivial high
segment is empty because \(N_h>x\), and the first two cases suffice.

Choose any fixed
\begin{equation}\label{eq:grh-cG-choice}
 0<c_{\mathrm G}<
 \min\left\{1-\frac{2.01}{4},\frac{c_2}{4}\right\}.
\end{equation}
By \eqref{eq:grh-parameters}--\eqref{eq:grh-gap-size},
\[
 G\ll L^{2.01/4},\qquad R=L^{1/4},\qquad
 LR^{-c_2}=L^{1-c_2/4}.
\]
Thus \eqref{eq:grh-endpoint-low}--\eqref{eq:grh-endpoint-high} imply
\eqref{eq:grh-profile-intro}, uniformly in \(h\) and \(y\).

Finally, let \(H(x)\ge1\) be any prescribed function with
\(\log H(x)=o(\log x)\).  It is eventually at most \(x\), and the estimate
just proved gives
\[
 \sup_{1\le h\le H(x)}|C_h(x)|
 \le\log(2H(x))+O((\log x)^{1-c_{\mathrm G}})=o(\log x).
\]
This is the final assertion of Theorem~\ref{thm:grh-all-shifts}.
\end{proof}

\begin{proof}[Proof of Corollary~\ref{cor:grh-scale-consequences}]
If \(h\le x^\delta\), Theorem~\ref{thm:grh-all-shifts} gives
\[
 |C_h(x)|\le\delta\log x+O((\log x)^{1-c_{\mathrm G}}).
\]
Taking the supremum, dividing by \(\log x\), and passing to the limsup
proves \eqref{eq:grh-polynomial-profile}.
\end{proof}

\section{Consequences and the remaining pointwise boundary}
\label{sec:global-consequences}

The all-prefix form of Theorem~\ref{thm:global-u} also makes the
exceptional set negligible for harmonic averaging.

\begin{corollary}[Harmonic sparsity]\label{cor:global-harmonic-sparsity}
For the set \(\cE_x\) in Theorem~\ref{thm:global-u}, and for every fixed
\(A>0\),
\[
 \sum_{\substack{h\le x\\h\in\cE_x}}\frac1h
 \ll_A(\log x)^{-A}.
\]
\end{corollary}

\begin{proof}
Write \(E(H)=|\cE_x\cap[1,H]|\).  Apply
\eqref{eq:global-prefix-intro} with exponent \(A+2\), and use partial
summation:
\[
 \sum_{\substack{h\le x\\h\in\cE_x}}\frac1h
 =\frac{E(x)}x+\int_1^x\frac{E(t)}{t^2}\,\dd t
 \ll_A(\log x)^{-A-2}(1+\log x).
\]
\end{proof}

Within the current proof architecture, the exceptional set is not removed
by a numerical retuning of parameters alone.  Two independent pointwise
inputs would be needed.
First, the middle-range KMT argument is unconditional only on a general good
set of moduli and can retain a Landau--Page character.  For a primitive real
character \(\chi\),
\[
 \sum_{n\ge1}\frac{\Li(n)\chi(n)}{n^s}
 =\frac{L(2s,\chi^2)}{L(s,\chi)},
\]
so a real zero of the denominator is visible in the Liouville twist rather
than being a notational artifact.  Second, prefix evacuation and the long
shift range use all fixed moments.  Emptying their exceptional sets would
require an endpoint, pointwise estimate for shifts comparable with the
summation scale.  Such an estimate is not supplied by the current
short-interval Fourier theorems.  GRH removes the modulus obstruction but
not the elementary harmonic mass below \(h\), which explains the
\(\log(2\min\{h,y\})\) term in Theorem~\ref{thm:grh-all-shifts}.

The results here concern logarithmically weighted correlations.  They do not
imply the ordinary Ces\`aro two-point Chowla conjecture, whose fixed-shift
case remains open.

\paragraph{Acknowledgement.}
Codex GPT-5.6-Sol assisted with mathematical exploration, proof development
and checking, literature organization, and LaTeX preparation.  The author
takes full responsibility for the content.

\appendix
\section{Sources and proof interfaces}\label{app:dependency-ledger}

For reference, this appendix lists the external inputs and the precise
interfaces at which they enter the proof.  In particular, the fixed-power
flexibility in Proposition~\ref{prop:flexible-menon} is established locally.

At the time of writing, the KMT, Pilatte, and almost-all short-interval
Fourier-uniformity inputs cited below have appeared in refereed journals.
The Tao--Ter\"av\"ainen input is arXiv:2512.01739v2
and remains a preprint.  Menon's arXiv:2607.15574v1 is also a preprint; we
use its printed sieve and character-twisted short-interval interfaces.  The harmonic
multiplier theorem itself uses neither recent preprint.
The global sparse-exception theorem uses Menon only through the
typical-set complement in Proposition~2.13; the character-twisted
Corollary~3.7 enters the separate fixed-polylogarithmic theorem.

\begin{center}
\small
\begin{tabularx}{\textwidth}{@{}>{\raggedright\arraybackslash}p{0.20\textwidth}>{\raggedright\arraybackslash}p{0.27\textwidth}X@{}}
\toprule
Interface & source & use in this paper \\
\midrule
Pilatte block geometry and harmonic masses & \citet[Lemma~2.3]{pilatte2025improved} & \eqref{eq:block-geometry}--\eqref{eq:block-masses}; separation, normalization, and dilation \\
Pilatte fourth moment & \citet[Lemma~C.2]{pilatte2025improved} & Input~\ref{input:fourth-moment} and \eqref{eq:Qh-fourth} \\
Typical-set complement & \citet[Proposition~2.13]{menon2026short} & Input~\ref{input:sieve-complement} and \eqref{eq:complement-bound} \\
Character-twisted short intervals & \citet[Proposition~3.1 and Corollary~3.7]{menon2026short} & Input~\ref{input:menon-character}; major arcs after the residue-gcd split, with the auxiliary last-range parameter fixed first \\
MRT minor-arc machinery & \citet[Section~3]{matomaki2015averaged} & Lemma~\ref{lem:weighted-minor} replays the argument and exposes every circle-cutoff dependency \\
Centred arbitrary-interval estimate & \citet[Theorem~3.3]{taoteravainen2026quantitative} & Input~\ref{input:tt-decoupling} and \eqref{eq:one-residue-centered} \\
Hybrid short-progressions variance and good-modulus structure & \citet[Theorem~1.5, Corollaries~1.6--1.7, Lemmas~7.4, 7.9, 8.1, 8.2, 9.1, and Proposition~9.4]{klurman2023shortaps} & \eqref{eq:kmt-specialized-variance}; the subpower rational major arc; Sections~\ref{sec:general-good-deletion}--\ref{sec:scored-middle-apf}; and the GRH interface in Section~\ref{sec:grh-all-shifts} \\
Davenport M\"obius exponential sums & \citet{davenport1937infiniteII} & the explicit square-divisor transfer in \eqref{eq:davenport-liouville}; dyadic harmonic multiplier decay and the long-shift moment theorem \\
Almost-all maximal short-interval Fourier uniformity for M\"obius & \citet[Theorem~1.1(i)]{matomaki2026higherII} & Lemma~\ref{lem:almost-all-liouville-fourier}; after the square-divisor transfer, the localized Toeplitz estimate and Theorem~\ref{thm:one-third-moments} \\
\bottomrule
\end{tabularx}
\end{center}

The following steps are proved in the present paper:
\begin{enumerate}[label=\(\roman*\)]
\item the weighted minor-arc estimate for every fixed logarithmic cutoff
      exponent, Lemma~\ref{lem:weighted-minor};
\item the localized residue-uniform major-arc argument and the resulting
      Proposition~\ref{prop:flexible-menon};
\item the one-sided typical truncation and direct fourth-root frequency
      interpolation in
      Section~\ref{sec:uncentring};
\item subset rescaling followed by the centred interval theorem in
      Section~\ref{sec:centering};
\item the parameter cancellation and all-scales logarithmic dilation in
      Section~\ref{sec:main-proof}.
\item the joint progression-of-starts Ramar\'e expansion and sparse fourth
      moment in Section~\ref{sec:progression-fourier};
\item the endpoint--residue decoupling, the Liouville adaptation of the KMT
      main-term argument, and the one-conductor Landau--Page splice in
      Section~\ref{sec:exceptional-modulus};
\item the fourth-root quotient-translation transfer, the coprime Pilatte pruning, and
      the complete \(\eta<1\) parameter construction in
      Section~\ref{sec:superpolylog-transfer}.
\item the binary-shell maximal harmonic multiplier estimate, its uniform
      $\ell^p$ interpolation, and the long-shift maximal all-moment theorem;
\item the integer-start and square-divisor transfer from the published
      M\"obius theorem, the translated local-block decomposition, its
      binary maximal refinement, and the maximal all-moment theorem beyond
      exponent $1/3$;
\item the robust-band general-good Liouville deletion, exceptional-zero
      Stieltjes localization, and linear bad-modulus score in
      Section~\ref{sec:general-good-deletion};
\item the scored partial-prefix major arc and its Abel transfer in
      Section~\ref{sec:scored-middle-apf};
\item the Mellin-localized divisor-dependent dilation, including the exact
      harmonic and Pilatte normalization, in
      Section~\ref{sec:mellin-localized-transfer};
\item prefix evacuation, low/middle/high band stitching, and the
      all-positive-shift corollary in Section~\ref{sec:global-u-proof};
\item under GRH, the twisted Liouville--character Perron estimate,
      all-moduli coefficient-gap APF, and scale profile in
      Section~\ref{sec:grh-all-shifts}.
\end{enumerate}

All localizations preserve the support conditions of the quoted results.
The minor arcs use their original ambient scale and localize only through a
bounded dual test function.  Each character estimate is applied on a single interval
\([Y,2Y]\), where arithmetic membership in \(\cT\) agrees with the finite
local set of the source.

For the KMT input, the paper checks separately the lower endpoint for
\(\epsilon_K\), membership in the automatic good-modulus range, and
typicality.  The last of these uses KMT's Lemma~9.1 in its actual form
\(y\ge1000\log Q\), rather than the stronger comparison
\(\log y>\log Q\).
The nonprincipal Liouville main term follows by a source-level adaptation of
KMT's prime pretentious-distance argument.  At zero twist the principal
character is treated by the prime number theorem for Liouville; at nonzero
twist Section~\ref{sec:general-good-deletion} retains the principal pole and
uses its cosine-integral contribution.  Landau--Page and
Siegel are used only to organize and count the possible exceptional
conductor, not to assert its existence.  The fixed stretched-logarithmic
theorem uses one global conductor across its complete outer family.

\section{Parameter and endpoint ledger}\label{app:parameter-ledger}

The proof has two unrelated families of parameters: Pilatte's prime-block
parameters and the Menon circle cutoff.  Table~\ref{tab:parameter-ledger}
records the choices and the inequalities they discharge.

\begin{table}[h]
\centering
\caption{Parameter ledger for a fixed exponent \(A>0\).}
\label{tab:parameter-ledger}
\small
\begin{tabularx}{\textwidth}{@{}>{\raggedright\arraybackslash}p{0.19\textwidth}>{\raggedright\arraybackslash}p{0.29\textwidth}X@{}}
\toprule
parameter & definition / range & role \\
\midrule
\(\eps_1\) & one sufficiently small absolute constant & ensures \eqref{eq:subset-loss} and fixes the absolute decoupling exponent \\
\(H,R,H_0\) & \(R=\log H\), \(H_0=\exp(R^{1-\eps_1})\) & Pilatte prime scales \\
\(J\) & \(\lfloor\eps_1^2\log R\rfloor\) & number of Pilatte blocks \\
\(h\) & \(1\le h\le R^{3A}\) & equivalent to \(h\le(\log x)^A\) after \(\log x=R^3\) \\
\(N\) & \(\log N\ge R^{2.1}\) & nontrivial dyadic outer scales \\
\(d_0\) & \(d_0\le H\) & fixed factor in the subset expansion \\
\(Z\) & \(N/d_0\), so \(\log Z\ge R^{2.1}-R\) & local base for Proposition~\ref{prop:flexible-menon} \\
\(M\) & \(H_0\le M\le H\), dyadic & divisor-product scale \\
\(K\) & \(\lceil hM\rceil\) & integer translation length \\
\(\WM\) & \(R^{12A+6}\) & flexible Menon circle cutoff \\
\(B_0\) & \(6A+3\) & fixed logarithmic-modulus exponent \\
\bottomrule
\end{tabularx}
\end{table}

The saving exponents used after the fourth-root interpolation are
\begin{align*}
 h^{1/4}\WM^{-1/16}R^{1/16}
 &\le R^{3A/4-(12A+6)/16+1/16}=R^{-5/16},\\
 R^{-5/16}(\log R)^{1/4}
 &\ll R^{-3/10},\\
 2^JV^{2J}R^{-3/10}
 &\le R^{-1/4},\\
 R^{-1/4}\log x\quad\text{with }\log x=R^3
 &= (\log x)^{1-1/12}.
\end{align*}
Thus the dyadic H\"older step, the subset expansion, and the final change
of variables each have an explicit and separate exponent entry.

For clarity, the four hypotheses of Proposition~\ref{prop:flexible-menon}
are discharged as follows:
\begin{align*}
 (\log K)^5&\le\WM, &&\text{by }\log K\le R+O_A(\log R),\\
 \WM&\le K^{1/(2C_M)}, &&\text{because }\log\WM=O_A(\log R)=o(\log K),\\
 \WM&\le(\log Z)^{B_0}, &&\text{by }12A+6<2.1(6A+3),\\
 K&\le Z, &&\text{by }R+O_A(\log R)<R^{2.1}-R.
\end{align*}
The truncation scale satisfies \(K<K_0\) by \eqref{eq:K0-large}, so the
first Menon prime range is always
\(P_1=\WM^{C_M}\), \(Q_1=K/\WM^3\) in the application.

We finish with endpoint cases used implicitly in the main text.
\begin{itemize}
\item If \(b=0\) in the residue split, then \(d=q\) and \(q_0=1\);
      character orthogonality reduces to the trivial character and
      \eqref{eq:exact-divisor-count} remains valid.
\item If \(L=Y\) in the major arcs, all local summands satisfy
      \(m\le y+L\le3Y\), exactly the support in the published character
      estimate.
\item Proposition~\ref{prop:flexible-menon} also covers the formal case
      \(K>K_0\) by \eqref{eq:prime-range-case2}; in the Chowla application,
      \eqref{eq:K0-large} gives the stronger \(K<K_0\).
\item Replacing real interval endpoints by integer endpoints changes only
      the endpoint cells already included in \eqref{eq:discrete-flexible-menon}
      and the translation error in \eqref{eq:translation-average}.
\item The major/minor partition is exhaustive because
      \eqref{eq:dirichlet-flexible} always provides a denominator
      \(1\le q\le K/\Omega\), and the two cases are \(q\le\Omega\) and
      \(q>\Omega\).
\end{itemize}

\subsection*{Super-polylogarithmic parameter family}

The super-polylogarithmic theorem uses a second Pilatte block system,
independent of the absolute \(\eps_1\)-system in the first part of the paper.
For Theorem~\ref{thm:super-main}, fix \(0<\eta<1\) first and put
\begin{equation}\label{eq:appendix-super-parameters}
\begin{gathered}
 \beta=5,\qquad \rho=5\eta,\qquad
 \gamma=\frac52(1+\eta),\qquad
 a=\frac{1-\eta}{1000},\\
 \sigma=\frac a5,\qquad
 \epsilon_K=R^{-a},\qquad \Omega=R^2.
\end{gathered}
\end{equation}
Only after these quantities have been fixed is the second block parameter
\(\varepsilon_b\) chosen sufficiently small that
\begin{equation}\label{eq:appendix-epsb-choice}
 \varepsilon_b<1/100,\qquad
 2^JV^{2J}\le R^{a/100}=R^{\sigma/20}.
\end{equation}
For each shift \(h\), primes dividing \(h\) are then removed from this second
block system as in Lemma~\ref{lem:coprime-block-pruning}.

\begin{center}
\small
\begin{tabularx}{\textwidth}{@{}>{\raggedright\arraybackslash}p{0.27\textwidth}>{\raggedright\arraybackslash}p{0.26\textwidth}X@{}}
\toprule
gate & required inequality & discharge \\
\midrule
KMT lower endpoint & \(a<(1-\varepsilon_b)/200\) & \(a<1/1000\), \(\varepsilon_b<1/100\) \\
automatic typicality & \(y_{\rm typ}\ge1000\log(hq)\) & \(\log y_{\rm typ}\ge R^{1-\varepsilon_b-2a+o(1)}\), hence \(y_{\rm typ}=\exp(R^{1-\varepsilon_b-2a+o(1)})\gg R^\rho\) \\
good moduli and one conductor & \(\rho<\gamma-200a\) & \(\gamma-200a-\rho=23(1-\eta)/10\) \\
block coprimality & retained \(p\nmid h\) & Lemma~\ref{lem:coprime-block-pruning}; deleted mass is \(R^{-A}\) for every fixed \(A\) \\
centred interval length & \(\gamma>2.01\), \(\rho<\gamma\) & \(\gamma>5/2\), \(\gamma-\rho=5(1-\eta)/2\) \\
outer low-prefix saving & \(\beta>\gamma\) & \(\beta-\gamma=5(1-\eta)/2\) \\
subset expansion & \(2^JV^{2J}\le R^{\sigma/20}\) & \eqref{eq:appendix-epsb-choice} \\
final shift exponent & \(\rho/\beta=\eta\) & \(\rho=5\eta\), \(\beta=5\) \\
\bottomrule
\end{tabularx}
\end{center}

The global KMT base range must be written, when \(\rho\) may exceed one, as
\[
 \log X'\ge R^\gamma-R^\rho-O(R)=R^\gamma(1+o(1)),
\]
not as \(R^\gamma-O(R)\).  This is sufficient for the zero-free-region,
good-modulus, and Landau--Page arguments because \(\gamma>\rho\).

The order of limits is
\[
 \eta\ \longrightarrow\ (\beta,\rho,\gamma,a,\sigma)\
 \longrightarrow\ \varepsilon_b\
 \longrightarrow\ C\ \longrightarrow\ R\to\infty.
\]
The fixed-polylogarithmic theorem retains the separate absolute
\(\eps_1\)-system from Section~\ref{sec:inputs}.  No displayed parameter may
depend on \(R\).  At \(\eta=1\), the relation \(\rho/\beta=1\) conflicts
with the strict scale chain \(\rho<\gamma<\beta\); this is the endpoint of
the present power-scale parameter family.  It is not the endpoint of the
paper's final shift range: Theorem~\ref{thm:global-u} crosses it through the
different middle-band mechanism recorded in
Appendix~\ref{app:global-parameter-ledger}.

\subsection*{Long-shift maximal moments}

Theorem~\ref{thm:long-shift-moments} uses no Pilatte, KMT, Landau--Page, or
Menon parameters.  The binary maximal decomposition takes a square root of
the dyadic Fourier decay.  Since Davenport supplies every fixed logarithmic
power, for each fixed $1<p<\infty$ we choose the resulting exponent
$A>1/\vartheta_p$, where
\[
 \vartheta_p=\frac{2}{\max\{p,p/(p-1)\}}.
\]
Then the dyadic maximal-operator norms in \eqref{eq:block-lp-op} are
summable, giving an $X$-uniform maximal $\ell^p$ bound.  For the tail statement on windows
$H\ge x/(\log x)^B$, after fixed $B,\varepsilon,C$ are chosen, take any
fixed
\[
 p>\max\{1,(B+C)/\varepsilon\}.
\]
Markov's inequality then gives an exceptional density
$O((\log x)^{-C})$ at height $(\log x)^\varepsilon$.

For Theorem~\ref{thm:one-third-moments}, fix $\theta>1/3$ and choose
\[
 0<\rho<\min\left\{\frac{\theta-1/3}{4},\frac1{12}\right\}.
\]
On a dyadic shell $(N,2N]$ above the shift scale, the physical block
length is
\[
 K_N=\left\lfloor\min\{H/4,N^{1-\rho}\}\right\rfloor.
\]
The inequalities $H\ge x^\theta$ and $N\le x$ imply, for all sufficiently
large $x$,
\[
 N^{1/3+\rho}\le K_N\le N^{1-\rho},
\]
which is precisely the fixed-margin range needed in
Lemma~\ref{lem:almost-all-liouville-fourier}.  For a target output power
$(\log x)^{-D_0}$, choose the logarithmic Fourier exponent $A_1$ and the
bad-start exponent $B_1$ so that
\[
 A_1\vartheta_p>D_0+2,
 \qquad B_1>D_0+2.
\]
The source Fourier exponent is taken twice as large before the binary
maximal decomposition.  Then the good blocks, bad blocks, and two
translated-grid boundary pieces
in \eqref{eq:local-operator-norm-sum} contribute respectively
$(\log N)^{-A_1\vartheta_p}$, $(\log N)^{-B_1}$, and
$O(N^{-\rho})$.  Taking $D_0=A+2$ for any prescribed tail exponent
$A>0$ absorbs the $O(\log x)$ dyadic shells and proves
\eqref{eq:one-third-high-tail-intro}.  No KMT, Landau--Page,
Tao--Ter\"av\"ainen, or Menon input enters
this theorem; its only recent input is the refereed Inventiones theorem
\citet[Theorem~1.1(i)]{matomaki2026higherII}.

\section{Global-range parameter and quantifier ledger}
\label{app:global-parameter-ledger}

This appendix collects the parameters used only in
Theorem~\ref{thm:global-u}.  They are distinct from the fixed-
\(\eta\) power-scale family in Appendix~\ref{app:parameter-ledger}.

Write \(L=\log x\).  The shift axis is divided at
\[
 H_{\rm lo}=\exp(L^{1/2}),\qquad H_{\rm hi}=x^{2/5}.
\]
The low range uses Theorem~\ref{thm:super-main} at the single fixed value
\(\eta=1/2\).  The high range uses
Theorem~\ref{thm:one-third-moments} at the single fixed value
\(\theta=2/5\).  Neither endpoint varies with \(x\).

For a middle dyadic band \(H<h\le2H\), put
\[
 S=\log H,\qquad R=S^{1/5},\qquad
 \epsilon=R^{-1/500},\qquad \Omega=R^2,
\]
and
\[
 Y=H^{3/2}\exp(10R^{2.01}).
\]
Since \(S\ge L^{1/2}\), every fixed negative power of \(R\) is a fixed
negative power of \(L\).  The two elementary geometric margins are
\[
 \frac{he^{2R}}Y\ll H^{-1/2}e^{-R^{2.01}},
 \qquad
 \log(Y/h)=\frac12S+10R^{2.01}\gg R^{2.01}.
\]
They respectively control translations and the centred interval length.

The exact KMT interface uses
\[
 \mathcal Q_{X',\epsilon^{6.6},\epsilon^{-88}},
\]
because Proposition~9.4 is used with source parameter
\(\epsilon^{1.1}\).  The conductor split and low-conductor prime scale are
\[
 q_*\gtrless (2X)^{\epsilon^{132}},\qquad
 P\ge X^{\epsilon^7},\qquad
 \vartheta=2\epsilon^{125}.
\]
The identities
\[
 (\epsilon^{6.6})^{20}=\epsilon^{132},\qquad
 125+7=132
\]
are responsible for these exponents.  With
\(ar=(1/500)(1/5)=1/2500\), the generator and entropy powers satisfy
\[
 1-132ar>88ar,\qquad 220ar<1,\qquad 6.6ar<1/20.
\]
These inequalities discharge the log-free zero-density count, the
Lemma~8.1 range, and the high-conductor lower parameter.

For every Abel prefix \(U\ge M/\Omega^4\) and \(q\le\Omega\),
\[
 \log(U/q)\ge R^{1-\varepsilon_b}-O(\log R),
\]
and hence
\[
 (\log(U/q))^{-1/200}\le\epsilon^{1.1},
 \qquad
 (U/q)^{\epsilon^2}\gg1000\log(hq).
\]
Thus both the source-parameter lower endpoint and automatic typicality hold
at the shortest prefix.

If \(b_1>0\) is the scored APF exponent, choose the second fixed Pilatte
block parameter so that
\[
 2^JV^{2J}\le R^{b_1/20}.
\]
The fourth-root transfer, dyadic divisor sum, and subset expansion then leave
\(R^{-b_1/6}\).  The centred term leaves \(R^{-c_{\rm dec}}\), and the
Mellin terminal cutoff costs only \(O(\log R)\).  Thus one may take
\[
 b_2=\min\{b_1/6,c_{\rm dec}\},\qquad
 c_{\rm mid}=\min\{c_3,b_2/20\}>0.
\]
The final global exponent is any fixed positive number below
\[
 \min\{c_{\rm lo},c_{\rm mid},c_{\rm hi}\}.
\]

The exceptional sets are always defined by fixed thresholds before a
sparsity exponent is requested.  For the prefix and high bands, the moment
order \(p=p(C)\) is selected only to estimate that already fixed set.  For
the Landau--Page traces, the Siegel exponent is likewise selected only after
\(C\).  The order is therefore
\[
 \begin{aligned}
 &\text{fix all analytic parameters and all threshold sets}\\
 &\qquad\longrightarrow x\to\infty
 \longrightarrow \text{estimate the same sets for each fixed }C,
 \end{aligned}
\]
which yields \(\exists\cE_x\,\forall C\), rather than a family of sets
depending on \(C\).

Finally, the obstruction \(\rho<\gamma<\beta\) at \(\eta=1\) recorded in
Appendix~\ref{app:parameter-ledger} remains a genuine endpoint of that
single pointwise power-scale family.  Theorem~\ref{thm:global-u} does not
violate the ledger: it crosses the endpoint by evacuating prefixes with
moments, using a new general-modulus high-tail argument, and stitching three
different shift regimes.

\bibliographystyle{abbrvnat}
\bibliography{references}

\end{document}